\documentclass{article}
\usepackage{geometry}
\usepackage{graphicx} 

\usepackage[colorlinks = true,
            linkcolor = blue,
            urlcolor  = blue,
            citecolor = blue,
            anchorcolor = blue]{hyperref}

\usepackage{amsmath,amssymb,amsthm}
\usepackage{nccmath,mathtools,bbm,mathrsfs}
\newcommand\numberthis{\addtocounter{equation}{1}\tag{\theequation}}

\usepackage[ruled,linesnumbered]{algorithm2e}

\usepackage[noblocks]{authblk}

\usepackage{comment}

\usepackage[noadjust]{cite}

\usepackage[capitalise]{cleveref}
\crefname{equation}{}{}

\newtheorem{lemma}{Lemma}
\newtheorem{corollary}{Corollary}
\newtheorem{theorem}{Theorem}
\newtheorem{proposition}{Proposition}

\theoremstyle{definition}
\newtheorem{assumption}{Assumption}

\theoremstyle{remark}
\newtheorem{remark}{Remark}

\crefname{assumption}{Assumption}{Assumptions}

\newenvironment{theorembis}[1]
  {\renewcommand{\thetheorem}{\ref*{#1}$'$}%
   \addtocounter{theorem}{-1}%
   \begin{theorem}}
  {\end{theorem}}

\usepackage{enumitem}
\setlist{topsep=0pt,itemsep=0pt,parsep=3pt,labelwidth=7pt,labelsep=8pt,leftmargin=\parindent,font=\upshape}

\DeclareMathOperator*{\argmin}{arg\,min}

\DeclarePairedDelimiterX\setv[2]{\{}{\}}{#1 \;\delimsize\vert\; #2}
\DeclarePairedDelimiterX\setc[2]{\{}{\}}{#1 : #2}
\DeclarePairedDelimiterX\Pcond[2]{(}{)}{#1 \mkern2mu\delimsize\vert\mkern2mu #2}
\DeclarePairedDelimiterX\Econd[2]{[}{]}{#1 \mkern2mu\delimsize\vert\mkern2mu #2}

\newcommand{\tran}{\mathsf{T}}
\newcommand{\ud}{\mathrm{d}}
\newcommand{\uj}{\mathrm{j}}
\newcommand{\sfG}{\mathsf{G}}

\title{\bf Zeroth-Order Nonsmooth Nonconvex Optimization with Convex Liftings and Its Application to State-Feedback $\mathcal{H}_\infty$ Policy Optimization%
\thanks{This work was supported by the National Natural Science Foundation of China under Grant 72431001 and Grant 72301008.}}
\author[1]{Xuhao Wang}
\author[1]{Yujie Tang}

\affil[1]{Department of Control Science and Systems Engineering, Peking University}

\date{}

\begin{document}

\maketitle

\begin{abstract}
Direct policy optimization is widely used in reinforcement learning and control, but generally leads to nonconvex optimization problems. For state-feedback $\mathcal{H}_\infty$ control, the policy objective is also nonsmooth, despite possessing a benign landscape whose hidden convexity can be revealed by the recently developed extended convex lifting framework. Motivated by recent advances in hidden convex optimization, we study zeroth-order optimization of nonsmooth, nonconvex problems admitting a convex lifting. We propose a zeroth-order proximal point algorithm: An inexact proximal-point outer loop constructs strongly convex subproblems, while an inner loop approximately solves each subproblem using only function evaluations. With probability at least $1-\delta$, our proposed algorithm returns an $\epsilon$-optimal solution using $\widetilde{O}\left(d\epsilon^{-3}\right)$ function evaluations, while all iterates remain feasible without explicit projection. Finally, we verify that the assumptions underlying our analysis hold for discrete-time state-feedback $\mathcal{H}_\infty$ policy optimization, yielding an oracle complexity of $\widetilde{O}\left(n_u n_x\epsilon^{-3}\right)$ for attaining a prescribed objective value gap, where $n_u\times n_x$ is the dimension of the feedback gain to be optimized over.
\end{abstract}

\tableofcontents

\section{Introduction}

In the past few years, reinforcement learning has achieved impressive performance across a wide range of applications, such as competitive games~\cite{silver2017mastering}, robotics~\cite{kaufmann2023drone, haarnoja2024soccer}, and large language model training~\cite{ouyang2022training, guo2025deepseek, jin2025search}. One of the key components of the success of reinforcement learning is the policy optimization approach, which parameterizes the set of candidate policies by finite-dimensional vectors and directly optimizes the performance over the parameterized policy. Due to the close connections between the theories of Markov decision processes and optimal control, policy optimization has also garnered significant interest in the control community. In fact, policy optimization for control has a long history, and some of the early works include~\cite{levine1970determination,hyland1984optimal,apkarian2006nonsmooth,apkarian2006nonsmooth2}. Modern investigations on policy optimization for control have started to focus on optimization landscape analysis, algorithm design, and theoretical performance guarantees. The work~\cite{fazel2018global} investigated the discrete-time linear-quadratic regulator (LQR) problem and established global convergence of the corresponding policy optimization algorithm. Subsequent studies revealed further landscape properties of both discrete-time and continuous-time LQR policy optimization~\cite{bu2019lqr,malik2019derivative,mohammadi2022convergence}. Recent works have also investigated policy optimization for other control problems, including risk-sensitive control~\cite{zhang2021h2hinf,zhang2021derivativefree}, Kalman filtering~\cite{umenberger2022globally}, linear quadratic Gaussian (LQG) control~\cite{tang2023analysis, zheng2022escaping, li2025policy}, distributed control~\cite{li2021distributed}, etc.; see \cite{hu2023foundation, talebi2024policy} for a survey.  

For control tasks subject to adversarial disturbances and/or model uncertainty, control theorists have established the robust control paradigm. One of the standard formulations for robust control is $\mathcal{H}_\infty$ control~\cite{zames1981feedback}, and existing literature has established roughly three broad approaches for $\mathcal{H}_\infty$ controller synthesis. One classical approach is the Riccati-equation-based approach~\cite{zhou1996robust}, which finds a (sub)optimal controller by solving certain Riccati equations. Another classical approach is the LMI-based approach, which recasts the controller synthesis problem into a convex semidefinite program (SDP) via the bounded-real lemma and suitable change of variables~\cite{boyd1994linear,gahinet1994linear,iwasaki1994all,scherer1997multiobjective}. These two approaches require explicit knowledge of the system models, which may be hard to obtain for a complex system. A third viable approach is just direct policy optimization~\cite{apkarian2006nonsmooth,apkarian2006nonsmooth2}, which can in general be more scalable and more flexible but has underdeveloped theoretical guarantees. The recent works \cite{guo2022global,guo2023complexity,watanabe2025policy} have started to focus on the convergence rates and oracle/sample complexities of state-feedback $\mathcal{H}_\infty$ policy optimization algorithms. However, so far such rate or complexity results are only with regard to certain stationarity measures (e.g., distance from the origin to the Goldstein subdifferential) rather than the objective value gap.

One of the major difficulties in establishing convergence rates of the objective value gap for state-feedback $\mathcal{H}_\infty$ policy optimization methods lies in the fact that the $\mathcal{H}_\infty$ cost is generally nonconvex and nonsmooth in controller parameters~\cite{apkarian2006nonsmooth, burke2005robust, goldstein1977optimization}. It is generally hard to get convergence guarantees on objective value gaps for nonconvex nonsmooth problems. Fortunately, the convex SDP reformulations from the LMI-based approach of controller synthesis provide tools to reveal the \emph{hidden convexity} of policy optimization, suggesting that for many benchmark control problems, their policy optimization formulations have more well-behaved landscape properties. Indeed, \cite{bu2019lqr,mohammadi2022convergence} exploited the LMI-based convex reformulations of LQR to establish important landscape properties of LQR policy optimization, including the celebrated Polyak–Łojasiewicz (PL) property. \cite{umenberger2022globally,tang2023analysis,hu2022connectivity,tang2023global} further investigated the optimization landscapes of Kalman filtering, LQG, and output-feedback $\mathcal{H}_\infty$ control. In \cite{zheng2026benign1,zheng2026benign2}, the authors formalized the Extended Convex Lifting (\texttt{ECL}) framework, a unified framework that bridges the convex reformulations of LQR, LQG, $\mathcal{H}_\infty$ control, etc. with their policy optimization problems. The work \cite{pai2026mixed} revealed hidden convexity in mixed $\mathcal{H}_2$/$\mathcal{H}_\infty$ policy optimization based on the \texttt{ECL} framework.
For state-feedback $\mathcal{H}_\infty$ control, \cite{guo2022global,watanabe2025policy} utilized similar ideas and established coercivity, weak convexity and weak PL property of the objective function, and showed that any Clarke stationary point is globally optimal. In light of these promising results, it is then natural to ask: \textbf{By exploiting its hidden convexity, can we establish rate/complexity results for the convergence of the objective value gap for state-feedback $\mathcal{H}_\infty$ policy optimization?}

Various forms of hidden convexity have been identified in optimization \cite{ben1996hidden,wu2007peeling} and have been employed for designing algorithms with strong convergence guarantees. Recently, the works \cite{fatkhullin2025stochastic,fatkhullin2025global} proposed algorithms for solving nonconvex problems with implicit hidden convexity and provided complexity guarantees in terms of the objective value gap.
\cite{bhaskara2025descent} studied convex optimization with misaligned stochastic gradients and applied it to solving hidden convex problems. However, it is not yet clear how to apply the above works to policy optimization in control. For one thing, these works mostly focused on a simpler form of hidden convexity where the nonconvex functions can be convexified via a direct transformation on the decision variable. On the other hand, the convex reformulations of optimal or robust control problems usually involve additional Lyapunov variables. Besides, the original and transformed feasible regions in these works are both assumed to be convex and closed, while the feasible region of state-feedback $\mathcal{H}_\infty$ policy optimization (and also LQR, LQG, etc.) is open and nonconvex, making simple projection operations inadmissible. In addition, it is usually tricky to obtain gradient information for $\mathcal{H}_\infty$ policy optimization~\cite{apkarian2006nonsmooth}.

\subsection{Our Contributions}

In this article, we investigate how to exploit the hidden convexity of state-feedback $\mathcal{H}_\infty$ policy optimization to design a zeroth-order algorithm with rate and complexity guarantees on the objective value gap. We formulate the notion of \emph{convex liftings}, which incorporates extra lifting variables in the convexification procedure and is tailored for state-feedback optimal/robust control problems (see \cref{assumption:convex_lifting}). This convex lifting notion is inspired by the recently developed \texttt{ECL} framework~\cite{zheng2026benign2} for revealing hidden convexity of a wide range of optimal and robust control problems. Specifically, our contributions can be summarized as follows:
\begin{enumerate}
\item \textbf{Algorithm:} We propose the Zeroth-Order Proximal Point Algorithm for nonconvex nonsmooth problems with convex liftings, following the structure of the inexact proximal point method~\cite{rockafellar1976monotone}: The outer loop builds a quadratically penalized strongly convex subproblem, while the inner loop employs zeroth-order gradient descent to produce an inexact solution to the subproblem. The inner loop does not involve any projection step but uses a specially tailored step size rule to ensure that all intermediate iterates will stay in the feasible region.

\item \textbf{High-probability rate and complexity guarantees:}
We prove high-probability rate and complexity bounds on the \emph{global objective value gap} for our proposed algorithm. Specifically, given any sufficiently small $\epsilon$ and $\delta$, the number of zeroth-order oracle queries needed to achieve $f(x(N))-f^\ast\leq\epsilon$ with probability at least $1-\delta$ is bounded by $\widetilde{O}\left(d\epsilon^{-3}\right)$, where $f^\ast$ is the globally optimal value and $d$ is the problem dimension (see \cref{theorem:main} for more details). Here we establish high-probability-type guarantees because a successful run of our algorithm requires all intermediate iterates of the inner loop to stay within the open and nonconvex feasible region in the presence of randomness from zeroth-order gradient estimators. Apart from the $O(d)$ factor, this oracle complexity bound matches the best-known first-order complexity bound for nonsmooth hidden convex optimization.

\item \textbf{Application to state-feedback $\mathcal{H}_{\infty}$ policy optimization:} We discuss in detail how to construct a convex lifting for the state-feedback $\mathcal{H}_\infty$ policy optimization problem, leveraging results from the LMI-based convex reformulation. After verifying all technical assumptions, we apply our proposed algorithm and obtain a state-feedback $\mathcal{H}_\infty$ policy optimization algorithm with strong guarantees on the convergence of the global objective value gap.
\end{enumerate}
In \cref{subsec:theoretical_results,subsec:H-infty_policy_optimization_result}, we shall provide detailed comparisons of our results with closely related existing works on hidden convexity optimization and state-feedback $\mathcal{H}_\infty$ policy optimization.

\subsection{Related Work}
\label{sec:related-work}

This subsection provides a non-exhaustive review of some closely related work.

\vspace{-9pt}
\paragraph{Policy optimization for control.}
As mentioned before, policy optimization for control has a long history \cite{levine1970determination,hyland1984optimal}. The survey~\cite{hu2023foundation} focused on theoretical foundations of policy optimization for continuous control. The survey \cite{talebi2024policy} explored a geometric perspective and its implications for algorithmic analysis.  

For LQR, \cite{fazel2018global} proved global convergence of policy-gradient methods by establishing the PL property and regularity on sublevel sets.  \cite{bu2019lqr} further studied first-order methods for discrete-time LQR and quantified the behavior of policy gradient, natural policy gradient, and quasi-Newton iterations. \cite{malik2019derivative} gave finite-sample guarantees for a stochastic derivative-free policy gradient method.  \cite{mohammadi2022convergence} analyzed convergence and sample complexity for model-free policy gradient methods for continuous-time LQR. For other types of control problems, \cite{sun2021convex} examined the relationship between classical convex parameterizations and policy optimization, and proposed a unified framework for explaining the PL property for a broad class of control problems. \cite{li2021distributed} proposed a zeroth-order distributed policy optimization method for decentralized LQR and established convergence and sample-complexity probabilistic guarantees. \cite{tang2023analysis} analyzed connectivity of the set of stabilizing dynamic controllers and structures of stationary points of the LQG landscape, and showed that stationary points corresponding to controllable and observable controllers are globally optimal.  \cite{zheng2022escaping} designed a perturbation-based method to escape the high-order saddles that can occur in LQG policy optimization.

For $\mathcal{H}_\infty$ control,  \cite{apkarian2006nonsmooth,apkarian2006nonsmooth2} studied direct policy optimization for $\mathcal{H}_\infty$, but their works lack solid optimality guarantees. \cite{guo2022global} established global convergence of Goldstein's subgradient method and its variants for state-feedback $\mathcal{H}_\infty$ policy optimization. \cite{guo2023complexity} proposed a derivative-free policy optimization framework for structured $\mathcal{H}_\infty$ synthesis, and analyzed its high-probability complexity bounds for finding a $(\delta,\epsilon)$-Goldstein stationary point.  \cite{watanabe2025policy} investigated the vanilla subgradient method for state-feedback $\mathcal{H}_\infty$ policy optimization and proved an $O(\epsilon^{-4})$ complexity bound for finding an $\epsilon$-inexact stationary point. 

\vspace{-9pt}
\paragraph{Hidden convexity for nonconvex problems.}
Broadly speaking, when a nonconvex optimization problem can be equivalently reformulated into a convex one, we say that it admits hidden convexity. \cite{ben1996hidden} exhibited hidden convexity for a special type of nonconvex quadratically constrained quadratic program. More recently, \cite{chen2025efficient} proposed the mirror stochastic gradient method that operates only in the original decision variable space for certain hidden convex problems; they established $\widetilde{O}(\epsilon^{-2})$ complexities and discussed its application in a network revenue management problem. \cite{fatkhullin2025stochastic} studied stochastic optimization when an implicit invertible map converts the problem to a convex form, proving global convergence and complexity bounds for projected stochastic (sub)gradient methods.  \cite{fatkhullin2025global} proposed an inexact proximal-point framework for functionally constrained hidden convex problems, obtaining global last-iterate guarantees for both smooth and nonsmooth setups. \cite{bhaskara2025descent} investigated gradient descent with misaligned stochastic gradients, and then applied it to hidden convex problems to obtain an $O(\epsilon^{-3})$ complexity bound. We remark that the works~\cite{chen2025efficient,fatkhullin2025stochastic,fatkhullin2025global,bhaskara2025descent} mostly focused on the setting where the nonconvex functions are convexified via a direct transformation on the decision variable, while for policy optimization in control, the hidden convexity appears more complicated as extra Lyapunov variables are usually involved in the convexification procedure.

Regarding policy optimization for control, existing works have already extensively exploited hidden convexity for landscape analysis on a case-by-case basis~\cite{bu2019lqr,bu2021topological,mohammadi2022convergence,umenberger2022globally,guo2022global,hu2022connectivity}. \cite{zheng2026benign2} formalized \texttt{ECL} as a bridge connecting nonconvex policy optimization problems with the convex reformulations, covering a wide range of classical control problems including LQR, LQG, and state- and output-feedback $\mathcal{H}_\infty$ control. \cite{zheng2026benign1} used \texttt{ECL} to show global optimality of nondegenerate Clarke stationary points for LQG and output-feedback $\mathcal{H}_\infty$ policy optimization.  \cite{watanabe2025policy} proved that the discrete-time $\mathcal{H}_\infty$ cost is weakly convex on convex subsets of its sublevel sets and established a weak Polyak--\L{}ojasiewicz inequality for state feedback, based on a similar idea to \texttt{ECL}.



\vspace{-9pt}
\paragraph{Zeroth-order optimization.}
Zeroth-order optimization methods only query function values but not derivatives during the optimization procedure. One popular class of zeroth-order methods is to construct gradient estimators using function values at finitely many points. One of the earliest such works is \cite{nemirovsky1983problem}, which proposed a single-point gradient estimator. \cite{flaxman2005online} applied the single-point gradient estimator to online bandit convex optimization and derived the corresponding regret bound. \cite{duchi2015optimal} focused on theoretical properties of two-point gradient estimators, and also derived information-theoretic lower bounds on the minimax convergence rate in stochastic convex optimization. \cite{nesterov2017random} analyzed two-point estimators with Gaussian perturbations and obtained rate and complexity results for various problem setups. \cite{shamir2017optimal} gave an optimal algorithm for bandit and zeroth-order nonsmooth convex optimization, using a central-difference version of two-point gradient estimators. For nonsmooth nonconvex problems, \cite{kornowski2024optimal} designed a zeroth-order algorithm that achieves the optimal complexity of producing a $(\delta,\epsilon)$-Goldstein stationary point. In this work, we employ the central-difference version of two-point gradient estimators in the inner-loop iterations.


\vspace{-9pt}
\paragraph{High-probability guarantees.}
Randomness of convergence behavior naturally arises for randomized algorithms or when stochastic errors occur because of random noise, small sample sizes, etc. It is popular to establish rate and complexity results with respect to the expectation of optimality/stationarity measures. On the other hand, high-probability rate and complexity guarantees are usually considered to be stronger and have been a topic of significant interest. 
In~\cite{rakhlin2012making}, the authors studied stochastic (sub)gradient descent for stochastic strongly convex optimization and derived a near-optimal high-probability convergence rate. \cite{hazankale2014beyond} introduced an epoch-based variant of stochastic subgradient descent and showed that it achieves a better high-probability convergence rate for stochastic strongly convex optimization. \cite{harvey2019tight} established a generalized version of Freedman's inequality in martingale concentration, and gave tight high-probability analysis for suffix-averaging and last iterate of stochastic subgradient descent for nonsmooth strongly convex problems. Then \cite{harvey2019simple} analyzed a non-uniform averaging scheme of stochastic subgradient descent for nonsmooth strongly convex problems, showing that it also achieves the optimal rate, and discussed an extension to sub-Gaussian noise. 
%
%
In this work, we employ the generalized Freedman inequality in~\cite{harvey2019tight} for convergence analysis of the inner loop. On the other hand, we also provide a high-probability guarantee for the intermediate iterates to stay within the feasible region, which we believe brings some new techniques and perspectives.


\paragraph{Notation.} For each positive integer $n$, the set $\{1,2,\ldots,n\}$ will be denoted by $[n]$. We let $\langle x,y\rangle=x^\tran y$ denote the standard inner product on Euclidean spaces, and let $\|x\|=\sqrt{\langle x,x\rangle}$ denote the standard $\ell_2$ norm. For a matrix $A\in\mathbb{R}^{m\times n}$, we let $\|A\|_F$ denote its Frobenius norm. The set of positive reals is denoted by $\mathbb{R}_{++}$. The sets of $m\times m$ real symmetric matrices, real positive semidefinite matrices, and real positive definite matrices are denoted by $\mathbbm{S}^m$, $\mathbbm{S}_+^m$, and $\mathbbm{S}_{++}^m$, respectively. When $B-A$ is positive semidefinite, we denote $A\preceq B$ or $B\succeq A$; when $B-A$ is positive definite, we denote $A\prec B$ or $B\succ A$. 
For a function $f:\mathcal{D}\rightarrow\mathbb{R}$, its epigraph is denoted by $\operatorname{epi}(f)\coloneqq\setc*{(x,\gamma)}{x\in\mathcal{D},\gamma\geq f(x)}$.
$\mathbb{S}_{d-1}\coloneqq \setc*{x\in\mathbb{R}^d}{\|x\|=1}$ denotes the unit sphere, and $\mathbb{B}_d\coloneqq \setc*{x\in\mathbb{R}^d}{\|x\|\leq 1}$ denotes the closed unit ball. The uniform distributions on $\mathbb{S}_{d-1}$ and $\mathbb{B}_d$ will be denoted by $\mathcal{U}(\mathbb{S}_{d-1})$ and $\mathcal{U}(\mathbb{B}_d)$, respectively.  
The Minkowski sum of two subsets $A,B$ of a Euclidean space will be denoted by $A+B=\setc*{x+y}{x\in A,y\in B}$, and we shall just denote $x+B=\{x\}+B$. 
We use $O(\cdot)$ and $\Theta(\cdot)$ as the big-O and big-Theta notations, and use $\widetilde{O}(\cdot)$ to hide polylogarithmic terms. 

\section{Problem Formulation and Preliminaries}

\subsection{A Motivating Example: State-Feedback $\mathcal{H}_\infty$ Policy Optimization}

We first introduce the state-feedback $\mathcal{H}_\infty$ policy optimization problem as a major motivating example. Consider a discrete-time LTI system
\begin{equation}\label{eq:Hinf_plant}
\begin{aligned}
x(t+1) &= Ax(t)+Bu(t)+ B_w w(t), 
\end{aligned}
\end{equation}
where $x(t)\in\mathbb{R}^{n_x}$ is the state variable, $u(t)\in\mathbb{R}^{n_u}$ is the control input, and $w(t)\in\mathbb{R}^{n_w}$ is the disturbance. The initial state $x(0)$ is set to be zero.

In state-feedback $\mathcal{H}_\infty$ control, we regard $w(t)$ as an adversarial disturbance with bounded energy, and seek to minimize a quadratic performance in the worst case. Specifically, we define the $\ell_2$ norm of a disturbance signal $w=(w(t))_{t=0}^\infty$ by $\|w\|_{\ell_2} = \left(\sum_{t=0}^\infty w(t)^\tran w(t)\right)^{\!1/2}$,
and say that $w$ has bounded energy whenever $\|w\|_{\ell_2}<+\infty$. We employ the quadratic performance metric
$\sum_{t=0}^\infty \left(
x(t)^\tran Q x(t)+u(t)^\tran R u(t)\right)$
to evaluate the system's performance of a given trajectory $(x(t),u(t))$. The state-feedback $\mathcal{H}_\infty$ control problem is then formulated as
\begin{equation}
\label{eq:Hinf_control_formulation}
\begin{aligned}
\min_{(u(t))_{t=0}^\infty}\  & \sup_{w:\|w\|_{\ell_2}\leq 1}
\sum_{t=0}^\infty\left(
x(t)^\tran Q x(t)+u(t)^\tran R u(t)\right), \\
\text{s.t.}\ & \cref{eq:Hinf_plant}\text{ and }x(0)=0,
\end{aligned}
\end{equation}
where each $u(t)$ should only depend on the historical state observations $x(1),\ldots,x(t)$. We impose the assumptions that $(A, B)$ is controllable, and that $Q$, $R$ and $W\coloneqq B_wB_w^\tran$ are positive definite matrices.

Classical robust control theory has established that there exists a static linear feedback controller $u(t)=K_\star x(t)$ that solves~\cref{eq:Hinf_control_formulation}, where $K_\star$ is a constant gain matrix. Therefore, without loss of generality, we can narrow our focus to static linear feedback controllers of the form
\begin{equation}\label{eq:static_linear}
u(t)=Kx(t)
\end{equation}
where $K\in\mathbb{R}^{n_u\times n_x}$ is a gain matrix to be determined. We shall just use $K$ to parameterize the controller~\cref{eq:static_linear} and refer to $K$ as a \emph{policy}. The closed-loop system with policy $K$ then becomes
\begin{equation}
\label{eq:Hinf_closed-loop}
x(t+1) = (A+BK)x(t)+B_w w(t),
\end{equation}
and (the square root of) the worst-case quadratic performance now becomes a function of $K$:
\[
\begin{aligned}
J_\infty(K)\coloneqq{} &  \left[\sup_{w:\|w\|_{\ell_2}\leq 1}
\sum_{t=0}^\infty x(t)^\tran (Q+K^\tran RK) x(t)\right]^{\!1/2} \\
\text{s.t.}\ & \cref{eq:Hinf_closed-loop} \text{ and }x(0)=0.
\end{aligned}
\]
It is straightforward to check that $J_\infty(K)\in\mathbb{R}_{++}\cup\{+\infty\}$ for all $K\in\mathbb{R}^{n_u\times n_x}$. Classical control theory has established that, under our assumptions on the system and performance matrices, $J_\infty(K)<+\infty$ if and only if $A+BK$ is Schur stable, i.e., all eigenvalues of $A+BK$ have magnitudes strictly less than~$1$. We therefore define the set
\[
\mathcal{K}\coloneqq
\setc*{K}{A+BK\text{ is Schur stable}},
\]
which naturally serves as the domain of the function $J_\infty$. We can then reformulate the original $\mathcal{H}_\infty$ control problem~\cref{eq:Hinf_control_formulation} as the following policy optimization problem:
\begin{equation} \label{eq:Hinf_policy_optimization}
\begin{aligned}
\inf_{K}\ \ & J_{\infty}(K)\\
\text{s.t.}\ \ &
K\in\mathcal{K}.
\end{aligned}
\end{equation}

It is known that~\cref{eq:Hinf_policy_optimization} is generally a nonconvex and nonsmooth optimization problem. Existing literature (see, e.g., \cite{guo2022global,guo2023complexity,watanabe2025policy}) has established some fundamental analytical and geometrical properties of $J_\infty$ and $\mathcal{K}$, which are summarized in the following proposition.

\begin{proposition}
Under the assumptions that $(A,B)$ is controllable and that $Q,R$ and $B_wB_w^\tran$ are positive definite, we have
\begin{enumerate}
\item $\mathcal{K}$ is an open subset of $\mathbb{R}^{n_u\times n_x}$ that can be nonconvex and unbounded.
\item $J_\infty$ is continuous but possibly not everywhere differentiable on $\mathcal{K}$. 
\item $J_\infty$ is coercive, i.e., $J_\infty(K)\rightarrow+\infty$ whenever $\|K\|_F\rightarrow+\infty$ or $\operatorname{dist}(K,\partial\mathcal{K})\rightarrow 0$. Here $\operatorname{dist}(K,\partial\mathcal{K})$ denotes the minimum distance from $K$ to any point on $\partial\mathcal{K}$, the boundary of $\mathcal{K}$.
\end{enumerate}
\end{proposition}

Because of the non-convexity and non-smoothness of $J_\infty$, even though policy optimization methods for $\mathcal{H}_\infty$ robust control were proposed in the 2000s (see, e.g., \cite{apkarian2006nonsmooth,apkarian2006nonsmooth2}), with software packages having been developed~\cite{gumussoy2009multiobjective,apkarian2014multi}, their theoretical guarantees had long been lacking.

With the recent surge in the study of policy optimization for control, several works \cite{guo2022global,watanabe2025policy,zheng2026benign2} have revealed hidden convexity underlying the nonconvex function $J_\infty$, which is closely related to the classical LMI-based convex reformulation of the state-feedback $\mathcal{H}_\infty$ control problem~\cite{gahinet1994linear,boyd1994linear,dullerud2000course}. Specifically, after lifting the epigraph of $J_\infty$ to a higher dimension by appending certain Lyapunov variable, one can construct a bijection that transforms the lifted set to a convex set. 
An important consequence of this ``convex lifting'' for $J_\infty$ is that all Clarke stationary points of $J_\infty$ are globally optimal, suggesting that first-order or zeroth-order optimization methods should be able to find globally optimal points despite the non-convexity of $J_\infty$. As reviewed in the introduction, the works \cite{guo2022global,guo2023complexity,watanabe2025policy} have investigated first-order and zeroth-order methods applied to the minimization of $J_\infty$ and have obtained promising results. However, these results all have certain limitations (see \cref{subsec:H-infty_policy_optimization_result} for detailed comparison). In particular, explicit convergence rate and oracle complexity results in terms of the objective value gap $J_\infty(K)-\inf_{K'}J_\infty(K')$ have remained unknown.

In this article, we consider nonsmooth, nonconvex optimization problems which possess certain ``convex lifting'' structures similar to the state-feedback $\mathcal{H}_\infty$ policy optimization problem. We shall develop an algorithm that has convergence rate and oracle complexity guarantees in terms of the objective value gap $f(x)-\inf_{x'} f(x')$.

\subsection{Problem Formulation}

Let $\mathcal{D}$ be an open and possibly nonconvex subset of $\mathbb{R}^d$, where we shall assume throughout the article that $d\geq 3$. Let $f:\mathcal{D}\rightarrow\mathbb{R}$ be a continuous function that we wish to minimize. Without loss of generality, we assume $\inf_{x\in\mathcal{D}}f(x)>0$, and let $f^\ast=\inf_{x\in\mathcal{D}}f(x)$ denote the optimal value of $f$ over $\mathcal{D}$. The function $f$ may be nonconvex or nonsmooth, but we impose the critical assumption of this article that the function $f$ admits a \emph{convex lifting}:
\begin{assumption}
\label{assumption:convex_lifting}
The function $f$ admits a \emph{convex lifting} $(\mathfrak{L}_{\mathrm{lft}},\mathfrak{F}_{\mathrm{cvx}},\Phi)$ in the following sense: $\mathfrak{L}_{\mathrm{lft}}$ is a subset of $\mathbb{R}^{d}\times\mathbb{R}\times\mathbb{R}^{d_\xi}$, $\mathfrak{F}_{\mathrm{cvx}}$ is a convex subset of $\mathbb{R}\times\mathbb{R}^{d+d_\xi}$, and $\Phi:\mathfrak{L}_{\mathrm{lft}}\rightarrow\mathfrak{F}_{\mathrm{cvx}}$ is a bijection. Furthermore, they satisfy the following conditions:
\begin{enumerate}
\item For every $x\in\mathbb{R}^d$ and $\lambda\in\mathbb{R}$, we have $(x,\lambda)\in\operatorname{epi}(f)$ if and only if there exists some $\xi\in\mathbb{R}^{d_\xi}$ such that $(x,\lambda,\xi)\in\mathfrak{L}_{\mathrm{lft}}$.
\item For every $(x,\lambda,\xi)\in\mathfrak{L}_{\mathrm{lft}}$, we have $\Phi(x,\lambda,\xi)=(\lambda,\zeta)$ for some $\zeta\in\mathbb{R}^{d+d_\xi}$.
\item Both $\Phi$ and $\Phi^{-1}$ are locally Lipschitz continuous.
\end{enumerate}
\end{assumption}

\Cref{assumption:convex_lifting} is adapted from the recent works \cite{zheng2026benign1,zheng2026benign2} on \emph{extended convex lifting} (\texttt{ECL}), a unified framework that allows convex analysis for nonconvex policy optimization formulations of benchmark control problems, including LQR, LQG, state-feedback $\mathcal{H}_\infty$ control, output-feedback $\mathcal{H}_\infty$ control, etc. Compared to the full version of \texttt{ECL} in \cite[Definition 1]{zheng2026benign2}, we have made the following simplifications: i) There is no auxiliary set that typically accounts for similarity transformation on dynamic feedback controllers; ii) all points in the graph $(x,f(x))$ have their associated lifting variable $\xi$, i.e., all points in the domain are \emph{nondegenerate} in the sense of \cite[Definition 2]{zheng2026benign2}. As shown in \cref{subsec:H-infty_policy_optimization_result} and in \cite[Section IV.B]{zheng2026benign2}, the simplified version presented in \cref{assumption:convex_lifting} still covers the state-feedback $\mathcal{H}_\infty$ policy optimization problem. We shall frequently refer to $\mathfrak{L}_{\mathrm{lft}}$ as the lifted set, as it lifts the epigraph of $f$ to a set of a higher dimension. The extra variable $\xi\in\mathbb{R}^{d_\xi}$ introduced in this lifting often corresponds to Lyapunov variables in the convex reformulation of state-feedback policy optimization problems. This extra variable $\xi$, together with the lifted set $\mathfrak{L}_{\mathrm{lft}}$, also differentiates our setup from~\cite{fatkhullin2025stochastic,chen2025efficient,fatkhullin2025global,bhaskara2025descent}, which mainly focus on a simpler situation where the objective function is convexified by a direct transformation on the decision variable without liftings.

Given the convex lifting $(\mathfrak{L}_{\mathrm{lft}},\mathfrak{F}_{\mathrm{cvx}},\Phi)$, we let
\[
\pi_{x,\lambda}:(x,\lambda,\xi)\mapsto (x,\lambda),
\qquad
\pi_{x}:(x,\lambda,\xi)\mapsto x,
\qquad
\pi_\lambda:(x,\lambda,\xi)\mapsto\lambda
\]
denote the respective canonical projection operators on $\mathbb{R}^d\times\mathbb{R}\times\mathbb{R}^{d_\xi}$. The first condition in \cref{assumption:convex_lifting} can then be equivalently written as
\[
\pi_{x,\lambda}(\mathfrak{L}_{\mathrm{lft}}) = \operatorname{epi}(f).
\]

Our goal is to solve the following nonconvex nonsmooth optimization problem:
\[
\min_{x\in\mathcal{D}} f(x).
\]
We shall further impose the restriction that only \emph{zeroth-order} information on the function $f$ is available, i.e., we can only query the function values of $f$ at finitely many points through a given zeroth-order oracle. For clarity and simplicity, we assume that the zeroth-order oracle is exact: No error or noise corrupts the function value returned by the oracle. The noisy situation is certainly important but is outside the scope of this article, and we leave it for future work.

With the zeroth-order restriction, our algorithm and theory can provide a basis for model-free extensions, where a good mathematical model of the system is lacking. For the state-feedback $\mathcal{H}_\infty$ policy optimization problem, \cite{guo2023complexity} has proposed a sample-based method for $\mathcal{H}_\infty$ cost estimation with guaranteed sample complexity. We also remark that, even when a good mathematical model is available, obtaining first-order information (such as Clarke subdifferential or Goldstein subdifferential) of a nonsmooth nonconvex function like $J_\infty$ can be difficult or tricky; see, e.g., the results on the subgradients of the $\mathcal{H}_\infty$ cost in~\cite{apkarian2006nonsmooth,apkarian2006nonsmooth2,zheng2026benign1}.

\subsection{Technical Assumptions}

We now state further technical assumptions that will be employed for convergence analysis.

\begin{assumption}
\label{assumption:technical_objective}
Given $f$ and its convex lifting $(\mathfrak{L}_{\mathrm{lft}},\mathfrak{F}_{\mathrm{cvx}},\Phi)$, for each $\Delta>0$, the following conditions hold:
\begin{enumerate}
\item The set
\[
\setc*{(x,\lambda,\xi)\in\mathfrak{L}_{\mathrm{lft}}}{\lambda\leq\Delta}
=\pi_{\lambda}^{-1}((-\infty,\Delta])\cap \mathfrak{L}_{\mathrm{lft}}
\]
is compact. Consequently, the sublevel set
\[
\mathcal{S}(\Delta)\coloneqq
\setc*{x\in\mathcal{D}}{f(x)\leq \Delta}
=\pi_{x}\!\left(\setc*{(x,\lambda,\xi)\in\mathfrak{L}_{\mathrm{lft}}}{\lambda\leq\Delta}\right)
\]
is compact.

\item There exists some $G_\Delta>0$ such that $f$ is $G_\Delta$-Lipschitz continuous on $\mathcal{S}(\Delta)$, i.e., $
|f(x)-f(y)|\leq G_\Delta\|x-y\|
$
for all $x,y\in\mathcal{S}(\Delta)$;
\item There exists some $\mu_\Delta>0$ such that $f$ is $\mu_\Delta$-weakly convex on any convex subset of $\mathcal{S}(\Delta)$. In other words, for any convex $C\subseteq\mathcal{S}(\Delta)$, the function $x\mapsto f(x)+\frac{\mu_\Delta}{2}\|x\|^2$ is convex on $C$.
\end{enumerate}
\end{assumption}

The second and third conditions in \cref{assumption:technical_objective} are common in nonsmooth nonconvex optimization~\cite{davis2019stochastic} except for the following subtlety: Since $J_\infty$ is coercive and its domain $\mathcal{D}$ is nonconvex and open, we can only impose ``semi-global'' Lipschitz continuity and weak convexity of the objective function. The first condition imposes certain boundedness that will be needed for convergence analysis.
In \cref{subsec:H-infty_policy_optimization_result}, we shall show that the objective function $J_\infty$ of state-feedback $\mathcal{H}_\infty$ policy optimization satisfies \cref{assumption:convex_lifting,assumption:technical_objective}.

\subsection{Preliminaries on Zeroth-Order Gradient Estimation}
\label{subsec:prelim_zeroth-order}

This part presents a brief introduction to the zeroth-order gradient estimation technique employed in our proposed algorithm.

For each $x\in\mathcal{D}$, $z\in\mathbb{S}_{d-1}$ and $\nu>0$ such that $x\pm\nu z\in\mathcal{D}$, we define
\begin{equation}
\label{eq:grad_est_definition}
\sfG(x;z,\nu)\coloneqq d\cdot\frac{f(x+\nu z)-f(x-\nu z)}{2\nu}\cdot z.
\end{equation}
One may perceive $\nu z$ as a small perturbation on $x$, in which $z$ represents the perturbation direction,  and $\nu$ controls the magnitude of the perturbation. We shall refer to $\nu$ as the \emph{smoothing radius}.

Since the domain $\mathcal{D}$ can be a proper subset of $\mathbb{R}^d$, given a fixed value of the smoothing radius~$\nu$, $\sfG(x;z,\nu)$ may not be defined if $x$ is too close to the boundary of $\mathcal{D}$. We therefore denote
\[
\mathcal{D}_\nu\coloneqq
\setc*{x\in\mathcal{D}}{x+\nu\mathbb{B}_d\subseteq\mathcal{D}}.
\]
It is straightforward to see that $\sfG(x;z,\nu)$ is well-defined over $x\in\mathcal{D}_\nu$ for any $z\in\mathbb{S}_{d-1}$. Since $\mathcal{D}$ is open, $\mathcal{D}_\nu$ will be open and nonempty for all sufficiently small $\nu>0$.

We have the following lemma that characterizes the expectation of $\sfG(x;z,\nu)$ when $z$ is uniformly sampled from the unit sphere.
\begin{lemma}
\label{lemma:expectation_grad_est}
Given $\nu>0$ such that $\mathcal{D}_\nu$ is nonempty, define $f_\nu:\mathcal{D}_\nu\rightarrow\mathbb{R}$ by
\begin{equation}
\label{eq:smoothed_version_func}
f_\nu(x) = \mathbb{E}_{y\sim\mathcal{U}(\mathbb{B}_d)}\!\left[f(x+\nu y)\right],\qquad \forall x\in\mathcal{D}_\nu.
\end{equation}
Then $f_\nu$ is differentiable over $\mathcal{D}_\nu$, and we have
\[
\mathbb{E}_{z\sim\mathcal{U}(\mathbb{S}_{d-1})}
\!\left[\sfG(x;z,\nu)\right]
=\nabla f_\nu(x),
\qquad\forall x\in\mathcal{D}_\nu.
\]
\end{lemma}

\Cref{lemma:expectation_grad_est} justifies that $\sfG(x;z,\nu)$ with $z\sim\mathcal{U}(\mathbb{S}_{d-1})$ can serve as a stochastic gradient of the function $f_\nu$ defined by~\cref{eq:smoothed_version_func}. By plugging $\sfG(x;z,\nu)$ into, say, the stochastic gradient descent method, one will obtain a zeroth-order optimization algorithm that tries to minimize $f_\nu$. The function $f_\nu$ is sometimes referred to as the ``smoothed version'' of the function~$f$. The result of \cref{lemma:expectation_grad_est} dates back to \cite{nemirovsky1983problem} (see Section 9.3.2 therein); we refer to \cite{shamir2017optimal} for a more recent study on this zeroth-order gradient estimator for nonsmooth convex problems.

The following lemma lists some standard properties of $f_\nu$; we provide a proof of these properties in \cref{appendix:proof_properties_smoothed_func} for completeness. We shall see that the smoothed version of $f$ serves as an important bridge in establishing the convergence of our algorithm and in guaranteeing high-probability descent of the function value by zeroth-order gradient descent.

\begin{lemma}
\label{lemma:properties_smoothed_func}
Given $\Delta>0$ and $\nu>0$, suppose $S$ is a subset of $\mathcal{S}(\Delta)$ such that $S+\nu\mathbb{B}_d\subseteq\mathcal{S}(\Delta)$. Then the following properties hold:
\begin{enumerate}
\item $|f_\nu(x)-f(x)|\leq G_\Delta\nu$ for all $x\in S$.
\item $|f_\nu(x)-f_\nu(y)|\leq G_\Delta\|x-y\|$ for all $x,y\in S$.
\item $f_\nu$ is $G_\Delta d/\nu$-smooth over $S$.
Consequently, for $x,y\in S$ such that $\lambda x+(1-\lambda)y\in S$ for all $\lambda\in[0,1]$, we have
\[
f_\nu(y)\leq f_\nu(x)+\langle \nabla f_\nu(x),y-x\rangle
+\frac{G_\Delta d}{2\nu}\|y-x\|^2.
\]
\end{enumerate}
\end{lemma}

The following lemma shows that the random variables $\|\sfG(x;z,\nu)\|$ and $\|\sfG(x;z,\nu)-\nabla f_\nu(x)\|$ are sub-Gaussian, which play a fundamental role in establishing high-probability convergence guarantees of the proposed algorithm. The proof is given in \cref{appendix:proof_lemma_sub-Gaussian_grad_est}.
\begin{lemma}
\label{lemma:sub-Gaussian_grad_est}
Let $\Delta>0$ be arbitrary. Suppose $\nu>0$ and $x\in\mathcal{S}(\Delta)$ satisfy $x+\nu\mathbb{B}_d\subseteq\mathcal{S}(\Delta)$. Then, for every $\epsilon>0$, we have
\begin{equation}
\label{eq:high_prob_norm_grad_est}
\mathbb{P}_z(\|\sfG(x;z,\nu)\|\geq \epsilon)\leq
\exp\!\left(
-\frac{\epsilon^2}{2G_\Delta^2 d}
\right),
\end{equation}
where $z\sim\mathcal{U}(\mathbb{S}_{d-1})$. Consequently,
\begin{equation}
\label{eq:sub-gaussiaon_grad_est_norm}
\mathbb{E}_z\!\left[\exp\!\left(\frac{\|\sfG(x;z,\nu)\|^2}{4G_\Delta^2 d}\right)\right]\leq 2,
\end{equation}
and
\begin{equation}
\label{eq:sub-Gaussian_norm_centered_grad}
\mathbb{E}_z\!\left[
\exp\!\left(\frac{\left\|\sfG(x;z,\nu)-\nabla f_\nu(x)\right\|^2}{8G_\Delta^2 d}\right)
\right]\leq 2.
\end{equation}
\end{lemma}

\section{Algorithm and Main Results}

\subsection{Algorithm Design}

\begin{algorithm}[th]
\linespread{1.15}\selectfont
\SetAlgoLined
\DontPrintSemicolon
\caption{Zeroth-Order Proximal Point Algorithm}\label{algorithm:main}
Initialize: $x(0)\in\mathcal{D}$, $m>0$, step size offsets $(t_0(n))_{n=1}^N$, smoothing radii $\left(\nu(n)\right)_{n=1}^N$, epoch lengths $\left(T(n)\right)_{n=1}^N$\;
\For{$n\in[N]$}{
$x(n)\leftarrow\mathtt{ZOGD}\!\left(x(n-1),m,\nu(n),T(n),t_0(n)\right)$.\;
}
\KwRet $x(N)$.\;
\vspace{6pt}
\SetKwProg{mysubroutine}{Subroutine}{}{}
  \mysubroutine{$\mathtt{ZOGD}(x_0,m,\nu,T,t_0)$}{
  Initialize step sizes by $\eta_t=\mfrac{3}{m(t+t_0)},\,\forall t\in[T]$. \;
  \For{$t\in[T]$}{
  Sample $z_{t}\sim\mathcal{U}(\mathbb{S}_{d-1})$ independently of $z_1,\ldots,z_{t-1}$.\;
  $\displaystyle g_t=\sfG(x_{t-1};z_{t},\nu)$.\;
  $x_t=x_{t-1}-\eta_t(g_t+2m(x_{t-1}-x_0))$.\;
  }
  \KwRet $\displaystyle\overline{x}_T = \mfrac{2}{T(T+2t_0-1)}\sum\nolimits_{t=1}^T (t+t_0-1) x_{t-1}$.\;
  }
\end{algorithm}

Our proposed algorithm, the \emph{Zeroth-Order Proximal Point Algorithm}, is detailed in \cref{algorithm:main}. Basically, \Cref{algorithm:main} can be viewed as performing the following type of proximal point iterations:
\begin{equation}
\label{eq:approxiate_proximal_point}
x(n)\approx \argmin_{x\in\mathcal{D}}
\!\left\{f(x)+m\|x-x(n-1)\|^2\right\}.
\end{equation}
Here $m>0$ is an algorithmic parameter; by the weak convexity of $f$ in \cref{assumption:technical_objective}, one can see that when $m$ is sufficiently large, $x\mapsto f(x)+m\|x-x(n-1)\|^2$ will be strongly convex (though still nonsmooth) over any convex subset of $\mathcal{D}$, making the subproblem in~\cref{eq:approxiate_proximal_point} relatively well-behaved. The adoption of the proximal point framework in the algorithm design was inspired by the recent work~\cite{fatkhullin2025global} on nonconvex optimization with hidden convexity. While \cite{fatkhullin2025global} mainly focuses on a simpler type of hidden convexity where the nonconvex functions are convexified by a direct transformation on the decision variable, our theoretical analysis demonstrates that the proximal point framework can also handle the convex lifting setup formulated in this article.

Then, to find the corresponding approximate minimizer in each proximal point iteration~\cref{eq:approxiate_proximal_point}, we design the subroutine \texttt{ZOGD}, a zeroth-order gradient descent procedure. Given a base point $x_0\in\mathbb{R}^d$ (which will be $x(n-1)$ for the $n$th outer-loop iteration), the \texttt{ZOGD} subroutine carries out the following inner-loop iterations:
\begin{equation}
\label{eq:zeroth-order_subgradient_descent}
x_t=x_{t-1} - \eta_t(\sfG(x_{t-1};z_t,\nu)+2m(x_{t-1}-x_0)),\qquad\forall t\in[T].
\end{equation}
Here each $z_t$ is sampled from the uniform distribution on the unit sphere $\mathbb{S}_{d-1}$, and $\sfG(x_{t-1};z_t,\nu)$ is the zeroth-order gradient estimator introduced in \cref{subsec:prelim_zeroth-order}. By \cref{lemma:expectation_grad_est}, we see that $\sfG(x_{t-1};z_t,\nu)$ provides a stochastic gradient of the smoothed version $f_\nu$, whose construction only requires zeroth-order information. The term $2m(x_{t-1}-x_0)$ in~\cref{eq:zeroth-order_subgradient_descent} corresponds to the gradient of the quadratic term $m\|x-x_0\|^2$ evaluated at $x_{t-1}$. Adding $\sfG(x_{t-1};z_t,\nu)$ and $2m(x_{t-1}-x_0)$ together, we get a stochastic gradient of the function $f_\nu(x)+m\|x-x_0\|^2$, i.e.,
\[
\mathbb{E}\Econd*{
\sfG(x_{t-1};z_t,\nu)+2m(x_{t-1}-x_0)}{x_{t-1}} = \left.\nabla_x\!\left(f_\nu(x)+m\|x-x_0\|^2\right)\right|_{x=x_{t-1}}.
\]
Thus, by properly choosing the step sizes $\eta_t$ and running the iterations~\cref{eq:zeroth-order_subgradient_descent} for a sufficiently large number of steps, one would expect that a sufficiently accurate minimizer of $f_\nu(x)+m\|x-x_0\|^2$ could be found. Then, by properly controlling the smoothing radius $\nu$, Part 1 of \cref{lemma:properties_smoothed_func} guarantees that an approximate minimizer of $f_\nu(x)+m\|x-x_0\|^2$ will also be an approximate minimizer of $f(x)+m\|x-x_0\|^2$. Consequently, we would expect that
\[
\texttt{ZOGD}(x_0,m,\nu,T,t_0)
\approx\argmin_{x\in\mathcal{D}}\!\left\{
f(x)+m\|x-x_0\|^2
\right\},
\]
assuming all the algorithmic parameters are chosen properly.

Note that we employ a specific form of step sizes in the subroutine \texttt{ZOGD}:
\[
\eta_t=\frac{3}{m(t+t_0)},\qquad\forall t\in[T],
\]
and use the weighted average
\[
\overline{x}_T=\frac{2}{T(T+2t_0-1)}\sum_{t=1}^T (t+t_0-1)x_{t-1}
\]
as the subroutine's output, where $t_0$ depends on the outer-loop iteration number $n$. There are two major technical reasons underlying this choice:
\begin{enumerate}
\item We note that the function $f(x)+m\|x-x_0\|^2$ is a \emph{strongly convex} and \emph{nonsmooth} function. For $\mu$-strongly convex and nonsmooth problems, it is standard to set $\eta_t=\alpha/t$ with $\alpha\geq 2/\mu$ and employ certain weighted averaging of the intermediate iterates to achieve the optimal $O(1/T)$ convergence rate for the subgradient descent method; see, e.g., \cite{lacoste2012simpler,harvey2019simple,harvey2019tight}.

We also remark that, while the smoothed version $f_\nu$ is $L$-smooth (see~\cite[Lemma 8]{yousefian2012stochastic}), making it tempting to employ the step size rule $\eta_t=\eta\propto 1/L$ for smooth optimization, we do not employ this choice as the constant $L$ is inversely proportional to $\nu$. Since the smoothing radius $\nu$ needs to be controlled to be sufficiently small, the standard choice $\eta_t\propto 1/L$ for smooth optimization will lead to excessively small step sizes and may result in inferior convergence rates. We therefore still adopt the step size rules for strongly convex and nonsmooth optimization even though $f_\nu+m\|x-x_0\|^2$ is smooth.

\item We further add an offset $t_0$ in the step sizes $\eta_t=3/[m(t+t_0)]$, to guarantee that each intermediate iterate $x_t$ in the subroutine \texttt{ZOGD} will remain in $\mathcal{D}_\nu$ with high probability; otherwise the gradient estimator $\sfG(x_{t-1};z_t,\nu)$ may not be well-defined, in which case \cref{algorithm:main} will fail. In fact, for many benchmark optimal or control problems (including LQG, state-feedback $\mathcal{H}_\infty$ control, output-feedback $\mathcal{H}_\infty$ control, etc.), one of the challenges in establishing theoretical guarantees for the associated policy optimization algorithms is to ensure all intermediate iterates will remain in the domain of the objective function. As shown in \cite{bu2021topological,guo2022global,tang2023analysis,zheng2026benign1}, their objective functions have \emph{open} and \emph{nonconvex} domains, making projection onto the domain inadmissible or impractical. We shall see that the incorporation of an appropriate offset $t_0$ will address this issue effectively.
\end{enumerate}

\subsection{Convergence and Complexity Guarantees}
\label{subsec:theoretical_results}

We now present the main result of this article, whose proof will be deferred to \cref{sec:proof_outline,sec:analysis_subroutine}.

\begin{theorem}
\label{theorem:main}
Given any initial point $x(0)\in\mathcal{D}$, let $\Delta_0\geq f(x(0))$ and $\mu=\mu_{24\Delta_0}$, $G=G_{24\Delta_0}$. Then, for any $\delta\in\left(0,\min\{1/5,1/{\ln d}\}\right)$ and sufficiently large $N$, by choosing the algorithmic parameters of \cref{algorithm:main} to satisfy
\[
m=\max\!\left\{\mu,\frac{G^2}{2\Delta_0}\right\},
\quad t_0(n)
=\Theta\!\left(d\ln\frac{n}{\delta}\right),
\quad
\nu(n)\leq \frac{G/m}{6n(n+1)},
\quad 
T(n)=\Theta\!\left(n^2d\ln\frac{n}{\delta}\right),
\]
we have
\[
f(x(N))-f^\ast\leq O\!\left(\frac{1}{N}\right)
\]
with probability at least $1-\delta$. Consequently, for any sufficiently small $\epsilon>0$, the number of zeroth-order queries needed to achieve $f(x(N))-f^\ast\leq\epsilon$ with probability at least $1-\delta$ is bounded by
\[
O\!\left(\frac{d}{\epsilon^3}\ln\frac{1}{\epsilon\delta}\right).
\]
\end{theorem}

\Cref{theorem:main} presents the high probability convergence rate and oracle complexity guarantees for \cref{algorithm:main}. We make some comparisons of our result with the recent literature~\cite{fatkhullin2025global}.
\begin{itemize}
\item The work~\cite{fatkhullin2025global} also proposed (inexact) proximal point algorithms for unconstrained and functionally constrained hidden convex problems. However, \cite{fatkhullin2025global} mainly focuses on a simpler type of hidden convexity, where a bijective transformation $c:\mathcal{X}\rightarrow\mathbb{R}^d$ directly convexifies a nonconvex function $f:\mathcal{X}\rightarrow\mathbb{R}$ via $f(x)=h\circ c(x)$ for some convex $h:c(\mathcal{X})\rightarrow\mathbb{R}$; note that the domain $\mathcal{X}$ is assumed to be convex and closed in \cite{fatkhullin2025global}. Although~\cite{fatkhullin2025global} also studied functionally constrained problems, and discussed generalizations where the invertibility of $c$ can be relaxed, it is not yet clear whether such extensions are directly applicable to the state-feedback $\mathcal{H}_\infty$ policy optimization problem.

\item In terms of the complexity bound, \cite{fatkhullin2025global} showed that their proposed algorithms achieve a first-order oracle complexity of $\widetilde{O}\!\left(\epsilon^{-3}\right)$ for the nonsmooth setting (see Table 1 therein), which is the best-known result for first-order methods applied to nonsmooth hidden convex problems. It can be seen that the complexity bound of our proposed zeroth-order algorithm matches their result except for an additional $O(d)$ factor. We point out that this $O(d)$ factor seems prevalent for derivative-free methods based on zeroth-order gradient estimators \cite{duchi2015optimal,shamir2017optimal,nesterov2017random,kornowski2024optimal}.
\end{itemize}

\subsection{Consequences for State-Feedback $\mathcal{H}_\infty$ Policy Optimization}
\label{subsec:H-infty_policy_optimization_result}

In this subsection, we shift our focus to state-feedback $\mathcal{H}_\infty$ policy optimization, and show how to construct a convex lifting so that the theoretical results in \cref{subsec:theoretical_results} can be applied. As mentioned before, it is known in classical control theory that, via the bounded real lemma and a suitable change of variable, the state-feedback $\mathcal{H}_\infty$ control problem can be reformulated as a convex semidefinite program. Some recent works~\cite{guo2022global,watanabe2025policy} have exploited this convex reformulation to investigate structural properties of the corresponding policy optimization problem~\cref{eq:Hinf_policy_optimization}, and \cite{zheng2026benign2} further proposed the unified extended convex lifting framework that is applicable to a broad range of benchmark control problems. Here we give an explicit construction based on the approach in~\cite{watanabe2025policy}.

The construction of a convex lifting for the state-feedback $\mathcal{H}_\infty$ policy optimization problem~\cref{eq:Hinf_policy_optimization} consists of the following steps:

\noindent\textbf{Step 1: Lifting.} We define the lifted set $\mathfrak{L}_{\mathrm{lft}}$ as
\begin{subequations}
\label{eq:Hinf_ECL_Llft}
\begin{equation}
\mathfrak{L}_{\mathrm{lft}}
\coloneqq
\setv*{(K,\gamma,P)\in\mathbb{R}^{n_u\times n_x}\times\mathbb{R}_{++}\times\mathbbm{S}_{++}^{n_x}}{
\mathscr{M}_{\mathrm{lft}}(K,\gamma,P)\preceq 0
},
\end{equation}
where
\begin{equation}
\mathscr{M}_{\mathrm{lft}}(K,\gamma,P)
=\begin{bmatrix}
(A\!+\!BK)^\tran P(A\!+\!BK) - P + Q+K^\tran RK & (A\!+\!BK)^\tran PB_w \\
B_w^\tran P(A\!+\!BK) & B_w^\tran PB_w-\gamma^2 I
\end{bmatrix}.
\end{equation}
\end{subequations}

\noindent\textbf{Step 2: Convex set.} We then define the convex set $\mathfrak{F}_{\mathrm{cvx}}$ by
\begin{subequations}
\label{eq:Hinf_ECL_Fcvx}
\begin{equation}
\mathfrak{F}_{\mathrm{cvx}}
=\setv*{(\gamma,Y,X)\in\mathbb{R}_{++}\times \mathbb{R}^{n_u\times n_x}\times\mathbbm{S}_{++}^{n_x}}{
\mathscr{M}_{\mathrm{cvx}}(\gamma,Y,X)\succeq 0
},
\end{equation}
where
\begin{equation}
\mathscr{M}_{\mathrm{cvx}}(\gamma,Y,X)
\coloneqq
\begin{bmatrix}
X & 0 & (AX\!+\!BY)^\tran & X & Y^\tran \\
0 & \gamma I & B_w^\tran & 0 & 0 \\
AX\!+\!BY & B_w & X & 0 & 0 \\
X & 0 & 0 & \gamma Q^{-1} & 0 \\
Y & 0 & 0 & 0 & \gamma R^{-1}
\end{bmatrix}.
\end{equation}
\end{subequations}
It is evident that $\mathscr{M}_{\mathrm{cvx}}(\gamma,Y,X)$ is affine in $(\gamma,Y,X)$, and thus $\mathfrak{F}_{\mathrm{cvx}}$ is indeed convex.

\noindent{}\textbf{Step 3: Smooth bijection.} We now define the mapping $\Phi:\mathfrak{L}_{\mathrm{lft}}\rightarrow \mathbb{R}_{++}\times \mathbb{R}^{n_u\times n_x}\times\mathbbm{S}_{++}^{n_x}$ by
\begin{equation}
\label{eq:Hinf_ECL_smooth_bijection}
\Phi(K,\gamma,P) = (\gamma, \gamma KP^{-1}, \gamma P^{-1}).
\end{equation}
The mapping $\Phi$ obviously satisfies the second condition in \cref{assumption:convex_lifting}.

The following proposition justifies that the triple $(\mathfrak{L}_{\mathrm{lft}},\mathfrak{F}_{\mathrm{cvx}},\Phi)$ constructed above is indeed a convex lifting for the objective function $J_\infty$ in the sense of \cref{assumption:convex_lifting}. The proof is given in \cref{appendix:proof_convex_lifting_Hinf}.
\begin{proposition}
\label{proposition:convex_lifting_Hinf}
Let $\mathfrak{L}_{\mathrm{lft}},\mathfrak{F}_{\mathrm{cvx}}$ and $\Phi$ be given by~\cref{eq:Hinf_ECL_Llft,eq:Hinf_ECL_Fcvx,eq:Hinf_ECL_smooth_bijection}, respectively. Then
\begin{enumerate}
\item For any $K\in\mathbb{R}^{n_u\times n_x}$ and $\gamma>0$, we have $(K,\gamma)\in\operatorname{epi}(J_\infty)$ if and only if there exists $P\in\mathbbm{S}_{++}^{n_x}$ such that $(K,\gamma,P)\in\mathfrak{L}_{\mathrm{lft}}$.
\item $\Phi$ is a bijection from $\mathfrak{L}_{\mathrm{lft}}$ to $\mathfrak{F}_{\mathrm{cvx}}$, and both $\Phi$ and $\Phi^{-1}$ are infinitely differentiable.
\end{enumerate}
\end{proposition}

Furthermore, the following proposition shows that the technical assumptions listed in \cref{assumption:technical_objective} are also satisfied by $J_\infty$ and its convex lifting.

\begin{proposition}
\label{proposition:Hinf_analytical_properties}
Given $J_\infty$ and its convex lifting $(\mathfrak{L}_{\mathrm{lft}},\mathfrak{F}_{\mathrm{cvx}},\Phi)$ constructed by ~\cref{eq:Hinf_ECL_Llft,eq:Hinf_ECL_Fcvx,eq:Hinf_ECL_smooth_bijection}, for each $\Delta>0$, we have:
\begin{enumerate}
\item The set $\setc*{(K,\gamma,P)\in\mathfrak{L}_{\mathrm{lft}}}{\gamma\leq\Delta}$
is compact.

\item There exists some $G_\Delta>0$ such that $J_\infty$ is $G_\Delta$-Lipschitz continuous on the sublevel set $\mathcal{S}(\Delta)=\setc*{K\in\mathcal{K}}{J_\infty(K)\leq\Delta}$.

\item There exists some $\mu_\Delta>0$ such that $J_\infty$ is $\mu_\Delta$-weakly convex on any convex subset of the sublevel set $\mathcal{S}(\Delta)$.
\end{enumerate}
\end{proposition}

The second statement of \cref{proposition:Hinf_analytical_properties} has been established in~\cite[Proposition~2]{guo2022global}, while the third statement has been proved in~\cite[Theorem~1]{watanabe2025policy}. A proof of the first statement will be provided in \cref{appendix:proof_Hinf_analytical_properties}.

Combining the above propositions with \cref{theorem:main}, we get the following corollary, which establishes oracle complexity bounds in terms of the objective value gap for state-feedback $\mathcal{H}_\infty$ policy optimization by \cref{algorithm:main}.

\begin{corollary}
\label{corollary:Hinf_complexity_result}
Let $K(0)\in\mathcal{K}$ be any stabilizing policy, let $K(N)$ denote the output of \cref{algorithm:main} applied to~\cref{eq:Hinf_policy_optimization} with initial point $K(0)$, and let $J_\infty^\ast$ denote the optimal value of~\cref{eq:Hinf_policy_optimization}. Then, for any sufficiently small $\epsilon>0$ and $\delta>0$, by appropriately choosing the algorithmic parameters, the number of zeroth-order queries needed to achieve
\[
J_\infty(K(N))-J_\infty^\ast\leq\epsilon
\]
with probability at least $1-\delta$ is bounded by
\[
O\!\left(\frac{n_un_x}{\epsilon^3}\ln\frac{1}{\epsilon\delta}\right).
\]
\end{corollary}

We next provide further discussions on the complexity result of \cref{corollary:Hinf_complexity_result} compared with the existing works~\cite{guo2022global,guo2023complexity,watanabe2025policy}.

\begin{itemize}
\item The work~\cite{guo2022global} studied Goldstein's subgradient method for the state-feedback $\mathcal{H}_\infty$ policy optimization problem:
\[
K(n)=K(n-1)-\eta(n)F(n)/\|F(n)\|,
\]
where $F(n)$ is the minimum-norm element of the $\eta(n)$-Goldstein subdifferential of $J_\infty$ at $K(n-1)$. While the authors showed that by setting $\eta(n)\propto 1/n$ one has $J_\infty(K(n))- J_\infty^\ast\rightarrow 0$, the rate of convergence for the objective value gap is not known. The authors also showed that by setting $\eta(n)=\eta$, the number of oracle queries needed to find an $(\eta,\epsilon)$-Goldstein stationary point is upper bounded by $O(1/(\eta\epsilon))$. However, as commented in~\cite{guo2022global}, $(\eta,\epsilon)$-Goldstein stationarity does not necessarily imply a small objective value gap even for very small $\eta$ and $\epsilon$.

Different from~\cite{guo2022global}, our work establishes convergence rate and oracle complexity results in terms of the objective value gap $J_\infty(K(N))-J_\infty^\ast$, i.e., we provide a complexity result for achieving \emph{global optimality} rather than \emph{stationarity}.

\item The work~\cite{guo2023complexity} proposed a zeroth-order algorithm for state-feedback $\mathcal{H}_\infty$ policy optimization. The authors showed that, by properly choosing the algorithmic parameters, there exists $p\in(0,1)$ such that i) all intermediate iterates will stay in $\mathcal{K}$ with probability at least $p$; ii) the number of zeroth-order oracle queries to find a $(\nu,\epsilon)$-Goldstein stationary point with probability at least $p$ is $O(d^{3/2}/(\nu\epsilon^4))$. It can be seen that the complexity guarantee in~\cite{guo2023complexity} is still with respect to stationarity rather than global optimality; this is partly because \cite{guo2023complexity} studied the more complicated structured $\mathcal{H}_\infty$ control problem, which generally does not admit hidden convexity. Furthermore, the probability $p$ given in~\cite{guo2023complexity} cannot be made arbitrarily close to $1$. Although this defect can be mitigated by repeating the algorithm multiple times and selecting a successful trial as the final output, it makes the algorithm and analysis less elegant.

Contrary to~\cite{guo2023complexity}, given an arbitrarily small $\delta>0$, in order for \cref{algorithm:main} to run successfully and return a near globally optimal point with probability $1-\delta$, one does not need to repeat the whole optimization procedure $\Theta(\ln(1/\delta))$ times and select one successful trial.

We mention that \cite{guo2023complexity} also investigated the case where the $\mathcal{H}_\infty$ cost is estimated from samples, a case not covered by this article. It would be an interesting future direction to combine the sample-based $\mathcal{H}_\infty$ cost estimation method in~\cite{guo2023complexity} with our proposed algorithm and establish its sample complexity bound.

\item The work~\cite{watanabe2025policy} established important analytical properties of $J_\infty$ (including weak convexity, weak PL property, etc.), and investigated the subgradient descent method for state-feedback $\mathcal{H}_\infty$ policy optimization. Note that the complexity guarantee originally given by \cite[Theorem~3]{watanabe2025policy} is still with respect to stationarity rather than global optimality, though by exploiting the weak PL property it seems possible to obtain an $O(\epsilon^{-4})$ complexity bound for the objective value gap. Moreover, in \cite[Theorem~3]{watanabe2025policy}, it is assumed rather than derived that all intermediate iterates will stay in the domain $\mathcal{K}$ with uniformly bounded objective values.

Compared to \cite[Theorem~3]{watanabe2025policy}, our complexity bound is directly with respect to the objective value gap, and its $\widetilde{O}(\epsilon^{-3})$ dependence on $\epsilon$ is better than $O(\epsilon^{-4})$ of \cite[Theorem~3]{watanabe2025policy}. More importantly, our result does not need to impose as an assumption that all intermediate iterates are in $\mathcal{K}$; rather, we theoretically establish that this will hold with high probability.
\end{itemize}

\section{Outline of Convergence Analysis}
\label{sec:proof_outline}

In this section, we provide an outline of the proof of \cref{theorem:main}. For notational convenience, we define $F_{x_0}:\mathbb{R}^d\rightarrow\mathbb{R}$ by
\[
F_{x_0}(x)\coloneq f(x)+m\|x-x_0\|^2,
\]
where $m$ is the parameter in \cref{algorithm:main}, and $x_0$ is an arbitrary point in $\mathcal{D}$.

The complete proof of \cref{theorem:main} consists of three parts:
\begin{enumerate}
\item Convergence of the subroutine: Given any $x_0\in\mathcal{D}$, as long as the algorithmic parameters are chosen properly, the subroutine \texttt{ZOGD} is able to return a sufficiently good minimizer of $F_{x_0}(x)$ over $x\in\mathcal{D}$
with high probability.

\item Convergence of the proximal point iterations: Assuming the subroutine \texttt{ZOGD} can provide a sufficiently good minimizer of $F_{x_0}(x)$ over $x\in\mathcal{D}$, by choosing the algorithmic parameters properly, the objective value gap $f(x(N))-f^\ast$ of the outer-loop proximal point iterations will be effectively upper bounded.

\item Convergence of \cref{algorithm:main}: We combine the results from the two previous parts to finally obtain the conclusions in \cref{theorem:main}.
\end{enumerate}

The first part of our analysis is technically challenging. As mentioned before, the domain of the objective function $\mathcal{D}$ is open and nonconvex, making it difficult to maintain all the iterates within $\mathcal{D}_\nu$ so that the gradient estimator $\sfG(x_{t-1};z_t,\nu)$ is well-defined. A common technique to address this issue is to exploit the coercivity of the objective function and let each iteration be a descent step. In this way, all $x_t$ will remain in a sublevel set that has a strictly positive distance to the boundary of $\mathcal{D}$. However, for our problem setup, adopting this technique raises further challenges:
\begin{itemize}
\item The function $f$ is possibly nonsmooth in our problem setup. As a result, ensuring $F_{x_0}(x_t)\leq F_{x_0}(x_{t-1})$ for all $t\in[T]$ becomes much harder compared to the smooth case~\cite{liao2025proximal,li2025subgradient}.
\item The zeroth-order gradient estimator $\sfG(x_{t-1};z_t,\nu)$ is random, suggesting that we need to establish high-probability convergence guarantees.
\end{itemize}
The above challenges have also been identified by existing literature. In \cite{guo2022global}, the authors exploited the fact that the minimum-norm element
of the Goldstein subdifferential generates a good descent direction; however, finding this minimum-norm element is itself highly nontrivial as discussed in~\cite{guo2022global}. In~\cite{guo2023complexity}, the authors employed the zeroth-order gradient estimator~\cref{eq:grad_est_definition} in their algorithm for state-feedback $\mathcal{H}_\infty$ policy optimization, and provided guarantees on the probability that their proposed algorithm will succeed, using tools from martingale concentration and a descent lemma on the smoothed version; however, their original guarantee was weak as the probability of success cannot be made arbitrarily close to $1$ (though it can be augmented to be $1-\delta$ by repeating the procedure $\Theta(\ln(1/\delta))$ times). In~\cite{watanabe2025policy}, the authors circumvented this issue by imposing it as an assumption that all intermediate iterates lie in $\mathcal{D}$, in their convergence theorem.

In this work, we address the aforementioned challenges by a somewhat different approach: Instead of showing that $F_{x_0}(x_t)\leq F_{x_0}(x_{t-1})$ with high probability, we try to bound the distance from $x_t$ to $x_\nu^\star= \argmin_x\{f_\nu(x)+m\|x-x_0\|^2\}$. By properly setting the parameters $m,\nu$ and $t_0$, we can find a ball $x_\nu^\star+R\mathbb{B}_d$ for some $R>0$, such that i) $x_\nu^\star+R\mathbb{B}_d$ is a subset of $\mathcal{D}_\nu$; and ii) $\|x_t-x_\nu^\star\|\leq R$ with high probability. In this way, we establish that $x_t\in\mathcal{D}_\nu$ for all $t\in[T]$ with high probability.

The complete proof of convergence of the subroutine is long, and therefore we postpone it to \cref{sec:analysis_subroutine}. Here we only present the final result.

\begin{theorem}
\label{theorem:subroutine_convergence}
Consider the subroutine $\mathtt{ZOGD}(x_0,m,\nu,T,t_0)$. Let $\Delta\geq f(x_0)$, and denote $\mu=\mu_{12\Delta},G=G_{12\Delta}$. Let $\delta\in(0,\min\{1/5,1/\ln d\})$ be arbitrary. Suppose $m=\max\{\mu,G^2/\Delta\}$, $\nu\leq G/(2m)$, and for each $t\in[T]$, the step sizes are given by
\[
\eta_t=\frac{3}{m(t+t_0)},\qquad\text{where}
\quad
t_0= 50d\ln(1+7\ln(700d))+25d\ln\frac{2\pi^2}{3\delta}.
\]
Then with probability at least $1-\delta$, we have $x_t\in \mathcal{D}_\nu$ for all $t\in[T]$, and
\[
F_{x_0}(\overline{x}_T)-\inf_{x\in\mathcal{D}} F_{x_0}(x)
\leq C\!\left(\frac{d\ln(1/\delta)}{T}\cdot \Delta
+\left[\frac{1}{T}
+\left(\frac{d\ln(1/\delta)}{T}\right)^{\!2}\right]\frac{m\Delta^2}{G^2}
\right)
+\frac{3G\nu}{2},
\]
where $C\geq 1$ is some absolute numerical constant.
\end{theorem}

Below we present the second and the third parts of our analysis.

\subsection{Convergence of the Proximal Point Iterations}

Let the initial point $x(0)\in\mathcal{D}$ be given, and let $\Delta\geq f(x(0))$ be arbitrary. By Part 1 of \cref{assumption:technical_objective}, the set $\setc*{(x,\lambda,\xi)\in\mathfrak{L}_{\mathrm{lft}}}{\lambda\leq\Delta}$ is compact, and thus its image under $\Phi$
\[
\setc*{(\gamma,\zeta)\in\mathfrak{F}_{\mathrm{cvx}}}{\gamma\leq\Delta}
\]
is also compact. We therefore may define $D_\Delta$ to be the finite diameter of the above set:
\[
D_{\Delta}\coloneqq \sup
\setc*{\left\|(\gamma_1,\zeta_1)-(\gamma_2,\zeta_2)\right\|}{(\gamma_i,\zeta_i)\in\mathfrak{F}_{\mathrm{cvx}}\text{ and }\gamma_i\leq\Delta\text{ for }i=1,2}.
\]
Furthermore, since $\Phi^{-1}$ is locally Lipschitz continuous,
there exists some $\chi_\Delta>0$ such that $\Phi^{-1}$ is $\chi_\Delta$-Lipschitz continuous on the compact set $\setc*{(\gamma,\zeta)\in\mathfrak{F}_{\mathrm{cvx}}}{\gamma\leq\Delta}$, i.e.,
\[
\|\Phi^{-1}(\gamma_1,\zeta_1)-\Phi^{-1}(\gamma_2,\zeta_2)\|
\leq\chi_\Delta\|(\gamma_1,\zeta_1)-(\gamma_2,\zeta_2)\|
\]
for all $(\gamma_1,\zeta_1),(\gamma_2,\zeta_2)\in \mathfrak{F}_{\mathrm{cvx}}$ with $\max\{\gamma_1,\gamma_2\}\leq\Delta$ (see, e.g., \cite[Theorem 2.1.6 \& Remark 2.1.7]{cobzas2019lipschitz}).

\begin{lemma}
\label{lemma:outer-loop_convergence}
Denote $F_\star(n) = \inf_{x\in\mathcal{D}}F_{x(n-1)}(x)$, and suppose for each $n\in[N]$, we have $f(x(n-1))\leq \Delta$ and
\[
F_{x(n-1)}\!\left(x(n)\right)
\leq F_\star(n) + \varepsilon_{\mathrm{in}}(n)
\]
for some $\Delta\geq f(x(0))$ and a sequence of positive reals $\left(\varepsilon_{\mathrm{in}}(n)\right)_{n=1}^N$. Then,
\begin{equation}
\label{eq:outerloop_func_value_raw}
f\!\left(x(N)\right)
-f^\ast
\leq
\frac{4m\chi_\Delta^2D_\Delta^2}{N+1}
+\frac{1}{N(N+1)}\sum_{n=1}^N n(n+1)\varepsilon_{\mathrm{in}}(n).
\end{equation}
\end{lemma}
\begin{proof}
By Part 1 of \cref{assumption:technical_objective}, the optimal value $f^\ast$ is attained at some $x^\ast\in\mathcal{D}$. Then by \cref{assumption:convex_lifting}, there exists $\xi^\ast\in\mathbb{R}^{d_\xi}$ such that $(x^\ast,f^\ast,\xi^\ast)\in\mathfrak{L}_{\mathrm{lft}}$. Evidently $f^\ast\leq\Delta$.

Now let $n\in[N]$ and $\alpha\in[0,1]$ be arbitrary. \Cref{assumption:convex_lifting} then guarantees that we can find $\xi(n-1)\in\mathbb{R}^{d_\xi}$ such that $(x(n-1),f(x(n-1)),\xi(n-1))\in\mathfrak{L}_{\mathrm{lft}}$. We let
\begin{align*}
u(n-1)={} &(f(x(n-1)),\zeta(n-1))=\Phi(x(n-1),f(x(n-1)),\xi(n-1)), \\
u^\ast={} & (f^\ast,\zeta^\ast)= \Phi(x^\ast,f^\ast,\xi^\ast),
\end{align*}
which are in $\mathfrak{F}_{\mathrm{cvx}}$. Since $\mathfrak{F}_{\mathrm{cvx}}$ is convex, we have $(1-\alpha)u(n-1)+\alpha u^\ast\in\mathfrak{F}_{\mathrm{cvx}}$. Note that the first component of $(1-\alpha)u(n-1)+\alpha u^\ast$ equals
\[
(1-\alpha)f(x(n-1))+\alpha f^\ast,
\]
and by the assumption $f(x(n-1))\leq\Delta$ of the lemma, we have $(1-\alpha)f(x(n-1))+\alpha f^\ast\leq\Delta$. We then form
\[
\begin{aligned}
(x_\alpha(n-1),\gamma_\alpha(n-1),\xi_\alpha(n-1)) ={} &
\Phi^{-1}\!\left((1-\alpha)u(n-1)+\alpha u^\ast\right) \\
={} &
\Phi^{-1}\!\left((1-\alpha)(f(x(n-1)),\zeta(n-1))+\alpha(f^\ast,\zeta^\ast)\right),
\end{aligned}
\]
which is in $\mathfrak{L}_{\mathrm{lft}}$. By Part 1 of \cref{assumption:convex_lifting} and the definitions of $D_\Delta$ and $\chi_\Delta$, we have
\begin{align*}
& \left\|x_\alpha(n-1)-x(n-1)\right\| \\
\leq{} & \|(x_\alpha(n-1),\gamma_\alpha(n-1),\xi_\alpha(n-1))-(x(n-1),f(x(n-1)),\xi(n-1))\| \\
={} & \left\|\Phi^{-1}
((1-\alpha)u(n-1)+\alpha u^\ast)
-\Phi^{-1}(u(n-1))
\right\| \\
\leq{} & \chi_\Delta
\|(1-\alpha)u(n-1)+\alpha u^\ast-u(n-1)\| \\
={} & \chi_\Delta\|\alpha\cdot (u^\ast-u(n-1))\|
\leq \chi_\Delta D_\Delta \alpha,
\numberthis
\label{eq:outerloop_distance_x_lambda}
\end{align*}
Furthermore, since 
\[
(x_\alpha(n-1),\gamma_\alpha(n-1))=\pi_{x,\lambda}(x_\alpha(n-1),\gamma_\alpha(n-1),\xi_\alpha(n-1))\in\operatorname{epi}(f),
\]
we have
\begin{equation}
\label{eq:outerloop_func_value_convex_combination}
f(x_\alpha(n-1))\leq \gamma_\alpha(n-1)
=(1-\alpha)f(x(n-1))+\alpha f^\ast.
\end{equation}

Now we are ready to derive the bound~\cref{eq:outerloop_func_value_raw}. For each $n\in[N]$, we have
\begin{align*}
f(x(n))-f^\ast
\leq {} & f(x(n))+m\|x(n)-x(n-1)\|^2-f^\ast \\
={} & F_{x(n-1)}(x(n))-f^\ast \\
\leq{} & F_\star(n)+\varepsilon_{\mathrm{in}}(n)-f^\ast \\
\leq{} &
F_{x(n-1)}(x_\alpha(n-1)) + \varepsilon_{\mathrm{in}}(n)-f^\ast \\
={} &
f(x_\alpha(n-1))+m\|x_\alpha(n-1)-x(n-1)\|^2+ \varepsilon_{\mathrm{in}}(n)-f^\ast \\
\leq{} &
(1-\alpha)(f(x(n-1))-f^\ast)
+m\chi_\Delta^2 D_\Delta^2\alpha^2+\varepsilon_{\mathrm{in}}(n),
\numberthis \label{eq:outer-loop_temp1}
\end{align*}
where the third step follows from the lemma's assumption, and we used~\cref{eq:outerloop_distance_x_lambda} and~\cref{eq:outerloop_func_value_convex_combination} in the last step. We now set $\alpha=2/(n+1)$ and obtain
\[
f(x(n))-f^\ast
\leq \frac{n-1}{n+1}(f(x(n-1))-f^\ast)+\frac{4m\chi_\Delta^2 D_\Delta^2}{(n+1)^2}+\varepsilon_{\mathrm{in}}(n).
\]
By induction, we get
\[
f(x(N))-f^\ast
\leq \prod_{n=1}^N\frac{n-1}{n+1}(f(x(0))-f^\ast)
+\sum_{n=1}^N
\left(\frac{4m\chi_\Delta^2 D_\Delta^2}{(n+1)^2}+\varepsilon_{\mathrm{in}}(n)\right)\prod_{k=n+1}^N\frac{k-1}{k+1}.
\]
Then we note that
\[
\prod_{k=n+1}^N\frac{k-1}{k+1}
=\frac{n}{n+2}\cdot\frac{n+1}{n+3}\cdots\frac{N-2}{N}\cdot\frac{N-1}{N+1}
=\frac{n(n+1)}{N(N+1)}.
\]
Thus
\[
\begin{aligned}
f(x(N))-f^\ast
\leq{} & 
\sum_{n=1}^N
\left(\frac{4m\chi_\Delta^2 D_\Delta^2}{(n+1)^2}+\varepsilon_{\mathrm{in}}(n)\right)\frac{n(n+1)}{N(N+1)} \\
\leq{} &
4m\chi_\Delta^2 D_\Delta^2 \sum_{n=1}^N\frac{1}{N(N+1)}
+ \sum_{n=1}^N \varepsilon_{\mathrm{in}}(n)\frac{n(n+1)}{N(N+1)} \\
={} &
\frac{4m\chi_\Delta^2D_\Delta^2}{N+1}
+\frac{1}{N(N+1)}\sum_{n=1}^N n(n+1)\varepsilon_{\mathrm{in}}(n),
\end{aligned}
\]
which is just the bound~\cref{eq:outerloop_func_value_raw}.
\end{proof}

\begin{remark}
The proof of \cref{{lemma:outer-loop_convergence}} basically follows the argument in \cite{fatkhullin2025global} but generalized to the convex lifting setup considered in this work. In addition, we employ a different choice of the quantity $\alpha$ after obtaining~\cref{eq:outer-loop_temp1}: The work~\cite{fatkhullin2025global} used a constant $\alpha$ that is proportional to $\epsilon$, the predetermined bound on the final objective value gap, while we let $\alpha$ vary with the iteration index $n$. In this way, the polylogarithmic factors in the complexity bound will have lower degrees. We also allow $\varepsilon_{\mathrm{in}}$, the accuracy of the subroutine's output, to depend on $n$; an advantage of this modification is that the algorithmic parameters of the subroutine \texttt{ZOGD} may be chosen to only depend on $n$, enabling \cref{algorithm:main} to loop indefinitely without the need to predetermine a bound on the final objective value gap.
\end{remark}

\subsection{Proof of \Cref{theorem:main}}

We are now ready to prove \cref{theorem:main}. Recall that we let $\Delta_0\geq f(x(0))$, and $\mu=\mu_{24\Delta_0}, G=G_{24\Delta_0}$. The specific value of $t_0$ will be set to
\[
t_0(n)= 50d\ln(1+7\ln(700d))+25d\ln\frac{2n(n\!+\!1)\pi^2}{3\delta}.
\]
Now denote
\[
\widetilde{E}_n= \bigcap_{\tau=1}^{n}
\left\{F_{x(\tau-1)}(x(\tau))
\leq F_\star(\tau)+\frac{\Delta_0}{\tau(\tau+1)}\right\},
\qquad\forall n\in[N],
\]
where we recall that $F_\star(n) = \inf_{x\in\mathcal{D}}F_{x(n-1)}(x)$. We also let $\widetilde{E}_0$ denote the certain event.

We first show by induction that on $\tilde{E}_{n-1}$ we have $f(x(n-1))\leq (2-1/n)\Delta_0$ for all $n\in[N]$. The base case $f(x(0))\leq \Delta_0$ follows by our choice of $\Delta_0$. Now suppose the event $\widetilde{E}_n$ occurs and $f(x(n-1))\leq [2-1/n]\Delta_0$ for some $n\in[N-1]$. Then we have
\[
\begin{aligned}
f(x(n))\leq{} &
F_{x(n-1)}(x(n))
\leq F_\star(n)+\frac{\Delta_0}{n(n+1)} \\
\leq{} & F_{x(n-1)}(x(n-1)) + \frac{\Delta_0}{n(n+1)}
=f(x(n-1))+\frac{\Delta_0}{n(n+1)} \\
\leq{} & \left(2-\frac{1}{n}\right)\Delta_0+\frac{\Delta_0}{n(n+1)}
=\left(2-\frac{1}{n+1}\right)\Delta_0.
\end{aligned}
\]
By induction, we see that $f(x(n-1))\leq (2-1/n)\Delta_0$ on $\widetilde{E}_{n-1}$ for all $n\in[N]$.

Next we show that $\mathbb{P}\big(\widetilde{E}_n\big)\geq 1-n\delta/(n+1)$, which is also by induction. The base case is trivial. Now suppose $\mathbb{P}\big(\widetilde{E}_{n-1}\big)\geq 1-(n-1)\delta/n$, and the event $\widetilde{E}_{n-1}$ occurs. Since $f(x(n-1))\leq 2\Delta_0$ on $\widetilde{E}_{n-1}$, by our choice of $m,\nu$ and $t_0(n)$, we can apply \cref{theorem:subroutine_convergence} and get
\[
\begin{aligned}
& F_{x(n-1)}(x(n))-F_\star(n) \\
\leq{} &
C\!\left(\frac{2d\Delta_0}{T(n)}\ln\frac{n(n\!+\!1)}{\delta}
+\left[\frac{1}{T(n)}
+\left(\frac{d}{T(n)}\ln\frac{n(n\!+\!1)}{\delta}\right)^{\!2}\right]\frac{4m\Delta_0^2}{G^2}
\right)
+\frac{G^2}{4mn(n\!+\!1)}
\end{aligned}
\]
with probability at least $1-\frac{\delta}{n(n+1)}$ conditioned on $\widetilde{E}_{n-1}$. By our choice of $m$, we have $G^2\leq 2m\Delta_0$ and consequently
\[
\frac{G^2}{4mn(n+1)}\leq\frac{\Delta_0}{2n(n+1)}.
\]
Moreover, by choosing
\[
T(n)=\left\lceil
24Cn(n+1)d\frac{m\Delta_0^2}{G^2}\cdot
\ln\frac{n(n\!+\!1)}{\delta}
\right\rceil
=\Theta\!\left(n^2d\ln\frac{n}{\delta}\right),
\]
we have
\[
C\!\left(\frac{2d\Delta_0}{T(n)}\ln\frac{n(n\!+\!1)}{\delta}
+\left[\frac{1}{T(n)}
+\left(\frac{d}{T(n)}\ln\frac{n(n\!+\!1)}{\delta}\right)^{\!2}\right]\frac{4m\Delta_0^2}{G^2}
\right)
\leq \frac{\Delta_0}{2n(n+1)},
\]
where we also used $G^2\leq 2m\Delta_0$. Now we get
\[
\begin{aligned}
F_{x(n-1)}(x(n))
\leq{} & F_\star(n)
+\frac{\Delta_0}{n(n+1)}.
\end{aligned}
\]
with probability at least $1-\frac{\delta}{n(n+1)}$ conditioned on $\widetilde{E}_{n-1}$. Consequently,
\[
\mathbb{P}\!\left(\widetilde{E}_{n}\right)
=\mathbb{P}\Pcond*{\widetilde{E}_{n}}{\widetilde{E}_{n-1}}\cdot \mathbb{P}\!\left(\widetilde{E}_{n-1}\right)
\geq\left(1-\frac{(n-1)\delta}{n}\right)\left(1-\frac{\delta}{n(n+1)}\right)\geq
1-\frac{n\delta}{n+1}.
\]
We can now apply mathematical induction to get $\mathbb{P}\big(\widetilde{E}_{n}\big)\geq 1-\frac{n\delta}{n+1}$ for all $n\in[N]$, and particularly,
\[
\mathbb{P}\!\left(\widetilde{E}_{N}\right)\geq 1-\delta.
\]
Finally, we can apply \cref{lemma:outer-loop_convergence} and see that
\[
f(x(N))-f^\ast
\leq \frac{4m\chi_{2\Delta_0}^2 D_{2\Delta_0}^2}{N+1}
+\frac{1}{N(N+1)}\sum_{n=1}^N n(n+1)\cdot\frac{\Delta_0}{n(n+1)}
=O\!\left(\frac{1}{N}\right)
\]
on the event $\widetilde{E}_{N}$. Summarizing all the obtained results completes the derivation of the convergence rate.

To show the complexity bound, we only need to notice that, by $f(x(N))-f^\ast\leq O\!\left(1/N\right)$, it suffices to set $N=\Theta(1/\epsilon)$ to obtain $f(x(N))-f^\ast\leq\epsilon$. Since each inner-loop iteration contains two function value queries, the total number of function value queries is then
\[
O\!\left(\sum_{n=1}^N T(n)\right)
=O\!\left(\sum_{n=1}^N
n^2d\ln\frac{n}{\delta}
\right)
= O\!\left(N^3 d\ln\frac{N}{\delta}
\right)
=O\!\left(\frac{d}{\epsilon^3}\ln\frac{1}{\epsilon\delta}\right).
\]

\section{Analysis of the Subroutine}
\label{sec:analysis_subroutine}

In this section, we analyze the convergence of the subroutine $\mathtt{ZOGD}$. The analysis will be conducted within the subroutine, so that the notations will be consistent with those of the subroutine in \cref{algorithm:main} rather than the outer loop. We consider the following version of iterations for the subroutine $\mathtt{ZOGD}(x_0,m,\nu,T)$:
\begin{equation}
\label{eq:main_inner_iteration}
\begin{aligned}
g_t={} & \sfG(x_{t-1};z_{t,b},\nu), \\
\widehat{x}_t ={} & x_{t-1}-\eta_t(g_t+2m(x_{t-1}-x_0)), \\
x_{t} ={} & \begin{cases}
\widehat{x}_t, & \text{if } \widehat{x}_{t}\in \mathcal{C}, \\
x_{t-1}, & \text{otherwise}
\end{cases}
\end{aligned}
\end{equation}
for each $t=1,2,\ldots,T$. We recall that $z_{t}$ for $t=1,\ldots,T$ are i.i.d. random vectors, each of which follows the distribution $\mathcal{U}(\mathbb{S}_{d-1})$. Compared to the subroutine detailed in \cref{algorithm:main}, we add a safeguard step that checks whether the result of one gradient step $\widehat{x}_t$ lies within some safeguard set $\mathcal{C}$. If $\widehat{x}_t\in\mathcal{C}$ we move the optimization procedure forward; otherwise we pause and discard this gradient update. The introduction of this safeguard step is purely to facilitate the theoretical analysis, but guarantees that $x_t$ remains in some safe region so that we can establish crucial properties of $f_\nu$ and the gradient estimator for convergence analysis. We shall later see that by properly choosing $\mathcal{C}$ and other parameters of the algorithm, we will have $x_t=\widehat{x}_t\in\mathcal{C}$ for all $t\in[T]$ with high probability, allowing us to mark the optimization procedure as failed whenever $\widehat{x}_t\notin\mathcal{C}$ occurs. The form of $\mathcal{C}$ will be specified later. We remark that, in practical implementations, it is not necessary to know $\mathcal{C}$ or perform the safeguard step $x_t=x_{t-1}$ when $\widehat{x}_{t}\notin\mathcal{C}$.

As mentioned earlier, the iterations~\cref{eq:main_inner_iteration} can be understood as a variant of zeroth-order gradient descent applied to the objective $F_{x_0}(x)=f(x)+m\|x-x_0\|^2$.
For notational convenience, we define the function $F_{\nu,x_0}:\mathcal{D}_\nu\rightarrow\mathbb{R}$ by
\[
\begin{aligned}
F_{\nu,x_0}(x) \coloneq{} & f_\nu(x)+m\|x-x_0\|^2.
\end{aligned}
\]
$F_{\nu,x_0}$ can be regarded as a ``partially smoothed'' version of $F_{x_0}$. We also define the filtration $(\mathcal{F}_t)_{t=0}^{T}$ by
\[
\mathcal{F}_0=\sigma(\{\varnothing\}),\qquad
\mathcal{F}_t = \sigma(z_1,z_2,\ldots,z_{t}),\quad \forall t\in[T].
\]
It can be seen that each $x_t$ is $\mathcal{F}_t$-measurable.

\subsection{Auxiliary Lemmas and Concentration Inequalities}

Before starting the proof pipeline, we first present some useful results on sub-Gaussian random variables and relevant concentration inequalities.

\begin{lemma}[{\cite[Proposition 2.6.1]{vershynin2026high}}]
\label{lemma:sub-Gaussian_zero_mean_MGF}
Let $X$ be a random variable satisfying $\mathbb{E}\!\left[\exp\!\left(\frac{X^2}{\kappa^2}\right)\right]\leq 2$
for some $\kappa>0$ and $\mathbb{E}[X]=0$. Then
\[
\mathbb{E}[\exp(\lambda X)]\leq\exp\!\left(\frac{3}{2}\lambda^2\kappa^2\right)
\]
for all $\lambda\in\mathbb{R}$.
\end{lemma}

The following Ville's inequality is a classical result in martingale concentration.

\begin{lemma}[Ville's inequality, {\cite[Lemma 1]{howard2020time}}]
\label{lemma:Ville_inequality}
Let $L_1,\ldots,L_t$ be a non-negative supermartingale. Then, for any $a>0$, we have
\[
\mathbb{P}\!\left(\exists k\in[t]\textrm{ s.t. }L_k\geq a\right)\leq\frac{\mathbb{E}[L_1]}{a}.
\]
\end{lemma}

The next lemma is a corollary of Ville's inequality, which provides a time-uniform high-probability bound on the sum of non-negative random variables satisfying certain conditional exponential-moment bounds. Its proof is given in \cref{appendix:proof_time-uniform_poly_growth_v} for completeness.

\begin{lemma}
\label{lemma:time-uniform_poly_growth_v}
Suppose $(h_t)_{t=1}^T$ is a sequence of non-negative random variables such that each $h_t$ is $\mathcal{F}_t$-measurable, and that
\[
\mathbb{E}\Econd*{\exp\!\left(\frac{h_t}{v_t}\right)}{\mathcal{F}_{t-1}}\leq 2.
\]
for some positive reals $v_1\leq v_2\leq\ldots\leq v_T$. Denote $V_t=\sum_{i=1}^t v_i$. Furthermore, suppose there exist positive constants $\alpha,\beta$ such that
\begin{equation}
\label{eq:lemma_time-uniform_poly_growth_v_condition}
\min_{t\in[T]}\frac{V_t}{v_t}\geq \alpha t+\beta.
\end{equation}
Then for $c>\ln 2$, we have
\[
\mathbb{P}\!\left(\exists t\in[T]\text{ s.t. }
\sum_{k=1}^t h_k> cV_t
\right)\leq\frac{\exp(-(\alpha+\beta)(c-\ln 2))}{1-\exp(-\alpha(c-\ln 2))}.
\]
\end{lemma}

The next theorem provides a time-uniform bound for analyzing the sum of conditionally zero-mean sub-Gaussian random variables. It will be used in finding an appropriate value for $t_0$ so that the failure probability will be bounded.

\begin{theorem}[{\cite[Theorem 1]{howard2021time}}]
\label{theorem:stitched_boundary}
Let $(\theta_t)_{t=1}^T$ be a sequence of random variables such that each $\theta_t$ is $\mathcal{F}_t$-measurable and
\[
\mathbb{E}\Econd*{\theta_t}{\mathcal{F}_{t-1}}=0\qquad\text{a.s.}
\]
Suppose for each $t\in[T]$, $v_{t-1}$ is a positive real number such that
\[
\mathbb{E}\Econd*{\exp(\lambda \theta_t)}{\mathcal{F}_{t-1}}
\leq \exp\!\left(\frac{\lambda^2}{2}v_{t-1}\right),\qquad\forall \lambda>0.
\]
Let $S_t=\sum_{i=1}^t \theta_i$ and $V_t=\sum_{i=1}^t v_{i-1}$. Then
\[
\mathbb{P}\!\left(\exists t\in[T]\text{ s.t. }S_t\geq u_\delta(V_t)\right)\leq\delta,
\]
where the function $u_\delta:[v_0,+\infty)\rightarrow[0,+\infty)$ is given by
\[
u_\delta(x) = \frac{e^{1/4}+e^{-1/4}}{\sqrt{2}}
\sqrt{x\left(2\ln\!\left(1+\ln\frac{x}{v_0}\right)+\ln\frac{\pi^2}{6\delta}\right)}.
\]
\end{theorem}

Finally, the following generalized Freedman inequality provides a powerful tool for establishing high-probability convergence guarantees of stochastic subgradient descent for strongly-convex optimization.

\begin{theorem}[Generalized Freedman inequality, {\cite{harvey2019tight}}]
\label{theorem:generalized_freedman}
Let $(\theta_t)_{t=1}^T$ be a sequence of random variables such that each $\theta_t$ is $\mathcal{F}_t$-measurable and
\[
\mathbb{E}\Econd*{\theta_t}{\mathcal{F}_{t-1}}=0\qquad\text{a.s.}
\]
Suppose for each $t\in[T]$, $v_{t-1}$ is a non-negative and $\mathcal{F}_{t-1}$-measurable random variable satisfying
\[
\mathbb{E}\Econd*{\exp(\lambda \theta_t)}{\mathcal{F}_{t-1}}
\leq \exp\!\left(\frac{\lambda^2}{2}v_{t-1}\right),\qquad\forall \lambda>0.
\]
Let $S_t=\sum_{i=1}^t \theta_i$ and $V_t=\sum_{i=1}^t v_{i-1}$. Let $\alpha_1,\ldots,\alpha_T,\beta$ be non-negative reals that are not simultaneously zero, and denote $\alpha=\max_{1\leq t\leq T}\alpha_t$.
Then we have
\[
\mathbb{P}\!\left(\exists t\in[T]\text{ s.t. } S_t\geq x\text{ and }V_t\leq \sum_{i=1}^t \alpha_i \theta_{i}+\beta\right)
\leq \exp\!\left(-\frac{x^2}{4\alpha x+8\beta}\right)
\]
for all $x>0$.
\end{theorem}

\subsection{Convex Neighborhoods of $x_0$}

As a starting point of our analysis, we first investigate certain convex neighborhoods of $x_0$. These convex neighborhoods will be fundamental for setting the domain over which $F_{x_0}$ and $F_{\nu,x_0}$ become convex, and for establishing critical analytical properties of $F_{x_0}$ and $F_{\nu,x_0}$.

Let $\Delta\geq f(x_0)$, and for notational simplicity we denote $G=G_{12\Delta}$ and $\mu=\mu_{12\Delta}$ in this section, so that $f$ is $G$-Lipschitz continuous on $\mathcal{S}(12\Delta)$ and is $\mu$-weakly convex on any convex subset of $\mathcal{S}(12\Delta)$. We set the algorithm parameter $m$ in~\cref{eq:main_inner_iteration} to be
\begin{equation}
\label{eq:parameter_choice_m}
m = \max\left\{\mu,
\frac{G^2}{\Delta}
\right\}.
\end{equation}
For each $1\leq\alpha\leq 12$, we denote
\[
\mathcal{S}_{x_0}(\alpha\Delta)\coloneq
\setc*{x\in\mathcal{D}}{f(x)+m\|x-x_0\|^2\leq \alpha\Delta}.
\]
Since $m\|x-x_0\|^2\geq 0$, we obviously have $\mathcal{S}_{x_0}(\alpha\Delta)\subseteq \mathcal{S}(\alpha\Delta)$. 

\begin{lemma}
\label{lemma:convex_subdomain_prox}
Let $\alpha\in(1,12]$ be arbitrary. The following statements hold:
\begin{enumerate}
\item $x_0+\frac{(\alpha-1)\Delta}{G}\mathbb{B}_d\subseteq\mathcal{S}(\alpha\Delta)$.
\item The function $F_{x_0}$ is $m$-strongly convex on $x_0+\frac{(\alpha-1)\Delta}{G}\mathbb{B}_d$.
\item The set $\mathcal{S}_{x_0}(\alpha\Delta)$ is a convex subset of $x_0+\frac{\gamma(\alpha)\Delta}{G}\mathbb{B}_d$, where
$\gamma(\alpha) = (1+\sqrt{1+4\alpha})/2$.
\end{enumerate}
\end{lemma}
\begin{proof}
We first show that
\begin{equation}
\label{eq:lemma_convex_subdomain_prox_step1}
x_0+\frac{(\alpha-1)\Delta}{G}\mathbb{B}_d\subseteq\mathcal{S}(\alpha\Delta).
\end{equation}
Suppose on the contrary that $\|x-x_0\|\leq\frac{(\alpha-1)\Delta}{G}$ and $x\notin\mathcal{S}(\alpha\Delta)$ for some $x\in\mathbb{R}^d$. Let
\[
R=\sup\setc*{r}{x_0+s(x-x_0)\in\mathcal{S}(\alpha\Delta)\text{ for all }s\in[0,r]}.
\]
Note that since $f(x_0)\leq\Delta<\alpha\Delta$ and $f$ is continuous, $R$ must be strictly positive, and since $x\notin\mathcal{S}(\alpha\Delta)$, we have $R<1$. Moreover, $x_0+s(x-x_0)\in\mathcal{S}(\alpha\Delta)\subseteq\mathcal{S}(12\Delta)$ for all $s\in[0,R]$. Therefore
\[
f(x_0+R(x-x_0))\leq f(x_0)+G R\|x-x_0\|
< \Delta+G\cdot \frac{(\alpha-1)\Delta}{G}=\alpha\Delta.
\]
However, since $f$ is continuous on the open set $\mathcal{D}$ and $\mathcal{S}(\alpha\Delta)$ is compact, we must have $f(x_0+R(x-x_0))=\alpha\Delta$, leading to a contradiction. Thus~\cref{eq:lemma_convex_subdomain_prox_step1} holds. 

Now that $x_0+\frac{(\alpha-1)\Delta}{G}\mathbb{B}_d$ is a convex subset of the sublevel set $\mathcal{S}(\alpha\Delta)$, we can conclude from Part 3 of \cref{assumption:technical_objective} that the function $F_{x_0}$ is $2(m-\mu/2)$-strongly convex on $x_0+\frac{(\alpha-1)\Delta}{G}\mathbb{B}_d$. By the definition of $m$, we have $2(m-\mu/2)\geq m$, and so $F_{x_0}$ is $m$-strongly convex on $x_0+\frac{(\alpha-1)\Delta}{G}\mathbb{B}_d$.

Next we prove the third statement. Let $x\in\mathcal{S}_{x_0}(\alpha\Delta)$ be arbitrary. We have
\[
\begin{aligned}
\alpha\Delta
\geq{} & f(x)+m\|x-x_0\|^2 \\
\geq{} & f(x_0)-G\|x-x_0\|+m\|x-x_0\|^2 \\
\geq{} & -G\|x-x_0\|+\frac{G^2}{\Delta}\|x-x_0\|^2,
\end{aligned}
\]
where the second step follows from the $G$-Lipschitz continuity of $f$ on $\mathcal{S}_{x_0}(\alpha\Delta)\subseteq\mathcal{S}(12\Delta)$. We rewrite the above inequality as
\[
\frac{G^2}{\Delta}\|x-x_0\|^2-G\|x-x_0\|-\alpha\Delta\leq 0,
\]
which implies that
\[
\begin{aligned}
\|x-x_0\|\leq{} & \frac{G+\sqrt{G^2+4\alpha G^2}}{2G^2/\Delta}
=\gamma(\alpha)\frac{\Delta}{G},
\end{aligned}
\]
where $\gamma(\alpha)=(1+\sqrt{1+4\alpha})/2$. We can then conclude that $\mathcal{S}_{x_0}(\alpha\Delta)\subseteq x_0+\frac{\gamma(\alpha)\Delta}{G}\mathbb{B}_d$. It is not hard to see that $\gamma(\alpha)$ is increasing over $\alpha\in[1,12]$ with its maximum being $\gamma(12)<11$. Therefore, $S_{x_0}(\alpha\Delta)$ is a sublevel set of the function $F_{x_0}$ on $x_0+\frac{11\Delta}{G}\mathbb{B}_d$, and thus is a convex set.
\end{proof}

After we have found certain convex neighborhoods of $x_0$ that are guaranteed to be subsets of the sublevel set $\mathcal{S}(12\Delta)$, we can now establish some fundamental analytical properties of $F_{\nu,x_0}$.

\begin{lemma}
\label{lemma:bigF_smoothed_ver_comparison}
\begin{enumerate}
\item Let $\nu>0$ and $x\in\mathcal{D}_\nu$ satisfy $x+\nu\mathbb{B}_d\subseteq\mathcal{S}(12\Delta)$. Then
\[
F_{x_0}(x)\leq F_{\nu,x_0}(x) + m\nu^2.
\]
\item Suppose $\nu\in(0,11\Delta/G)$. Then $F_{\nu,x_0}$ is an $m$-strongly convex function on $x_0+\left(\frac{11\Delta}{G}-\nu\right)\mathbb{B}_d$.
\end{enumerate}
\end{lemma}
\begin{proof}
\begin{enumerate}
\item Since $x+\nu\mathbb{B}_d$ is a convex subset of $\mathcal{S}(12\Delta)$, by Part 3 of \cref{assumption:technical_objective}, the function $F_{x_0}$ is convex on $x+\nu\mathbb{B}_d$. As a result, from $x=\frac{1}{2}(x+\nu y)+\frac{1}{2}(x-\nu y)$, we have
\[
\begin{aligned}
F_{x_0}(x)
\leq\frac{1}{2}\left(F_{x_0}(x+\nu y)
+F_{x_0}(x-\nu y)\right).
\end{aligned}
\]
We then take the expectation with respect to $y\sim\mathcal{U}(\mathbb{B}_d)$ and get
\[
F_{x_0}(x)\leq 
\mathbb{E}_{y\sim\mathcal{U}(\mathbb{B}_d)}[F_{x_0}(x+\nu y)],
\]
where we used the fact that $-y$ and $y$ are identically distributed when $y\sim\mathcal{U}(\mathbb{B}_d)$. On the other hand
\[
\begin{aligned}
\mathbb{E}_{y\sim\mathcal{U}(\mathbb{B}_d)}[F_{x_0}(x+\nu y)]
={} & f_{\nu }(x) + m\mathbb{E}_{y\sim\mathcal{U}(\mathbb{B}_d)}\!\left[\|x+\nu y-x_0\|^2\right] \\
={} & f_{\nu}(x) + m\|x-x_0\|^2 + m\nu^2\,\mathbb{E}_{y\sim\mathcal{U}(\mathbb{B}_d)}\!\left[\|y\|^2\right] \\
\leq{} & F_{\nu,x_0}(x)+m\nu^2,
\end{aligned}
\]
where we used the elementary fact that $\mathbb{E}_{y\sim\mathcal{U}(\mathbb{B}_d)}\!\left[\|y\|^2\right]\leq 1$. The proof is now complete.

\item Evidently $x+\nu \mathbb{B}_d\subseteq x_0+\frac{11\Delta}{G}\mathbb{B}_d\subseteq\mathcal{S}(12\Delta)$ for $x\in x_0+\left(\frac{11\Delta}{G}-\nu\right)\mathbb{B}_d$. Thus $F_{\nu,x_0}$ is well-defined on $x_0+\left(\frac{11\Delta}{G}-\nu\right)\mathbb{B}_d$. We then have
\[
\begin{aligned}
F_{\nu,x_0}(x)-\frac{m}{2}\|x-x_0\|^2 ={} &  \mathbb{E}_{y\sim\mathcal{U}(\mathbb{B}_d)}[f(x+\nu y)]
+\frac{m}{2}\|x-x_0\|^2 \\
={} & \mathbb{E}_{y\sim\mathcal{U}(\mathbb{B}_d)}\!\left[f(x+\nu y)+\frac{m}{2}\|x+\nu y-x_0\|^2\right] - \frac{m\nu^2}{2}\mathbb{E}_{y\sim\mathcal{U}(\mathbb{B}_d)}\!\left[\|y\|^2\right] \\
={} & \mathbb{E}_{y\sim\mathcal{U}(\mathbb{B}_d)}\!\left[\tilde{F}(x+\nu y)\right]-\frac{m\nu^2d}{2(d+2)}
\end{aligned}
\]
for all $x\in x_0+\left(\frac{11\Delta}{G}-\nu\right)\mathbb{B}_d$, where we temporarily denote $\tilde{F}(x)= f(x)+\frac{m}{2}\|x-x_0\|^2$. On the other hand, by \cref{assumption:technical_objective} and the definition of $m$, the function $\tilde{F}$ is convex on $x_0+\frac{11\Delta}{G}\mathbb{B}_d$. Therefore the smoothed version $\mathbb{E}_{y\sim\mathcal{U}(\mathbb{B}_d)}\!\big[\tilde{F}(x+\nu y)\big]$ is also convex, and by using the above identity, we see that $F_{\nu,x_0}$ is strongly convex. \qedhere
\end{enumerate}
\end{proof}

Part 2 of \cref{lemma:bigF_smoothed_ver_comparison} then allows us to define
\begin{equation}
\label{eq:optimal_F_nu_x0_def}
x_{\nu}^{\star}\coloneq \argmin_{x\in x_0+\frac{10\Delta}{G}\mathbb{B}_d} F_{\nu,x_0}(x)
\end{equation}
for $\nu\leq\Delta/G$. The next lemma bounds the distance from $x_{\nu}^{\star}$ to $x_0$.

\begin{lemma}
\label{lemma:x_nu_star_distance_x0}
Suppose $\nu\leq G/(2m)$. Then $\|x_{\nu}^{\star}-x_0\|\leq 2\Delta/G$.
\end{lemma}
\begin{proof}
The choice of $m$ given in~\cref{eq:parameter_choice_m} implies that
\begin{equation}
\label{eq:nu_upper_bound_Delta_G}
\nu\leq \frac{G\Delta}{2G^2} \leq \frac{\Delta}{2G}.
\end{equation}
By the definition of $x_{\nu}^{\star}$, we have
\[
\begin{aligned}
F_{x_0}(x_{\nu}^{\star})\leq{} & F_{\nu,x_0}(x_{\nu}^{\star})+m\nu^2 \\
\leq{} & F_{\nu,x_0}(x_0)+m\nu^2=f_\nu(x_0)+m\nu^2 \\
\leq{} & f(x_0)+G \nu+m\nu^2 \\
\leq{} & \Delta + G\cdot\frac{\Delta}{2G} + m\cdot\frac{G}{2m}\cdot\frac{\Delta}{2G}
=\frac{7}{4}\Delta.
\end{aligned}
\]
Thus $x_{\nu}^{\star}\in\mathcal{S}_{x_0}(7\Delta/4)$. On the other hand, by Part 3 of \cref{lemma:convex_subdomain_prox}, we have
\[
\mathcal{S}_{x_0}(7\Delta/4)
\subseteq x_0+\frac{\gamma(7/4)\Delta}{G}\mathbb{B}_d
=x_0+\frac{(1+2\sqrt{2})\Delta}{2G}\mathbb{B}_d.
\]
Combining the previous results and using $(1+2\sqrt{2})/2\leq 2$ completes the proof.
\end{proof}

Henceforth, we assume that $\nu\leq G/(2m)\leq\Delta/(2G)$. We then set the safeguard set $\mathcal{C}$ in~\cref{eq:main_inner_iteration} as 
\begin{equation}
\mathcal{C}\coloneqq x_{\nu}^{\star}+\frac{\sqrt{60}\Delta}{G}\mathbb{B}_d.
\end{equation}
Evidently $\mathcal{C}$ is a compact set; \cref{lemma:x_nu_star_distance_x0} implies that $x_0\in\mathcal{C}$. Furthermore, it is straightforward to verify that $\mathcal{C}\subseteq x_0+\frac{10\Delta}{G}\mathbb{B}_d$: Indeed, for any $x\in \mathcal{C}$, we have
\[
\|x-x_0\|\leq \|x-x_\nu^\star\|+\|x_\nu^\star-x_0\| \leq\left(\sqrt{60}+2\right)\frac{\Delta}{G}\leq\frac{10\Delta}{G}.
\]
Consequently, we have 
\[
\mathcal{C}+\nu\mathbb{B}_d\subseteq\left(x_0+\frac{10\Delta}{G}\mathbb{B}_d\right)+\nu\mathbb{B}_d
\subseteq
x_0+\frac{11\Delta}{G}\subseteq
\mathcal{S}(12\Delta),
\]
which guarantees that i) we have $\mathcal{C}\subseteq\mathcal{D}_\nu$, so that $\sfG(x;z,\nu)$ and $F_{\nu,x_0}(x)$ are well-defined for $x\in\mathcal{C}$; ii) $F_{\nu,x_0}$ satisfies the properties in \cref{lemma:bigF_smoothed_ver_comparison}; and iii) $f_\nu$ and $\sfG(x;z,\nu)$ satisfy the properties listed in \cref{lemma:expectation_grad_est,lemma:properties_smoothed_func,lemma:sub-Gaussian_grad_est} with $G_\Delta$ replaced by $G=G_{12\Delta}$.

\subsection{Bounding the Failure Probability}

In this subsection, we shall show that, with high probability, all $\widehat{x}_t$ will lie in the safeguard set $\mathcal{C}$, meaning that $x_t=\widehat{x}_t$.

We first characterize the evolution of the squared distance from $\widehat{x}_t$ to $x_\nu^\star$, assuming that all the previous $\widehat{x}_1,\ldots,\widehat{x}_{t-1}$ lie in $\mathcal{C}$.
\begin{lemma}
Suppose $\nu\leq G/(2m)$, and $\eta_t m\leq 1/24$ for all $t\in[T]$. Let $t\in[T]$ be arbitrary, and suppose $\widehat{x}_1,\ldots,\widehat{x}_{t-1}\in\mathcal{C}$. Then
\begin{equation}
\label{eq:inner_loop_dist_x_nu_star_bound}
\begin{aligned}
\left\|\widehat{x}_t-x_{\nu}^{\star}\right\|^2
\leq{} & \prod_{k=1}^t\left(1 \!-\! \eta_k m\right)\|x_0-x_{\nu}^{\star}\|^2
+2\sum_{k=1}^t \!\left[\eta_k^2\prod_{j=k+1}^t\left(1 \!-\! \eta_j m\right)\right]
\left(\|g_k\|^2+8m^2\|x_0-x_\nu^\star\|^2\right) \\
& - 2\sum_{k=1}^t\!\left[
\eta_k\prod_{j=k+1}^t\left(1 \!-\! \eta_j m\right)
\right]\langle x_{k-1}-x_{\nu}^{\star},g_k-\nabla f_\nu(x_{k-1})\rangle.
\end{aligned}
\end{equation}
\end{lemma}
\begin{proof}
Let $k\in[t]$ be arbitrary. By Part 2 of \cref{lemma:bigF_smoothed_ver_comparison}, $F_{\nu,x_0}$ is $m$-strongly convex on $\mathcal{C}$. Therefore
\begin{equation}
\label{eq:inner_loop_strong_convexity_1}
F_{\nu,x_0}(x_{\nu}^{\star})\geq F_{\nu,x_0}(x_{k-1})+\langle \nabla F_{\nu,x_0}(x_{k-1}),x_{\nu}^{\star}-x_{k-1}\rangle+\frac{m}{2}\|x_{\nu}^{\star}-x_{k-1}\|^2.
\end{equation}
On the other hand, \cref{lemma:x_nu_star_distance_x0} implies that $x_{\nu}^{\star}$ is in the interior of $\mathcal{C}$. Therefore
\begin{equation}
\label{eq:inner_loop_strong_convexity_2}
F_{\nu,x_0}(x_{k-1})-F_{\nu,x_0}(x_{\nu}^{\star})\geq\frac{m}{2}\|x_{k-1}-x_{\nu}^{\star}\|^2.
\end{equation}
Now we use the iteration $\widehat{x}_{k}=x_{k-1}-\eta_k(g_k+2m(x_{k-1}-x_0))$ and obtain
\begin{align*}
\|\widehat{x}_k-x_{\nu}^{\star}\|^2
={} &
\|{x}_{k-1}-x_{\nu}^{\star}\|^2
+\eta_k^2 \|g_k+2m(x_{k-1}-x_0)\|^2 \\
& -2\eta_k\langle x_{k-1}-x_{\nu}^{\star},\nabla F_{\nu,x_0}(x_{k-1})\rangle
-2\eta_k\langle x_{k-1}-x_{\nu}^{\star},g_k-\nabla f_\nu(x_{k-1})\rangle \\
\leq{} & \|{x}_{k-1}-x_{\nu}^{\star}\|^2
+\eta_k^2 \|g_k+2m(x_{k-1}-x_0)\|^2 \\
& +2\eta_k(F_{\nu,x_0}(x_{\nu}^{\star})-F_{\nu,x_0}(x_{k-1}))-\eta_k m\|x_{\nu}^{\star}-x_{k-1}\|^2 \\
& -2\eta_k\langle x_{k-1}-x_{\nu}^{\star},g_k-\nabla f_\nu(x_{k-1})\rangle \\
\leq{} & \|x_{k-1}-x_{\nu}^{\star}\|^2+\eta_k^2 \|g_k+2m(x_{k-1}-x_0)\|^2 \\
& -2\eta_k m\|x_{k-1}-x_{\nu}^{\star}\|^2
-2\eta_k\langle x_{k-1}-x_{\nu}^{\star},g_k-\nabla f_\nu(x_{k-1})\rangle \\
={} & (1-2\eta_k m)\|x_{k-1}-x_{\nu}^{\star}\|^2
+\eta_k^2 \|g_k+2m(x_{k-1}-x_0)\|^2 \\
& - 2\eta_k\langle x_{k-1}-x_{\nu}^{\star},g_k-\nabla f_\nu(x_{k-1})\rangle,
\end{align*}
where we used~\cref{eq:inner_loop_strong_convexity_1} in the first inequality and \cref{eq:inner_loop_strong_convexity_2} in the second inequality.
Finally, we note that
\[
\begin{aligned}
\|g_k+2m(x_{k-1}-x_0)\|^2
\leq{} & 2\|g_k\|^2+8m^2\|x_{k-1}-x_\nu^\star-(x_0-x_\nu^\star)\|^2\\
\leq{} & 2\|g_k\|^2 + 16m^2\|x_0-x_\nu^\star\|^2+16m^2\|x_{k-1}-x_\nu^\star\|^2.
\end{aligned}
\]
Thus
\[
\begin{aligned}
\|\widehat{x}_k-x_\nu^\star\|^2
\leq{} & (1-2\eta_k m + 16\eta_k^2m^2)\|x_{k-1}-x_\nu^\star\|^2
+2\eta_k^2(\|g_k\|^2 + 8m^2\|x_0-x_\nu^\star\|^2) \\
& - 2\eta_k\langle x_{k-1}-x_{\nu}^{\star},g_k-\nabla f_\nu(x_{k-1})\rangle \\
\leq{} & \left(1-\frac{4}{3}\eta_k m\right)\|x_{k-1}-x_\nu^\star\|^2
+2\eta_k^2(\|g_k\|^2 + 8m^2\|x_0-x_\nu^\star\|^2) \\
& - 2\eta_k\langle x_{k-1}-x_{\nu}^{\star},g_k-\nabla f_\nu(x_{k-1})\rangle,
\end{aligned}
\]
where we used $\eta_k m\leq 1/24$ to obtain $1-2\eta_km+16\eta_k^2m^2\leq 1-4\eta_k m/3$. The bound~\cref{eq:inner_loop_dist_x_nu_star_bound} now follows by induction and noting that $\widehat{x}_{k-1}=x_{k-1}$ for $k\in[t]$.
\end{proof}

\begin{corollary}
Suppose $\nu\leq G/(2m)$. Let $t\in[T]$ be arbitrary, and suppose
\[
\eta_k=\frac{3}{m(k+t_0)},\quad\forall k\in[t]
\]
for some $t_0\geq 71$. For each positive integer $i$, define the polynomial
\begin{equation}
\label{eq:definition_q_polynomials}
q_i(z;t_0) \coloneq \prod_{j=0}^{i-1} (z+t_0-j)
=(z+t_0)(z+t_0-1)\cdots(z+t_0-i+1).
\end{equation}
Then, assuming that $\widehat{x}_1,\ldots,\widehat{x}_{t-1}\in\mathcal{C}$, we have
\begin{equation}
\label{eq:inner_loop_dist_x_nu_star_bound2}
\begin{aligned}
q_4(t;t_0)\|\widehat{x}_t-x_{\nu}^{\star}\|^2
\leq{} & [q_4(0;t_0)+
48t(t\!+\!t_0\!-\!1)(t\!+\!3t_0\!-\!3)
]\|x_0-x_{\nu}^{\star}\|^2 \\
&
+\frac{18}{m^2}\sum_{k=1}^t q_2(k-2;t_0)\|g_k\|^2 \\
& -\frac{6}{m}\sum_{k=1}^t q_3(k-1;t_0)\langle x_{k-1}-x_{\nu}^{\star},g_k-\nabla f_{\nu}(x_{k-1})\rangle.
\end{aligned}
\end{equation}
\end{corollary}
\begin{proof}
Let $k\in[t]$ be arbitrary. By the chosen step sizes, we have
\[
\begin{aligned}
\prod_{j=k+1}^t\left(1-\eta_j m\right)
={} & \prod_{j=k+1}^t
\frac{j+t_0-4}{j+t_0}
=\frac{q_4(k;t_0)}{q_4(t;t_0)}.
\end{aligned}
\]
Therefore
\[
2\eta_k^2\prod_{j=k+1}^t(1-\eta_j m)
=
\frac{18q_3(k-1;t_0)}{m^2(k+t_0)q_4(t;t_0)}
\leq \frac{18q_2(k-2;t_0)}{m^2q_4(t;t_0)},
\]
and
\[
2\eta_k\prod_{j=k+1}^t(1-\eta_j m)
=\frac{6q_3(k-1;t_0)}{mq_4(t;t_0)}.
\]
We then note that
\[
\begin{aligned}
& 2\sum_{k=1}^t \!\left[\eta_k^2\prod_{j=k+1}^t\left(1 \!-\! \eta_j m\right)\right]
\cdot 8m^2\|x_0-x_\nu^\star\|^2 \\
\leq{} & \frac{144}{q_4(t;t_0)}
\sum_{k=1}^t q_2(k-2;t_0)
\cdot \|x_0-x_\nu^\star\|^2 \\
\leq{} & \frac{48t(t+t_0-1)(t+3t_0-3)}{q_4(t;t_0)}\cdot \|x_0-x_\nu^\star\|^2,
\end{aligned}
\]
where in the last step we used
\[
\begin{aligned}
\sum_{k=1}^t q_2(k-2;t_0)
\leq{} &
\sum_{k=1}^t (k+t_0-2)^2
\leq 
\int_1^{t+1}(x+t_0-2)\,\ud x \\
={} &
\frac{1}{3}[(t+t_0-1)^3-(t_0-1)^3]
\leq \frac{t}{3}(t+t_0-1)(t+3t_0-3).
\end{aligned}
\]
Combining these results with~\cref{eq:inner_loop_dist_x_nu_star_bound} completes the proof.
\end{proof}

We now define the following events:
\begin{align}
E_1 \coloneq{} & 
\bigcap_{t=1}^T
\left\{
-\frac{6}{m}\sum_{k=1}^t
q_3(k-1;t_0)
\langle x_{k-1}-x_{\nu}^{\star},g_k-\nabla f_{\nu}(x_{k-1})\rangle\leq\frac{48\Delta^2}{G^2}
q_4(t;t_0)\right\},
\\
E_2 \coloneq{} &
\bigcap_{t=1}^T
\left\{
\frac{18}{m^2}\sum_{k=1}^t q_2(k-2;t_0) \|g_k\|^2
\leq \frac{4\Delta^2}{G^2}q_4(t;t_0)
\right\}.
\end{align}

\begin{lemma}
\label{lemma:E1_cap_E2_implies_bounded_x}
Suppose $\nu\leq G/(2m)$ and $\eta_t=\frac{3}{m(t+t_0)}$ for all $t\in[T]$ for some $t_0\geq 71$. Then
\[
E_1\cap E_2\subseteq 
\bigcap_{t=1}^T\left\{x_t=\widehat{x}_t\in\mathcal{C}\right\}.
\]
\end{lemma}
\begin{proof}
Let us temporarily denote
\[
h(t,t_0)=(t+t_0)(t+t_0-2)(t+t_0-3)-96t(t+3t_0-3).
\]
Then we have
\[
\frac{\partial h(t,t_0)}{\partial t_0}
=6+3t^2+3t_0^2-10t_0+t(6t_0-298),
\]
which is greater than $0$ if $t_0\geq 71$ and $t\geq 1$. Hence
\[
h(t,t_0)\geq h(t,71)
=t^3+112t^2-5741t+333132,
\]
and it is elementary to check that $h(t,71)\geq 0$ for all $t\geq 1$. We see that $h(t;t_0)\geq 0$ whenever $t_0\geq 71$ and $t\geq 1$. Consequently, $96t(t+3t_0-3)\leq (t+t_0)(t+t_0-2)(t+t_0-3)$, which leads to
\begin{equation}
\label{eq:E1_cap_E2_implies_bounded_x_temp1}
\begin{aligned}
\frac{q_4(t;t_0)+48t(t+t_0-1)(t+3t_0-3)}{q_4(t;t_0)}
\leq 1+\frac{48t(t+3t_0-3)}{(t+t_0)(t+t_0-2)(t+t_0-3)}
\leq \frac{3}{2}
\end{aligned}
\end{equation}
for all $t\geq 1$.

We now prove $E_1\cap E_2\subseteq\bigcap_{t=1}^T \{x_t=\widehat{x}_t\in\mathcal{C}\}$ by induction. By \cref{lemma:x_nu_star_distance_x0}, we know that $\|x_0-x_{\nu}^{\star}\|\leq 2\Delta/G$. Then \cref{eq:inner_loop_dist_x_nu_star_bound2} and~\cref{eq:E1_cap_E2_implies_bounded_x_temp1} lead to
\[
\begin{aligned}
q_4(1;t_0) \|\widehat{x}_1-x_{\nu}^{\star}\|^2
\leq{} & \frac{3}{2}q_4(1;t_0)\left(\frac{2\Delta}{G}\right)^{\!2}
+\frac{18}{m^2}q_2(-1;t_0)\|g_1\|^2 \\
&
-\frac{6}{m}q_3(0;t_0)\langle x_0-x_{\nu}^{\star},g_1-\nabla f_\nu(x_0)\rangle.
\end{aligned}
\]
On the other hand, by definition, on $E_1\cap E_2$ we have
\begin{align*}
-\frac{6}{m}q_3(0;t_0)\langle x_0-x_{\nu}^{\star},g_1-\nabla f_\nu(x_0)\rangle
\leq{} &
\frac{48\Delta^2}{G^2}q_4(1;t_0), 
\\
\frac{18}{m^2}q_2(-1;t_0)\|g_1\|^2\leq{} &
\frac{4\Delta^2}{G^2}q_4(1;t_0).
\end{align*}
Thus on $E_1\cap E_2$ we have
\[
\begin{aligned}
q_4(1;t_0) \|\widehat{x}_1-x_{\nu}^{\star}\|^2
\leq{} & \frac{3}{2}q_4(1;t_0)\left(\frac{2\Delta}{G}\right)^2
+\frac{48\Delta^2}{G^2}q_4(1;t_0)
+\frac{4\Delta^2}{G^2}q_4(1;t_0)
\leq
\frac{60\Delta^2}{G^2}q_4(1;t_0)
\end{aligned}
\]
We see that $E_1\cap E_2\subseteq \{x_1=\widehat{x}_1\in\mathcal{C}\}$, and the base case is proved.

Now suppose $E_1\cap E_2\subseteq \bigcap_{k=1}^{t-1}\{x_k=\widehat{x}_k\in\mathcal{C}\}$ for some $t\in\{2,\ldots,T\}$, and suppose $E_1\cap E_2$ holds. Then we have
\[
-\frac{6}{m}\sum_{k=1}^t q_3(k-1;t_0)\langle x_{k-1}-x_{\nu}^{\star},g_k-\nabla f_{\nu}(x_{k-1})\rangle\leq\frac{48\Delta^2}{G^2}
q_4(t;t_0),
\]
and
\[
\frac{18}{m^2}\sum_{k=1}^t 
q_2(k-2;t_0)\|g_k\|^2
\leq \frac{4\Delta^2}{G^2}q_4(t;t_0).
\]
Since
\[
E_1\cap E_2\subseteq \bigcap_{k=1}^{t-1}\{x_k=\widehat{x}_k\in\mathcal{C}\}=\{\widehat{x}_1,\ldots,\widehat{x}_{k-1}\text{ are all in }\mathcal{C}\},
\]
we can now apply~\cref{eq:inner_loop_dist_x_nu_star_bound2} together with the above bounds to get
\[
\begin{aligned}
q_4(t;t_0)\|\widehat{x}_t-x_{\nu}^{\star}\|^2
\leq{} & \frac{3}{2}q_4(t;t_0)\left(\frac{2\Delta}{G}\right)^{\!2}
+\frac{48\Delta^2}{G^2}
q_4(t;t_0)+\frac{4\Delta^2}{G^2}
q_4(t;t_0)
\leq
\frac{60\Delta^2}{G^2}q_4(t;t_0),
\end{aligned}
\]
showing that $\widehat{x}_t\in\mathcal{C}$. Therefore $E_1\cap E_2\subseteq\bigcap_{k=1}^t \{x_k=\widehat{x}_k\in\mathcal{C}\}$. By induction we see that $E_1\cap E_2\subseteq\bigcap_{t=1}^T \{x_t=\widehat{x}_t\in\mathcal{C}\}$. The proof is now complete.
\end{proof}

Now we try to find the value of $t_0$, so that the events $E_1$ and $E_2$ will have high probabilities. The following lemma bounds the probability of $E_2$.

\begin{lemma}
\label{lemma:probability_E2}
Let $\delta\in(0,1/5)$ be arbitrary. Suppose $\nu\leq G/(2m)$, and for each $t\in[T]$,
\[
\eta_t=\frac{3}{m(t+t_0)},\qquad
\text{where}\quad
t_0\geq 6d\ln\frac{4}{\delta}.
\]
Then $\mathbb{P}(E_2)\geq 1-\delta$.
\end{lemma}
\begin{proof}
For each $t\in[T]$, let
\begin{align*}
h_t ={} & (t+t_0-2)(t+t_0-3)\|g_t\|^2, \\
v_t={} & 4(t+t_0-2)(t+t_0-3)G^2d, \\
V_t={} &
\sum_{k=1}^t v_k=\frac{4G^2d}{3}t
\left[(t+t_0-2)(t+2t_0-4)+(t_0-1)(t_0-3)\right].
\end{align*}
We then have $h_t\geq 0$ and
\[
\begin{aligned}
\mathbb{E}\Econd*{\exp\!\left(\frac{h_t}{v_t}\right)}{\mathcal{F}_{t-1}}
={} & \mathbb{E}\Econd*{\exp\!\left(\frac{\|g_t\|^2}{4G^2d}\right)}{\mathcal{F}_{t-1}}
\leq 2,
\end{aligned}
\]
where the last inequality follows from \cref{lemma:sub-Gaussian_grad_est}. Next, we note that
\[
\begin{aligned}
\frac{V_t}{v_t}-\frac{t}{3}
={} &
\frac{(t_0-1)(1+(2t_0-5)/t)}{3(1+(t_0-2)/t)(1+(t_0-3)/t)}.
\end{aligned}
\]
Since the function $x\mapsto \frac{1+(a+b)x}{(1+ax)(1+bx)}$ is decreasing over $x\in(0,+\infty)$ when $a>0$ and $b>0$, we see that
\[
\frac{(t_0-1)(1+(2t_0-5)/t)}{3(1+(t_0-2)/t)(1+(t_0-3)/t)}
\geq \frac{(t_0-1)(1+(2t_0-5)/1)}{3(1+(t_0-2)/1)(1+(t_0-3)/1)}=\frac{2}{3},
\]
and therefore
\[
\frac{V_t}{v_t}\geq\frac{t+2}{3}
\]
for all $t\in[T]$. We can now apply \cref{lemma:time-uniform_poly_growth_v} and obtain
\[
\mathbb{P}\!\left(
\exists t\in[T]\text{ s.t. }
\sum_{k=1}^t h_k> \left(\ln\frac{4}{\delta}\right)V_t \right)
\leq \frac{\exp\!\left(-\left(\ln\frac{4}{\delta}-\ln 2\right)\right)}
{1-\exp\!\left(-\frac{1}{3}\left(\ln\frac{4}{\delta}-\ln 2\right)\right)}
\leq\delta
\]
for any $\delta\in(0,1/5)$. Then, since
\[
\begin{aligned}
q_3(t-1;t_0)
-t\left[(t+t_0-2)(t+2t_0-4)+(t_0-1)(t_0-3)\right]
=t_0(t_0-3)^2+2t_0-6
\geq 0,
\end{aligned}
\]
we have
\[
\begin{aligned}
\frac{V_t}{q_4(t;t_0)}
={} & \frac{4G^2d\cdot t\left[(t+t_0-2)(t+2t_0-4)+(t_0-1)(t_0-3)\right]}{3q_4(t;t_0)} \\
\leq{} &
\frac{4G^2d\cdot q_3(t-1;t_0)}{3q_4(t;t_0)}
=\frac{4G^2d}{3(t+t_0)}
\leq\frac{4G^2d}{3t_0}.
\end{aligned}
\]
As a result, by $t_0\geq 6d\ln(4/\delta)$, we get
\[
\left(\ln\frac{4}{\delta}\right)V_t
\leq \left(\ln\frac{4}{\delta}\right)\frac{4G^2d}{18d\ln(4/\delta)}\cdot q_4(t;t_0)
\leq \frac{2G^2}{9}q_4(t;t_0)\leq \frac{2m^2\Delta^2}{9G^2}q_4(t;t_0),
\]
where we used $G\leq m\Delta/G$, a consequence of the definition of $m$. We can now conclude that
\[
\begin{aligned}
\mathbb{P}(E_2^c) ={} &
\mathbb{P}\!\left(
\exists t\in[T]\text{ s.t. }
\sum_{k=1}^t h_k> \frac{2m^2\Delta^2}{9G^2}q_4(t;t_0) \right) \\
\leq{} & \mathbb{P}\!\left(
\exists t\in[T]\text{ s.t. }
\sum_{k=1}^t h_k> \left(\ln\frac{4}{\delta}\right)V_t \right)
\leq\delta,
\end{aligned}
\]
which finishes the proof.
\end{proof}

To bound the probability of $E_1$, we need an auxiliary elementary lemma.

\begin{lemma}
\label{lemma:elementary_log_bound}
Given positive constants $A,B$ satisfying $AB\geq 1$, let
\[
h(x) = A\ln(1+B\ln x)-x.
\]
Then for all $x\geq 1$, we have
\[
h(x)\leq A\ln\!\left(1+B\ln(AB)\right)-1.
\]
\end{lemma}
\begin{proof}
We have
\[
h'(x) = \frac{AB}{x(1+B\ln x)}-1.
\]
It is straightforward to check that for $x\geq AB\geq 1$, we have $h'(x)\leq 0$. Therefore
\begin{align*}
\sup_{x\geq 1} h(x)
={} &\sup_{x\in[1,AB]} h(x) \\
\leq{} & \sup_{x\in[1,AB]}A\ln(1+B\ln x)-\inf_{x\in[1,AB]}x \\
\leq{} &
A\ln(1+B\ln (AB))-1.
\qedhere
\end{align*}
\end{proof}

The next lemma bounds the probability of $E_1$.
\begin{lemma}
\label{lemma:probability_E1}
Let $\delta\in(0,1)$ be arbitrary. Suppose $\nu\leq G/(2m)$, and for each $t\in[T]$,
\[
\eta_t=\frac{3}{m(t+t_0)},\qquad
\text{where}\quad
t_0\geq 50d\ln(1+7\ln(700d))
+25d\ln\frac{\pi^2}{6\delta}.
\]
Then $\mathbb{P}(E_1)\geq 1-\delta$.
\end{lemma}
\begin{proof}
For each $t\in[T]$, let
\[
\theta_{t} = -\frac{q_3(t-1;t_0)\langle x_{t-1}-x_{\nu}^{\star},g_t-\nabla f_\nu(x_{t-1})\rangle}{12\sqrt{10d}\Delta}.
\]
We then have
\[
\begin{aligned}
& \mathbb{E}\Econd*{\exp\!\left(
\frac{3\theta_{t}^2}{(q_3(t-1;t_0))^2}
\right)}{\mathcal{F}_{t-1}} \\
\leq{} &
\mathbb{E}\Econd*{\exp\!\left(
\frac{\left\| x_{t-1}-x_{\nu}^{\star}\right\|^2\left\|g_t-\nabla f_\nu(x_{t-1})\right\|^2}{480\Delta^2 d}
\right)}{\mathcal{F}_{t-1}} \\
\leq{} &
\mathbb{E}\Econd*{\exp\!\left(
\frac{60\Delta^2/G^2\left\|g_t-\nabla f_\nu(x_{t-1})\right\|^2}{480\Delta^2 d}
\right)}{\mathcal{F}_{t-1}} \\
={} & \mathbb{E}\Econd*{\exp\!\left(
\frac{\left\|g_t-\nabla f_\nu(x_{t-1})\right\|^2}{8G^2 d}
\right)}{\mathcal{F}_{t-1}}\leq 2,
\end{aligned}
\]
where in the second step we used $\left\| x_{t-1}-x_{\nu}^{\star}\right\|^2\leq 60\Delta^2/G^2$ since $x_{t-1}\in\mathcal{C}$, and the last step follows from \cref{lemma:sub-Gaussian_grad_est}. \Cref{lemma:sub-Gaussian_zero_mean_MGF} then implies
\[
\mathbb{E}\Econd*{\exp(\lambda\theta_{t})}{\mathcal{F}_{t-1}}
\leq \exp\!\left(\frac{\lambda^2}{2} (q_3(t-1;t_0))^2\right),
\quad\forall \lambda\in\mathbb{R}.
\]
Denote $V_t=\sum_{k=1}^t (q_3(k-1;t_0))^2$.
We can now apply \cref{theorem:stitched_boundary} and get
\[
\begin{aligned}
& \mathbb{P}\!\left(
\exists t\in[T]\text{ s.t. }
\sum_{k=1}^t\theta_{k}\geq 
u_{\delta,t}
\right)\leq \delta,
\end{aligned}
\]
where
\[
u_{\delta,t} = \frac{e^{1/4}+e^{-1/4}}{\sqrt{2}}
\sqrt{V_t\left[2\ln\!\left(1+\ln\frac{V_t}{(q_3(0;t_0))^2}\right)+
\ln\frac{\pi^2}{6\delta}\right]}.
\]
Then, noting that $\left(\frac{e^{1/4}+e^{-1/4}}{\sqrt{2}}\right)^{\!2}\leq 20/9$ and $V_t\leq t(q_3(t-1;t_0))^2$, we get
\[
\begin{aligned}
u_{\delta,t}^2
\leq{} & 
\frac{20}{9}t(q_3(t-1;t_0))^2
\left[
2\ln\!\left(1+\ln\frac{t(q_3(t-1;t_0))^2}{(q_3(0;t_0))^2}\right)+\ln\frac{\pi^2}{6\delta}
\right].
\end{aligned}
\]
Then, since
\[
\begin{aligned}
\frac{t(q_3(t-1;t_0))^2}{(q_3(0;t_0))^2}
={} & t^7\frac{(1+(t_0-1)/t)^2(1+(t_0-2)/t)^2(1+(t_0-3)/t)^2}{t_0^2(t_0-1)^2(t_0-2)^2}
\leq
t^7
\end{aligned}
\]
we see that
\[
\begin{aligned}
u_{\delta,t}^2
\leq{} & 
\frac{20}{9}t(q_3(t-1;t_0))^2
\left[
2\ln\!\left(1+7\ln t\right)+\ln\frac{\pi^2}{6\delta}
\right] \\
={} & \frac{20(q_4(t;t_0))^2}{9}\cdot\frac{t}{(t+t_0)^2}\left[
2\ln\!\left(1+7\ln t\right)+\ln\frac{\pi^2}{6\delta}
\right] \\
\leq{} &
\frac{20(q_4(t;t_0))^2}{9}\cdot\frac{1}{t+2t_0}\left[
2\ln\!\left(1+7\ln t\right)+\ln\frac{\pi^2}{6\delta}
\right].
\end{aligned}
\]
Now, by the choice of $t_0$ and \cref{lemma:elementary_log_bound}, we see that
\[
2t_0\geq 
100d\ln(1+7\ln t)-t+50d\ln\frac{\pi^2}{6d},\qquad\forall t\geq 1,
\]
which implies
\[
\frac{1}{t+2t_0}\left[
2\ln\!\left(1+7\ln t\right)+\ln\frac{\pi^2}{6\delta}
\right]\leq \frac{1}{50d},
\]
and consequently,
\[
u_{\delta,t}^2\leq \frac{2(q_4(t;t0))^2}{45d}
\leq \frac{2m^2\Delta^2}{45G^4d} (q_4(t;t0))^2.
\]
Now we can conclude that
\[
\begin{aligned}
\mathbb{P}(E_1^c)
={} & \mathbb{P}\!\left(
\exists t\in[T]\text{ s.t. }
\sum_{k=1}^t\theta_k> \frac{2m\Delta}{3G^2\sqrt{10d}}q_4(t;t_0)\right) \\
\leq{} & \mathbb{P}\!\left(
\exists t\in[T]\text{ s.t. }
\sum_{k=1}^t\theta_k> u_{\delta,t}
\right)\leq\delta,
\end{aligned}
\]
which completes the proof.
\end{proof}

We can now summarize the previous results and establish that, by choosing $t_0$ appropriately, we have $\widehat{x}_t=x_t\in\mathcal{C}$ for all $t=1,\ldots,T$ with high probability.

\begin{theorem}
Let $\delta\in(0,1/5)$ be arbitrary. Suppose $\nu\leq G/(2m)$, and for each $t\in[T]$,
\[
\eta_t=\frac{3}{m(t+t_0)},\qquad
\text{where}\quad
t_0\geq
50d\ln(1+7\ln(700d))
+25d\ln\frac{\pi^2}{3\delta}
.
\]
Then $\widehat{x}_t=x_t\in\mathcal{C}$ for all $t=1,\ldots,T$ with probability at least $1-\delta$.
\end{theorem}
\begin{proof}
Note that when $d\geq 3$ and $\delta\in(0,1/5)$, we have
\[
50d\ln(1+7\ln(700d))
+25d\ln\frac{\pi^2}{3\delta}
\geq\max\left\{71,6d\ln\frac{10}{\delta}\right\}.
\]
Therefore we can apply \cref{lemma:probability_E2,lemma:probability_E1} to get
\[
\mathbb{P}(E_1\cap E_2)\geq 1-\left(\frac{\delta}{2}+\frac{\delta}{2}\right) = 1-\delta.
\]
Then, \cref{lemma:E1_cap_E2_implies_bounded_x} shows that, on $E_1\cap E_2$ we have $x_t=\widehat{x}_t\in\mathcal{C}$.
\end{proof}

\subsection{Convergence to the Optimal Point}

Now we proceed to the analysis of the convergence rate.

\begin{lemma}
\label{lemma:convergence_raw_gap_bound}
Suppose $\nu\leq G/(2m)$ and $\eta_t=\frac{3}{m(t+t_0)}$ for each $t\in[T]$ for some $t_0\geq 143$. Let
\[
\overline{x}_T=\sum_{t=1}^T \frac{2(t+t_0-1)}{T(T+2t_0-1)}x_{t-1}.
\]
Then, as long as $x_t\in\mathcal{C}$ for all $t\in[T]$, we have
\begin{equation}
\label{eq:convergence_raw_gap_bound}
\begin{aligned}
F_{\nu,x_0}(\overline{x}_T)-F_{\nu,x_0}(x^\star_\nu)
\leq{} & \frac{2}{T(T+2t_0-1)}\sum_{t=1}^T(t+t_0-1)\langle x_{t-1}-x_\nu^\star, \nabla f_\nu(x_{t-1})-g_t\rangle \\
& + \frac{6}{mT(T+2t_0-1)}\sum_{t=1}^T\|g_t\|^2
+ \frac{m(48T+t_0(t_0-1)/3)}{T(T+2t_0-1)}\|x_0-x_\nu^\star\|^2.
\end{aligned}
\end{equation}
\end{lemma}
\begin{proof}
Strong convexity of $F_{\nu,x_0}$ over $\mathcal{C}$ gives
\begin{equation}
\label{eq:convergence_raw_gap_bound_temp}
\begin{aligned}
F_{\nu,x_0}(x_{t-1})-F_{\nu,x_0}(x_\nu^\star)
\leq{} & \langle x_{t-1}-x_\nu^\star, \nabla f_\nu(x_{t-1})+2m(x_{t-1}-x_0)\rangle - \frac{m}{2}\|x_{t-1}-x_\nu^\star\|^2 \\
={} & \langle x_{t-1}-x_\nu^\star, g_{t}+2m(x_{t-1}-x_0)\rangle - \frac{m}{2}\|x_{t-1}-x_\nu^\star\|^2 \\
& + \langle x_{t-1}-x_\nu^\star, \nabla f_\nu(x_{t-1})-g_{t}\rangle.
\end{aligned}
\end{equation}
Since $\widehat{x}_t=x_t$ for all $t\in[T]$, the first term on the right-hand side can be written as
\[
\begin{aligned}
& \langle x_{t-1}-x_\nu^\star, g_{t}+2m(x_{t-1}-x_0)\rangle \\
={} &
\frac{1}{\eta_t}\langle x_{t-1}-x_\nu^\star,x_t-x_{t-1}\rangle \\
={} & \frac{1}{2\eta_t}\left(
\|x_{t-1}-x_t\|^2 + \|x_{t-1}-x_\nu^\star\|^2
-\|x_{t}-x_\nu^\star\|^2
\right) \\
={} & \frac{1}{2\eta_t}\left(
\|x_{t-1}-x_\nu^\star\|^2
-\|x_{t}-x_\nu^\star\|^2
\right)
+\frac{\eta_t}{2}\|g_t+2m(x_{t-1}-x_0)\|^2.
\end{aligned}
\]
Then, noting that
\[
\begin{aligned}
\|g_t+2m(x_{t-1}-x_0)\|^2
\leq{} & 2\|g_t\|^2+8m^2\|x_{t-1}-x_\nu^\star-(x_0-x_\nu^\star)\|^2 \\
\leq{} & 2\|g_t\|^2 + 16m^2 \|x_{t-1}-x_\nu^\star\|^2
+16m^2\|x_0-x_\nu^\star\|^2,
\end{aligned}
\]
we get
\[
\begin{aligned}
& \langle x_{t-1}-x_\nu^\star, g_{t}+2m(x_{t-1}-x_0)\rangle-\frac{m}{2}\|x_{t-1}-x_\nu^\star\|^2 \\
\leq{} & 
\frac{1}{2\eta_t}\left(
\|x_{t-1}-x_\nu^\star\|^2
-\|x_{t}-x_\nu^\star\|^2
\right) + \eta_t \|g_t\|^2 \\
&
-m\left(\frac{1}{2}-8\eta_t m\right)\|x_{t-1}-x_\nu^\star\|^2
+8\eta_tm^2\|x_0-x_\nu^\star\|^2 \\
\leq{} & 
\frac{1}{2\eta_t}\left(
\|x_{t-1}-x_\nu^\star\|^2
-\|x_{t}-x_\nu^\star\|^2
\right) + \eta_t \|g_t\|^2
-\frac{m}{3}\|x_{t-1}-x_\nu^\star\|^2
+8\eta_t m^2 \|x_0-x_\nu^\star\|^2,
\end{aligned}
\]
where we used $\eta_t m\leq 1/48$ that follows from $t_0\geq 143$. Plugging in $\eta_t=3/(m(t+t_0))$, we get
\[
\begin{aligned}
& \langle x_{t-1}-x_\nu^\star, g_{t}+2m(x_{t-1}-x_0)\rangle-\frac{m}{2}\|x_{t-1}-x_\nu^\star\|^2 \\
\leq{} & \frac{m}{6}\left[
(t+t_0-2)\|x_{t-1}-x_\nu^\star\|^2
-(t+t_0)\|x_{t}-x_\nu^\star\|^2
\right]
+\frac{3}{m(t+t_0)}\|g_t\|^2
+\frac{24m}{t+t_0}\|x_0-x_\nu^\star\|^2
\end{aligned}
\]
Plugging it back into~\cref{eq:convergence_raw_gap_bound_temp} and multiplying both sides by $t+t_0-1$, we get
\[
\begin{aligned}
& (t+t_0-1)\left(F_{\nu,x_0}(x_{t-1})-F_{\nu,x_0}(x_\nu^\star)\right) \\
\leq{} & \frac{m}{6}\left[(t+t_0-2)(t+t_0-1)\|x_{t-1}-x_\nu^\star\|^2-(t+t_0-1)(t+t_0)\|x_t-x_\nu^\star\|^2\right] \\
& + \frac{3}{m}\|g_t\|^2+24m\|x_0-x_\nu^\star\|^2
+(t+t_0-1)\langle x_{t-1}-x_\nu^\star, \nabla f_\nu(x_{t-1})-g_{t}\rangle.
\end{aligned}
\]
Taking the telescoping sum from $t=1$ to $t=T$ and discarding the negative of squares leads to
\[
\begin{aligned}
\sum_{t=1}^T (t \!+\! t_0 \!-\! 1)\!\left(F_{\nu,x_0}(x_{t-1}) - F_{\nu,x_0}(x_\nu^\star)\right)
\leq{} & \sum_{t=1}^T(t \!+\! t_0 \!-\! 1) \langle x_{t-1}-x_\nu^\star, \nabla f_\nu(x_{t-1})-g_{t}\rangle \\
& + \frac{3}{m}\sum_{t=1}^T \|g_t\|^2
+ \left[24T+\frac{1}{6}t_0(t_0-1)\right]m\|x_0-x_\nu^\star\|^2.
\end{aligned}
\]
Finally, by dividing both sides by $\sum_{t=1}^T(t+t_0-1)=T(T+2t_0-1)/2$, and noting that Jensen's inequality gives
\[
F_{\nu,x_0}(\overline{x}_T)\leq \frac{2}{T(T+2t_0-1)}\sum_{t=1}^T (t+t_0-1)F_{\nu,x_0}(x_{t-1}),
\]
we obtain the desired bound.
\end{proof}

With \cref{lemma:convergence_raw_gap_bound} at hand, the next steps would be to bound each term on the right-hand side of~\cref{eq:convergence_raw_gap_bound} with high probability. We shall see that bounding $\sum_{t=1}^T \|g_t\|^2$ with high probability follows a standard approach. Denote
\begin{equation}
E_3(\delta) \coloneq \left\{\sum_{t=1}^T\|g_t\|^2\leq 4G^2d\left(T\ln 2+\ln\frac{1}{\delta}\right)\right\}.
\end{equation}

\begin{lemma}
$\mathbb{P}(E_3(\delta))\geq 1-\delta$ for any $\delta\in(0,1)$.
\end{lemma}
\begin{proof}
For each $t\in[T]$, let $L_0=1$ and
\[
L_t = \exp\left(\frac{\sum_{k=1}^t\|g_t\|^2}{4G^2d}-t\ln 2\right).
\]
Then by \cref{lemma:sub-Gaussian_grad_est}, we have
\[
\mathbb{E}\Econd*{\frac{L_t}{L_{t-1}}}{\mathcal{F}_{t-1}}
=\mathbb{E}\Econd*{\exp\!\left(\frac{\|g_t\|^2}{4G^2d}-
\ln 2\right)}{\mathcal{F}_{t-1}}\leq 1,
\qquad\forall t\in[T],
\]
showing that $(L_t)_{t=0}^T$ is a non-negative supermartingale. By Ville's inequality (\cref{lemma:Ville_inequality}), we get
\[
\begin{aligned}
\mathbb{P}(E_3(\delta)^c)
\leq{} & \mathbb{P}\!\left(
\exists t\in[T]\text{ s.t. }
\sum_{k=1}^t \|g_k\|^2
> 4G^2d\left(t\ln 2+\ln\frac{1}{\delta}\right)
\right) \\
={} &
\mathbb{P}\!\left(
\exists t\in[T]\text{ s.t. }
L_t>\delta^{-1}
\right)
\leq\delta.
\end{aligned}
\]
The proof is now complete.
\end{proof}

Bounding the sum of $(t+t_0-1)\langle x_{t-1}-x_\nu^\star,\nabla f_{\nu}(x_{t-1})-g_t\rangle$, however, would require some extra work to get the tightest rate, as argued in \cite{harvey2019tight,harvey2019simple}. We shall first make some preparations. Denote
\begin{align*}
\mathscr{V}_t \coloneq{} &
\sum_{k=0}^{t-1} (k+t_0)^2\|x_{k}-x_\nu^\star\|^2.
\end{align*}

\begin{lemma}
\label{lemma:var_process_bound_convergence}
Suppose $t_0\geq 143$. Then on $E_1\cap E_2$, we have
\[
\begin{aligned}
\mathscr{V}_T
\leq{} &
\left(t_0^3+75T^2\right)\|x_0-x_\nu^\star\|^2
+\frac{20T}{m^2}\sum_{t=1}^{T}\|g_t\|^2 \\
& + \sum_{t=1}^{T}\mathfrak{a}_t\cdot\frac{(t+t_0-1)
\langle x_{t-1}-x_\nu^\star,\nabla f_\nu(x_{t-1})-g_t\rangle}{\sqrt{24d}G},
\end{aligned}
\]
where each $\mathfrak{a}_t$ is a non-negative real number satisfying $\mathfrak{a}_t\leq 30G\sqrt{d}(T-t)/m$.
\end{lemma}
\begin{proof}
By \cref{lemma:E1_cap_E2_implies_bounded_x}, on $E_1\cap E_2$ we have $x_t=\widehat{x}_t\in\mathcal{C}$ for all $t\in[T]$. Thus we can apply~\cref{eq:inner_loop_dist_x_nu_star_bound2} and get
\[
\begin{aligned}
\|\widehat{x}_t-x_{\nu}^{\star}\|^2
\leq{} & \frac{q_4(0;t_0)+48t(t\!+\!t_0\!-\!1)(t\!+\!3t_0\!-\!3)}{q_4(t;t_0)}\|x_0-x_{\nu}^{\star}\|^2 \\
&
+\frac{18}{m^2}\sum_{k=1}^t \frac{q_2(k-2;t_0)}{q_4(t;t_0)}\|g_k\|^2 \\
& -\frac{6}{m}\sum_{k=1}^t \frac{q_3(k-1;t_0)}{q_4(t;t_0)}\langle x_{k-1}-x_{\nu}^{\star},g_k-\nabla f_{\nu}(x_{k-1})\rangle.
\end{aligned}
\]
Consequently,
\begin{equation}
\label{eq:sum_t2_sq_dist_bound_temp}
\begin{aligned}
\mathscr{V}_T \leq{} &
\|x_0-x_\nu^\star\|^2
\left(\sum_{t=0}^{T-1}\frac{[(q_4(0;t_0)+48t(t\!+\!t_0\!-\!1)(t\!+\!3t_0\!-\!3)](t+t_0)^2}{q_4(t;t_0)}\right) \\
&
+ \frac{18}{m^2}\sum_{t=1}^{T-1}
\frac{(t+t_0)^2}{q_4(t;t_0)}\sum_{k=1}^t q_2(k-2;t_0) \|g_k\|^2 \\
& + \frac{6}{m}\sum_{t=1}^{T-1}\frac{(t+t_0)^2}{q_4(t;t_0)}\sum_{k=1}^t
q_3(k-1;t_0)
\langle x_{k-1}-x_\nu^\star,\nabla f_\nu(x_{k-1})-g_k\rangle.
\end{aligned}
\end{equation}
We next treat each term on the right-hand side. To bound the first term, we have
\[
\begin{aligned}
& \sum_{t=0}^{T-1}\frac{[(q_4(0;t_0)+48t(t\!+\!t_0\!-\!1)(t\!+\!3t_0\!-\!3)](t+t_0)^2}{q_4(t;t_0)} \\
={} & 
q_4(0;t_0)\sum_{t=0}^{T-1}\frac{t+t_0}{(t\!+\!t_0\!-\!1)(t\!+\!t_0\!-\!2)(t\!+\!t_0\!-\!3)}
+48\sum_{t=1}^{T-1}\frac{t(t\!+\!3t_0\!-\!3)(t\!+\!t_0)}{(t\!+\!t_0\!-\!2)(t\!+\!t_0\!-\!3)} \\
\leq{} & q_4(0;t_0)\frac{t_0}{t_0-3}\sum_{t=0}^{T-1}\frac{1}{(t+t_0-1)(t+t_0-2)}
+48\cdot \frac{(3t_0-2)(t_0+1)}{(t_0-1)(t_0-2)}\sum_{t=1}^{T-1}t \\
\leq{} & t_0^3+75T^2,
\end{aligned}
\]
where we used the fact that $t\mapsto (t\!+\!3t_0\!-\!3)(t\!+\!t_0)/[(t\!+\!t_0\!-\!2)(t\!+\!t_0\!-\!3)]$ is  decreasing over $t\geq 1$ in the first inequality. For the second term of~\cref{eq:sum_t2_sq_dist_bound_temp}, we have
\[
\begin{aligned}
\sum_{t=1}^{T-1}
\frac{(t+t_0)^2}{q_4(t;t_0)}
\sum_{k=1}^tq_2(k-2;t_0) \|g_k\|^2
={} & 
\sum_{k=1}^{T-1}q_2(k-2;t_0)\|g_k\|^2
\sum_{t=k}^{T-1}\frac{(t\!+\!t_0)}{(t\!+\!t_0\!-\!1)(t\!+\!t_0\!-\!2)(t\!+\!t_0\!-\!3)} \\
\leq{} & \sum_{k=1}^{T-1}
q_2(k-2;t_0)\|g_k\|^2\sum_{t=k}^{T-1}\frac{t_0}{(t_0\!-\!1)(k\!+\!t_0\!-\!2)(k\!+\!t_0\!-\!3)} \\
\leq{} & \frac{10}{9}\sum_{k=1}^{T-1}(T-k)\|g_k\|^2
\leq\frac{10T}{9}\sum_{k=1}^{T}\|g_k\|^2.
\end{aligned}
\]
For the last term, we have
\[
\begin{aligned}
& \frac{6}{m}\sum_{t=1}^{T-1}\frac{(t+t_0)^2}{q_4(t;t_0)}\sum_{k=1}^t
q_3(k-1;t_0)
\langle x_{k-1}-x_\nu^\star,\nabla f_\nu(x_{k-1})-g_k\rangle \\
={} &
\frac{6}{m}\sum_{k=1}^{T-1} 
\langle x_{k-1}-x_\nu^\star,\nabla f_\nu(x_{k-1})-g_k\rangle
\cdot q_3(k-1;t_0)\sum_{t=k}^{T-1}\frac{t+t_0}{(t+t_0-1)(t+t_0-2)(t+t_0-3)} \\
={} & \sum_{k=1}^{T-1}\frac{\langle x_{k-1}-x_\nu^\star,\nabla f_\nu(x_{k-1})-g_k\rangle (k+t_0-1)}{\sqrt{24d}G}\cdot \mathfrak{a}_k
\end{aligned}
\]
and we can bound $\mathfrak{a}_k$ by
\[
\begin{aligned}
\mathfrak{a}_k ={} &
\frac{6\sqrt{24d}G}{m}q_2(k-2;t_0)\sum_{t=k}^{T-1}\frac{t+t_0}{(t+t_0-1)(t+t_0-2)(t+t_0-3)}
\\
\leq{} & 
\frac{6\sqrt{24d}G}{m}q_2(k-2;t_0)\sum_{t=k}^{T-1}\frac{t_0}{(t_0-1)(k+t_0-2)(k+t_0-3)}
\leq \frac{30G\sqrt{d}}{m}(T-k).
\end{aligned}
\]
We further let $\mathfrak{a}_T=0$. Summarizing these results, we get the desired bound.
\end{proof}

Now we denote
\begin{align*}
\vartheta_t \coloneq{} &
\frac{1}{\sqrt{24d}G}(t+t_0-1)\langle x_{t-1}-x_\nu^\star,\nabla f_\nu(x_{t-1})-g_t\rangle, \\
\mathfrak{b}_T(\delta) \coloneq{} &
\left(t_0^3+75T^2\right)\|x_0-x_\nu^\star\|^2
+\frac{80TG^2d}{m^2}
\left(T\ln 2+\ln\frac{1}{\delta}\right), \\
\Gamma_T(\delta)\coloneq{} &
\max\!\left\{
\frac{240G\sqrt{d}(T-1)}{m}\ln\frac{1}{\delta},
\sqrt{16\mathfrak{b}_T(\delta)\ln\frac{1}{\delta}}
\right\},
\end{align*}
and
\[
E_4(\delta)\coloneq
\left\{
\sum_{t=1}^T\vartheta_t \leq \Gamma_T
\text{ or }
\mathscr{V}_T
> \sum_{t=1}^{T}\mathfrak{a}_t\vartheta_t + \mathfrak{b}_T(\delta)
\right\},
\]
where $\mathfrak{a}_1,\ldots,\mathfrak{a}_{T-1}$ are given by \cref{lemma:var_process_bound_convergence}. The next lemma bounds the probability of $E_4(\delta)$.
\begin{lemma}
$\mathbb{P}(E_4(\delta))\geq 1-\delta$ for any $\delta\in(0,1)$.
\end{lemma}
\begin{proof}
By \cref{lemma:sub-Gaussian_grad_est}, we have
\[
\begin{aligned}
& \mathbb{E}\Econd*{\exp\!\left(\frac{3\vartheta_t^2}{(t+t_0-1)^2\|x_{t-1}-x_\nu^\star\|^2}\right)}{\mathcal{F}_{t-1}} \\
\leq{} & \mathbb{E}\Econd*{
\exp\!\left(\frac{\|g_t-\nabla f_\nu(x_{t-1})\|^2
}{8G^2d}\right)}{\mathcal{F}_{t-1}}\leq 2.
\end{aligned}
\]
\Cref{lemma:sub-Gaussian_zero_mean_MGF} then leads to
\[
\mathbb{E}\Econd*{\exp(\lambda \vartheta_t)}{\mathcal{F}_{t-1}}
\leq\exp\!\left(\frac{1}{2}(t+t_0-1)^2\|x_{t-1}-x_\nu^\star\|^2\right).
\]
Noting that $\max_{1\leq t\leq T}\mathfrak{a}_t=30G\sqrt{d}(T-1)/m$, by \cref{theorem:generalized_freedman}, we get
\[
\begin{aligned}
& \mathbb{P}\!\left(
\sum_{t=1}^T\vartheta_t > \Gamma_T\text{ and }
\mathscr{V}_T\leq \sum_{t=1}^T\mathfrak{a}_t\vartheta_t+\mathfrak{b}_T(\delta)
\right) \\
\leq{} &
\exp\!\left(
-\frac{\Gamma_T^2}{120G\sqrt{d}(T-1)\Gamma_T/m+8\mathfrak{b}_T(\delta)}
\right).
\end{aligned}
\]
Now, by the definition of $\Gamma_T$, we have
\[
\frac{120G\sqrt{d}(T-1)}{m\Gamma_T}
\leq\frac{1}{2\ln(1/\delta)}
\qquad\text{and}\qquad
\frac{8\mathfrak{b}_T(\delta)}{\Gamma_T^2}\leq\frac{1}{2\ln(1/\delta)},
\]
which then implies that
\[
\begin{aligned}
& \exp\!\left(
-\frac{\Gamma_T^2}{120G\sqrt{d}(T-1)\Gamma_T/m+8\mathfrak{b}_T(\delta)}
\right) \\
={} &
\exp\!\left(
-\frac{1}{120G\sqrt{d}(T-1)/(m\Gamma_T)+8\mathfrak{b}_T(\delta)/\Gamma_T^2}
\right)\leq\delta.
\end{aligned}
\]
We have now obtained the desired result.
\end{proof}

We can now summarize all the previous results, and obtain the convergence guarantee of the inner loop iterations.

\begin{theorembis}{theorem:subroutine_convergence}
Let $\delta\in(0,\min\{1/5,1/\ln d\})$ be arbitrary. Suppose $m=\max\{\mu,G^2/\Delta\}$, $\nu\leq G/(2m)$, and for each $t\in[T]$, the step sizes are given by
\[
\eta_t=\frac{3}{m(t+t_0)},\qquad\text{where}
\quad
t_0= 50d\ln(1+7\ln(700d))+25d\ln\frac{2\pi^2}{3\delta}.
\]
Then with probability at least $1-\delta$, we have $x_t=\widehat{x}_t\in\mathcal{C}\subseteq\mathcal{D}_\nu$ for all $t\in[T]$, and
\[
F_{x_0}(\overline{x}_T)-\inf_{x\in\mathcal{D}} F_{x_0}(x)
\leq C\!\left(\frac{d\ln(1/\delta)}{T}\cdot \Delta
+\left[\frac{1}{T}
+\left(\frac{d\ln(1/\delta)}{T}\right)^{\!2}\right]\frac{m\Delta^2}{G^2}
\right)
+\frac{3G\nu}{2},
\]
where $C>0$ is some absolute numerical constant.
\end{theorembis}
\begin{proof}
In this proof, we use $C_1,C_2,C_3,\ldots$ to denote certain absolute numerical constants.

Note that, by our choice of $t_0$, we have
\[
\mathbb{P}\!\left(E_1\cap E_2\cap E_3(\delta/4)\cap E_4(\delta/4)\right)
\geq 1-\delta.
\]
Now suppose the event $E_1\cap E_2\cap E_3(\delta/4)\cap E_4(\delta/4)$ occurs. Then $x_t=\widehat{x}_t\in\mathcal{C}$ for all $t\in[T]$. Thus we can apply \cref{lemma:convergence_raw_gap_bound} together with the definitions of $E_3(\delta/4)$ and $E_4(\delta/4)$ to get
\[
\begin{aligned}
F_{\nu,x_0}(\overline{x}_T)
-F_{\nu,x_0}(x_\nu^\star)
\leq{} &
\frac{2}{T(T+2t_0-1)}\sqrt{24d}G\Gamma_T(\delta/4) \\
&
+\frac{6}{mT(T+2t_0-1)}\cdot 4G^2d\left(T\ln 2+\ln\frac{4}{\delta}\right) \\
& + \frac{m(48T+t_0(t_0-1)/3)}{T(T+2t_0-1)}\|x_0-x_\nu^\star\|^2.
\end{aligned}
\]
Then, we note that
\[
\begin{aligned}
\sqrt{16\mathfrak{b}_T(\delta)\ln\frac{1}{\delta}}
\leq{} & C_1\!\left(\sqrt{(t_0^3+T^2)\frac{\Delta^2}{G^2}+\frac{T^2G^2d}{m^2}\ln\frac{1}{\delta}}\right)
\leq
C_1\!\left(
\left(t_0^{3/2}+T\right)\frac{\Delta}{G}+\frac{TG\sqrt{d}}{m}\ln\frac{1}{\delta}
\right),
\end{aligned}
\]
where we used $\|x_0-x_\nu^\star\|^2\leq 4\Delta^2/G^2$. Therefore
\[
\begin{aligned}
\Gamma_T(\delta/4)
\leq{} &
C_2\!\left(
\frac{G\sqrt{d}T}{m}\ln\frac{1}{\delta}
+\left(t_0^{3/2}+T\right)\frac{\Delta}{G}
\right) \\
\leq{} &
C_3\!\left(
\frac{\sqrt{d}T\Delta}{G}\ln\frac{1}{\delta}
+\left(d^{\frac{3}{2}}\left(\ln\ln d\right)^{\frac{3}{2}}+T\right)\frac{\Delta}{G}
\right) \\
\leq{} &
C_4\!\left(\frac{\sqrt{d}T\Delta}{G}\ln\frac{1}{\delta}
+d^{\frac{3}{2}}\left(\ln\frac{1}{\delta}\right)^{\!\frac{3}{2}}\frac{\Delta}{G}\right),
\end{aligned}
\]
where we used the choice of $t_0$ as well as $m\geq G^2/\Delta$ in the second step, and used $\delta\leq 1/{\ln d}$ in the third step. Consequently,
\[
\begin{aligned}
F_{\nu,x_0}(\overline{x}_T)
-F_{\nu,x_0}(x_\nu^\star)
\leq{} &
C_5\!\left(
\frac{\sqrt{d}G}{T^2}\left[\frac{\sqrt{d}T\Delta}{G}\ln\frac{1}{\delta}
+d^{\frac{3}{2}}\left(\ln\frac{1}{\delta}\right)^{\!\frac{3}{2}}\frac{\Delta}{G}
\right]
+\frac{G^2d}{mT}\ln\frac{1}{\delta}
+\frac{T+t_0^2}{T^2}\frac{m\Delta^2}{G^2}
\right) \\
\leq{} &
C_6\!\left(
\frac{d\Delta}{T}\ln\frac{1}{\delta}
+\frac{d^2\Delta}{T^2}\left(\ln\frac{1}{\delta}\right)^{\!\frac{3}{2}}
+\left(\frac{1}{T}+\frac{d^2(\ln\ln d)^2}{T^2}\right)\frac{m\Delta^2}{G^2}
\right) \\
\leq{} &
C_7\!\left(
\frac{d\Delta}{T}\ln\frac{1}{\delta}
+\left(\frac{1}{T}+\frac{d^2(\ln(1/\delta))^2}{T^2}\right)\frac{m\Delta^2}{G^2}
\right),
\end{aligned}
\]
where we again used $m\geq G^2/\Delta$ to combine relevant terms, and used $\delta<1/{\ln d}$ in the last step. 

As a final step of the proof, we notice that $F_{x_0}(x)\geq \Delta\geq f(x_0)=F_{x_0}(x_0)$ for all $x\in\mathcal{D}\backslash\mathcal{S}_{x_0}(\Delta)$, and thus
\[
\begin{aligned}
\inf_{x\in\mathcal{D}} F_{x_0}(x)
={} & \inf_{x\in \mathcal{S}_{x_0}(\Delta)} F_{x_0}(x)
=\inf_{x\in x_0+\frac{10\Delta}{G}\mathbb{B}_d} F_{x_0}(x) \\
\geq{} &
\inf_{x\in x_0+\frac{10\Delta}{G}\mathbb{B}_d} (F_{\nu,x_0}(x)-G\nu)
=F_{\nu,x_0}(x_\nu^\star)-G\nu,
\end{aligned}
\]
where the inequality follows from Part 1 of \cref{lemma:properties_smoothed_func}. Consequently, 
\[
\begin{aligned}
F_{x_0}(\overline{x}_T)-\inf_{x\in\mathcal{D}} F_{x_0}(x)
\leq{} & F_{\nu,x_0}(\overline{x}_T) + m\nu^2
-F_{\nu,x_0}(x_\nu^\star)+G\nu \\
\leq{} & F_{\nu,x_0}(\overline{x}_T)
-F_{\nu,x_0}(x_\nu^\star) + \frac{3G\nu}{2},
\end{aligned}
\]
where we used Part 1 of \cref{lemma:bigF_smoothed_ver_comparison} and also $m\nu^2\leq G\nu/2$ by our choice of $\nu$. Combining all the previous results finishes the proof.
\end{proof}

\section{Conclusion}

In this work, we study zeroth-order optimization for nonsmooth, nonconvex problems with convex liftings, with discrete-time state-feedback $\mathcal{H}_\infty$ policy optimization being a specific motivating example. We propose the zeroth-order proximal point algorithm, and show that with high probability, i) all intermediate iterates will remain in the feasible region even though the algorithm is projection-free; and ii) the algorithm returns an $\epsilon$-optimal solution with oracle complexity $\widetilde{O}(d\epsilon^{-3})$. By applying the proposed algorithm to discrete-time state-feedback $\mathcal{H}_\infty$ control, we obtain a policy optimization algorithm which achieves global optimality with guaranteed oracle complexity. Some promising future directions are as follows: i) It would be interesting to see if the complexity bound can be further improved by establishing stronger landscape properties of state-feedback $\mathcal{H}_\infty$ policy optimization. ii) Our algorithm assumes access to an exact zeroth-order oracle; how to develop a purely sample-based method that combines $\mathcal{H}_\infty$ cost estimation and optimization would be another interesting future direction.

\section*{Statement on the Use of AI Systems.}
The authors have had active discussions with ChatGPT 5.6 Sol on several technical details of this work. Particularly, the AI system assisted in formulating and proving~\cref{lemma:time-uniform_poly_growth_v}, and also assisted in pinpointing the important reference~\cite{howard2021time} which contains~\cref{theorem:stitched_boundary}; both \cref{lemma:time-uniform_poly_growth_v} and \cref{theorem:stitched_boundary} are critical tools in improving the dependence of the oracle complexity on the dimension $d$. The AI system also proposed an alternative approach for bounding the failure probability, but after careful inspection, we abandoned this approach because it leads to inferior complexity bounds. In addition, the AI system was used for polishing the language of the article. The authors have independently verified, substantially revised, and rewritten all AI-assisted materials, and take full responsibility for the accuracy and content of the manuscript.

\bibliographystyle{unsrt}
\bibliography{refs.bib}

\appendix

\crefalias{section}{appendix}
\crefalias{subsection}{appendix}

\section{Proofs of Auxiliary Lemmas}

\subsection{Proof of \Cref{lemma:properties_smoothed_func}}
\label{appendix:proof_properties_smoothed_func}

Note that $S+\nu\mathbb{B}_d\subseteq\mathcal{S}(\Delta)\subseteq\mathcal{D}$ implies $S\subseteq\mathcal{D}_\nu$, and therefore $f_\nu(x)$ and $\nabla f_\nu(x)$ are well-defined for all $x\in S$. Now we proceed to prove the three claims. 
\begin{enumerate}
\item We have
\[
\begin{aligned}
|f_\nu(x)-f(x)| ={} & \left|
\mathbb{E}_{y\sim\mathcal{U}(\mathbb{B}_d)}[f(x+\nu y)-f(x)]
\right| \\
\leq{} & \mathbb{E}_{y\sim\mathcal{U}(\mathbb{B}_d)}\!\left[\left|f(x+\nu y)-f(x)\right|\right] \\
\leq{} & \mathbb{E}_{y\sim\mathcal{U}(\mathbb{B}_d)}\!\left[G_\Delta\|\nu y\|\right]\leq G_\Delta\nu.
\end{aligned}
\]
\item We have
\[
\begin{aligned}
|f_\nu(x)-f_\nu(y)|
={} & 
\left|\mathbb{E}_{w\sim\mathcal{U}(\mathbb{B}_d)}[f(x+\nu w)-f(y+\nu w)]\right| \\
\leq{}  &
\mathbb{E}_{w\sim\mathcal{U}(\mathbb{B}_d)}\!\left[\left|f(x+\nu w)-f(y+\nu w)\right|\right] \\
\leq{} & \mathbb{E}_{w\sim\mathcal{U}(\mathbb{B}_d)}\!\left[G_\Delta\|x-y\|\right] = G_\Delta\|x-y\|.
\end{aligned}
\]
\item By \cref{lemma:expectation_grad_est}, we have
\[
\begin{aligned}
\left\|\nabla f_\nu(x)-\nabla f_\nu(y)\right\|
= & \left\|\frac{d}{2\nu}\mathbb{E}_{z\sim\mathcal{U}(\mathbb{S}_{d-1})}
\!\left[(f(x+\nu z)-f(x-\nu z)-f(y+\nu z)+f(y-\nu z))z\right]\right\| \\
\leq{} & \frac{d}{2\nu}\mathbb{E}_{z\sim\mathcal{U}(\mathbb{S}_{d-1})}
\!\left[|f(x+\nu z)-f(y+\nu z)|+|f(x-\nu z)-f(y-\nu z)|\right] \\
\leq{} & \frac{d}{2\nu}\mathbb{E}_{z\sim\mathcal{U}(\mathbb{S}_{d-1})}
\!\left[2G_\Delta\|x-y\|\right] \\
={} & \frac{G_\Delta d}{\nu}\|x-y\|.
\end{aligned}
\]
Now suppose all convex combinations of $x$ and $y$ are in $S$. Then
\[
\begin{aligned}
f_\nu(y)-f_\nu(x) ={} & \int_0^1 \langle\nabla f_\nu(x+s(y-x)),y-x\rangle\ud s \\
={} & \langle\nabla f_\nu(x),y-x\rangle
+\int_0^1\langle\nabla f_\nu(x+s(y-x))-\nabla f_\nu(x),y-x\rangle\ud s \\
\leq{} & \langle\nabla f_\nu(x),y-x\rangle+
\int_0^1\|\nabla f_\nu(x+s(y-x))-\nabla f_\nu(x)\|\|y-x\|\,\ud s \\
\leq{} &
\langle\nabla f_\nu(x),y-x\rangle
+\|y-x\|\int_0^1 L_{\nu,\Delta}s\|y-x\|\,\ud s \\
={} & \langle\nabla f_\nu(x),y-x\rangle
+\frac{L_{\nu,\Delta}}{2}\|y-x\|^2.
\end{aligned}
\]
The proof is now complete.
\end{enumerate}

\subsection{Proof of \Cref{lemma:sub-Gaussian_grad_est}}
\label{appendix:proof_lemma_sub-Gaussian_grad_est}

The proof of \cref{lemma:sub-Gaussian_grad_est} is based on the following concentration inequality; this concentration inequality was also used in \cite{shamir2017optimal} for bounding $\mathbb{E}_z\!\left[\|\sfG(x;z,\nu)\|^2\right]$.

\begin{lemma}[\cite{aubrun2026optimal}]
\label{lemma:Levy_concentration_sphere}
Let $h:\mathbb{S}_{d-1}\rightarrow\mathbb{R}$ be a $1$-Lipschitz continuous function. Then for every $\epsilon>0$, we have
\[
\mathbb{P}_z\!\left(|h(z)-\mathbb{E}_z[h(z)]|\geq \epsilon\right)
\leq \exp\!\left(
-\frac{d\epsilon^2}{2}\right),
\]
where $z\sim\mathcal{U}(\mathbb{S}_{d-1})$.
\end{lemma}

Now we prove \cref{lemma:sub-Gaussian_grad_est}. Define
\[
h_x(z) = \frac{1}{G_\Delta}\cdot\frac{f(x+\nu z)-f(x-\nu z)}{2\nu}.
\]
When $z\sim\mathcal{U}(\mathbb{S}_{d-1})$, we have $\mathbb{E}_z[h_x(z)]=0$ by symmetry. Moreover, for all $z_1,z_2\in\mathbb{S}_{d-1}$, we have
\[
\begin{aligned}
|h_x(z_1)-h_x(z_2)|
={} & \frac{1}{2\nu G_\Delta}
|f(x+\nu z_1)-f(x-\nu z_1)-f(x+\nu z_2)+f(x-\nu z_2)| \\
\leq{} & \frac{1}{2\nu G_\Delta}
\left(|f(x+\nu z_1)-f(x+\nu z_2)|
+|f(x-\nu z_1)-f(x-\nu z_2)|\right) \\
\leq{} & \frac{1}{2\nu G_\Delta}
\cdot 2G_\Delta \nu\|z_1-z_2\|
=\|z_1-z_2\|,
\end{aligned}
\]
showing that $h_x$ is a $1$-Lipschitz continuous function on $\mathbb{S}_{d-1}$. We can now apply \cref{lemma:Levy_concentration_sphere} to obtain
\[
\mathbb{P}_z\!\left(
|h_x(z)|\geq \epsilon
\right) \leq\exp\!\left(-\frac{d\epsilon^2}{2}\right),\qquad\forall \epsilon>0.
\]
On the other hand, note that for any $z\in\mathbb{S}_{d-1}$, we have $\|G(x;z,\nu)\| = G_\Delta d |h_x(z)|$. Therefore
\[
\mathbb{P}_z\!\left(\|G(x;z,\nu)\|\geq \epsilon\right)
=
\mathbb{P}_z\!\left(\|h_x(z)\|\geq \frac{\epsilon}{G_\Delta d}\right)
\leq \exp\!\left(
-\frac{\epsilon^2}{2G_\Delta^2 d}\right),
\]
which is just~\cref{eq:high_prob_norm_grad_est}.

The bound~\cref{eq:sub-gaussiaon_grad_est_norm} is a standard result for sub-Gaussian random variables: Note that
\[
\begin{aligned}
\mathbb{E}_z\!\left[\exp\!\left(\frac{\|G(x;z,\nu)\|^2}{4G_\Delta^2 d}\right)\right]
={} & 1+\sum_{k=1}^\infty\frac{1}{k!}\mathbb{E}_z\!\left[\left(\frac{\|G(x;z,\nu)\|}{2G_\Delta\sqrt{d}}\right)^{\!2k}\right],
\end{aligned}
\]
while
\[
\begin{aligned}
\mathbb{E}_z\!\left[\left(\frac{\|G(x;z,\nu)\|}{2G_\Delta\sqrt{d}}\right)^{\!2k}\right]
={} & \int_0^{+\infty}\mathbb{P}_z\!\left(\frac{\|G(x;z,\nu)\|}{2G_\Delta\sqrt{d}}
\geq u\right) 2k u^{2k-1}\,\ud u \\
\leq{} &
\int_0^{+\infty}
\exp\!\left(
-2u^2
\right)2ku^{2k-1}\,\ud u \\
={} & \frac{k!}{2^{k}}.
\end{aligned}
\]
Therefore
\[
\mathbb{E}_z\!\left[\exp\!\left(\frac{\|G(x;z,\nu)\|^2}{4G_\Delta^2 d}\right)\right]
\leq 1+\sum_{k=1}^\infty\frac{1}{k!}\frac{k!}{2^k}
=2,
\]
which completes the proof of~\cref{eq:sub-gaussiaon_grad_est_norm}.

To show the inequality~\cref{eq:sub-Gaussian_norm_centered_grad}, we note that 
\[
\begin{aligned}
\left\|G(x;z,\nu) - \nabla f_\nu(x)\right\|^2
\leq{} &
\left(1+\frac{1}{3}\right)\|G(x;z,\nu)\|^2 + (1+3)\left\|\nabla f_\nu(x)\right\|^2 \\ 
\leq{} &
\frac{4}{3}\|G(x;z,\nu)\|^2 + 4G_\Delta^2
\end{aligned}
\]
Hence
\[
\begin{aligned}
\mathbb{E}_z\!\left[\exp\!\left(\frac{\left\|G(x;z,\nu)-\nabla f_\nu(x)\right\|^2}{8G_\Delta^2d}\right)\right]
\leq{} & \mathbb{E}_z\!\left[\exp\!\left(\frac{\left\|G(x;z,\nu)\right\|^2}{6G_\Delta^2d} + \frac{1}{2d}\right)\right] \\
\leq{} & 
\mathbb{E}_z\!\left[\exp\left(\frac{\left\|G(x;z,\nu)\right\|^2}{4G_\Delta^2d}\cdot \frac{2}{3}\right)\right] \cdot e^{1/6},
\end{aligned}
\]
where we used $1/(2d)\leq 1/6$ for $d\geq 3$. From \cref{eq:sub-gaussiaon_grad_est_norm} and Jensen's inequality, we can get 
\[
\mathbb{E}_z\!\left[\exp\left(\frac{\left\|G(x;z,\nu)\right\|^2}{4G_\Delta^2d}\cdot \frac{2}{3}\right)\right] 
\leq \left(\mathbb{E}_z\!\left[\exp\frac{\left\|G(x;z,\nu)\right\|^2}{4G_\Delta^2d}\right] \right) ^{\!\frac{2}{3}}
\leq 2^\frac{2}{3}.
\]
Thus
\[
\mathbb{E}_z\!\left[\exp\!\left(\frac{\left\|G(x;z,\nu)-\nabla f_\nu(x)\right\|^2}{8G_\Delta^2d}\right)\right] \leq 2^{\frac{2}{3}}\cdot e^{1/6}\leq 2,
\]
which is just~\cref{eq:sub-Gaussian_norm_centered_grad}.

\subsection{Proof of \Cref{lemma:time-uniform_poly_growth_v}}
\label{appendix:proof_time-uniform_poly_growth_v}

Let $t\in[T]$ be arbitrary. We have, for any $k\in[t]$ and $\lambda\in(0,1/v_k]$,
\[
\begin{aligned}
\mathbb{E}\Econd*{\exp(\lambda h_k-\lambda v_k\ln 2)}{\mathcal{F}_{k-1}}
={} & 2^{-\lambda v_k}\mathbb{E}\Econd*{(\exp(h_k/v_k))^{\lambda v_k}}{\mathcal{F}_{k-1}} \\
\leq{} & 2^{-\lambda v_k}\left(\mathbb{E}\Econd*{\exp(h_k/v_k)}{\mathcal{F}_{k-1}}\right)^{\lambda v_k} \\
\leq{} & 2^{-\lambda v_k} 2^{\lambda v_k}=1,
\end{aligned}
\]
where the second step follows by Jensen's inequality and $\lambda v_k\leq 1$. As a result, if we let
\[
L_k(\lambda)=\exp\!\left(\lambda\left(\sum_{i=1}^k h_i-V_k\ln 2\right)\right),
\]
then $(L_k(\lambda))_{k=1}^t$ forms a non-negative supermartingale whenever $0\leq \lambda\leq 1/\max_{k\leq t} v_k=1/v_t$. Noting that $\mathbb{E}[L_1(1/v_t)]\leq 1$, we can now apply Ville's 
inequality to $(L_k(1/v_t))_{k=1}^t$ and get
\[
\mathbb{P}\!\left(\sum_{k=1}^t h_k\geq cV_t\right)
=\mathbb{P}\!\left((L_t(1/v_t)\geq e^{(c-\ln 2)V_t/v_t}\right)\leq\exp\!\left(-\frac{V_t}{v_t}(c-\ln 2)\right)
\]
whenever $c>\ln 2$. Then by the condition~\cref{eq:lemma_time-uniform_poly_growth_v_condition}, we have
\[
\mathbb{P}\!\left(\sum_{k=1}^t h_k\geq cV_t\right)
\leq \exp(-(\alpha t+\beta)(c-\ln 2)).
\]
By taking the union bound, we get
\[
\begin{aligned}
\mathbb{P}\!\left(\exists t\in[T]\text{ s.t. }
\sum_{k=1}^t h_k> cV_t
\right)\leq{} & \sum_{t=1}^T \exp(-(\alpha t+\beta)(c-\ln 2)) \\
\leq{} & \exp(-\beta(c-\ln 2))\sum_{t=1}^\infty\left(\exp(-\alpha(c-\ln 2))\right)^t \\
={} & \frac{\exp(-(\alpha+\beta)(c-\ln 2))}{1-\exp(-\alpha(c-\ln 2))},
\end{aligned}
\]
which gives the desired bound.

\section{Proofs for the Convex Lifting of $J_\infty$}

\subsection{Proof of \cref{proposition:convex_lifting_Hinf}}
\label{appendix:proof_convex_lifting_Hinf}

\noindent\textbf{Part 1.} It is known in classical control theory that, by introducing the transfer matrix
\[
\mathbf{T}_K(z)=\begin{bmatrix}
Q^{1/2} \\
R^{1/2}K
\end{bmatrix}(zI-A-BK)^{-1}B_w,
\]
the objective function $J_\infty$ can be equivalently written as the $\mathcal{H}_\infty$ norm of $\mathbf{T}_K$, i.e.,
\[
J_\infty(K) = \|\mathbf{T}_K\|_{\mathcal{H}_\infty}
=\max_{\omega\in[0,2\pi)}
\sigma_{\max}\!\left(\mathbf{T}_{K}(e^{\uj\omega})\right),
\]
where $\sigma_{\max}(\cdot)$ denotes the largest singular value of a matrix. The following discrete-time bounded real lemma then plays a fundamental role in convexifying the epigraph of $J_\infty$.

\begin{lemma}[Non-strict Bounded Real Lemma, {\cite{vaidyanathan1985discrete}}]
\label{lemma:DT_bounded_real}
Let $\mathbf{G}(z)=C(zI-A)^{-1}B+D$, where $A,B,C,D$ are real matrices of proper dimensions, and let $\gamma>0$ be arbitrary. Suppose $(A,B)$ is controllable and $(C,A)$ is observable. Then the following statements are equivalent:
\begin{enumerate}
\item $A$ is Schur stable and $\|\mathbf{G}\|_{\mathcal{H}_\infty}\leq\gamma$.
\item There exists a real positive definite matrix $P$ such that
\[
\begin{bmatrix}
A^\tran PA-P & A^\tran P B \\
B^\tran PA & B^\tran P B - \gamma^2 I
\end{bmatrix}+\begin{bmatrix}
C^\tran \\ D^\tran
\end{bmatrix}
\begin{bmatrix}
C & D
\end{bmatrix}\preceq 0.
\]
\end{enumerate}
\end{lemma}


Now we are ready to prove the first statement of \cref{proposition:convex_lifting_Hinf}. Let $K\in\mathbb{R}^{n_u\times n_x}$ and $\gamma>0$ be arbitrary. Since $B_wB_w^\tran$ is assumed to be positive definite, $B_w$ has full row rank, and thus $(A+BK, B_w)$ is controllable. Furthermore, since $Q$ is assumed to be positive definite, the matrix $\begin{bmatrix}
Q & K^\tran R^{1/2}
\end{bmatrix}^\tran$ has full column rank, which further implies that $\left(\begin{bmatrix}
Q & K^\tran R^{1/2}
\end{bmatrix}^\tran,A+BK\right)$ is observable. Thus we can apply the non-strict bounded real lemma to the transfer matrix $\mathbf{T}_K$ and conclude that $K\in\mathcal{K}$ and $J_\infty(K)=\|\mathbf{T}_K\|_{\mathcal{H}_\infty}\leq\gamma$ if and only if there exists a positive definite matrix $P$ such that
\[
\begin{bmatrix}
(A+BK)^\tran P(A+BK)-P & (A+BK)^\tran PB_w \\
B_w^\tran P(A+BK) & B_w^\tran PB_w-\gamma^2 I
\end{bmatrix}
+\begin{bmatrix}
Q^{1/2} & K^\tran R^{1/2} \\
0 & 0
\end{bmatrix}\begin{bmatrix}
Q^{1/2} & 0 \\ R^{1/2}K & 0
\end{bmatrix}
\preceq 0,
\]
which is just $\mathscr{M}_{\mathrm{lft}}(K,\gamma,P)\preceq 0$. In other words, $(K,\gamma)\in\operatorname{epi}(J_\infty)$ if and only if there exists a real positive definite matrix $P$ such that $(K,\gamma,P)\in\mathfrak{L}_{\mathrm{lft}}$. The proof of Part 1 is now complete.

\noindent\textbf{Part 2.} Temporarily denote $\mathcal{U}=\mathbb{R}^{n_u\times n_x}\times\mathbb{R}_{++}\times \mathbbm{S}_{++}^{n_x}$ and $\mathcal{V}=\mathbb{R}_{++}\times\mathbb{R}^{n_u\times n_x}\times\mathbbm{S}_{++}^{n_x}$. It is straightforward to see that the domain of $\Phi$ can be extended to $\mathcal{U}$, i.e., we let
\[
\Phi(K,\gamma,P) = (\gamma,\gamma KP^{-1},\gamma P^{-1}),
\qquad\forall K\in\mathbb{R}^{n_u\times n_x},\gamma\in\mathbb{R}_{++},P\in\mathbbm{S}_{++}^{n_x},
\]
and $\Phi$ is obviously infinitely differentiable on the open set $\mathcal{U}\supseteq \mathfrak{L}_{\mathrm{lft}}$. Moreover, it is straightforward to check that the inverse of $\Phi:\mathcal{U}\rightarrow\mathcal{V}$ is given by
\[
\Phi^{-1}(\gamma,Y,X)=(YX^{-1},\gamma,\gamma X^{-1}),
\qquad\forall (\gamma,Y,X)\in\mathcal{V},
\]
and that $\Phi^{-1}$ is infinitely differentiable on the open set $\mathcal{V}$.

Now we only need to prove that $\Phi(\mathfrak{L}_{\mathrm{lft}})\subseteq\mathfrak{F}_{\mathrm{cvx}}$ and $\Phi^{-1}(\mathfrak{F}_{\mathrm{cvx}})\subseteq\mathfrak{L}_{\mathrm{lft}}$. Let $(K,\gamma,P)\in\mathfrak{L}_{\mathrm{lft}}$ be arbitrary, and denote $(\gamma,Y,X)=\Phi(K,\gamma,P)$. Note that $\mathscr{M}_{\mathrm{lft}}(K,\gamma,P)\preceq 0$ can be equivalently written as
\[
\begin{aligned}
-\begin{bmatrix}
(A\!+\! BYX^{-1})^\tran & I & (YX^{-1})^\tran \\
B_w^\tran & 0 & 0
\end{bmatrix}
\begin{bmatrix}
\gamma X^{-1} & 0 & 0 \\ 0 & Q & 0 \\
0 & 0 & R
\end{bmatrix}
\begin{bmatrix}
A\!+\!BYX^{-1} & B_w \\
I & 0 \\
YX^{-1} & 0
\end{bmatrix}
+\begin{bmatrix}
\gamma X^{-1} & 0 \\ 0 & \gamma^2 I
\end{bmatrix}
\succeq 0,
\end{aligned}
\]
where we have used $K=YX^{-1}$ and $P=\gamma X^{-1}$. Since $X=\gamma P^{-1}\succ 0$, we can left- and right-multiply both sides by $\gamma^{-1/2}\operatorname{diag}(X,I)$ to get the equivalent matrix inequality
\[
\begin{aligned}
-\begin{bmatrix}
(AX\!+\! BY)^\tran & X & Y^\tran \\
B_w^\tran & 0 & 0
\end{bmatrix}
\begin{bmatrix}
X^{-1} & 0 & 0 \\ 0 & \gamma^{-1}Q & 0 \\
0 & 0 & \gamma^{-1}R
\end{bmatrix}
\begin{bmatrix}
AX\!+\!BY & B_w \\
X & 0 \\
Y & 0
\end{bmatrix}
+\begin{bmatrix}
X & 0 \\ 0 & \gamma I
\end{bmatrix}
\succeq 0.
\end{aligned}
\]
Note that $X$, $Q$ and $R$ are all positive definite, and so by Schur complement, the above matrix inequality is equivalent to
\begin{equation}
\label{eq:Mcvx_succ_0_in_proof}
\begin{bmatrix}
X & 0 & (AX\!+\!BY)^\tran & X & Y^\tran \\
0 & \gamma I & B_w^\tran & 0 & 0 \\
AX\!+\!BY & B_w & X & 0 & 0 \\
X & 0 & 0 & \gamma Q^{-1} & 0 \\
Y & 0 & 0 & 0 & \gamma R^{-1}
\end{bmatrix}
\succeq 0,
\end{equation}
which is just $\mathscr{M}_{\mathrm{cvx}}(\gamma,Y,X)\succeq 0$. Therefore $(\gamma,Y,X)=\Phi(K,\gamma,P)\in\mathfrak{F}_{\mathrm{cvx}}$, and by the arbitrariness of $(K,\gamma,P)\in\mathfrak{L}_{\mathrm{lft}}$, we get $\Phi(\mathfrak{L}_{\mathrm{lft}})\subseteq\mathfrak{F}_{\mathrm{cvx}}$.

Conversely, let $(\gamma,Y,X)\in\mathfrak{F}_{\mathrm{cvx}}$ be arbitrary. By denoting $K=YX^{-1}$ and $P=\gamma X^{-1}$ so that $(K,\gamma,P)=\Phi^{-1}(\gamma,Y,X)$, we can just reverse the previous calculation to obtain $\mathscr{M}_{\mathrm{lft}}(K,\gamma,P)\preceq 0$ from $\mathscr{M}_{\mathrm{cvx}}(\gamma,Y,X)\succeq 0$, showing that $\Phi^{-1}(\gamma,Y,X)\in\mathfrak{L}_{\mathrm{lft}}$. By the arbitrariness of $(\gamma,Y,X)\in\mathfrak{F}_{\mathrm{cvx}}$, we get $\Phi^{-1}(\mathfrak{F}_{\mathrm{cvx}})\subseteq\mathfrak{L}_{\mathrm{lft}}$. The proof of Part 2 is now complete.

\subsection{Proof of Part 1 of \cref{proposition:Hinf_analytical_properties}}
\label{appendix:proof_Hinf_analytical_properties}

Let $\Delta>0$ be arbitrary. We shall prove that the set
\[
\setc*{(\gamma,Y,X)\in\mathfrak{F}_{\mathrm{cvx}}}{\gamma\leq\Delta}
\]
is a compact set. Then by Part 2 of \cref{proposition:convex_lifting_Hinf}, we see that $\setc*{(K,\gamma,P)\in\mathfrak{L}_{\mathrm{lft}}}{\gamma\leq\Delta}$ is compact.

Let $(\gamma,Y,X)\in\mathfrak{F}_{\mathrm{cvx}}$ satisfy $\gamma\leq\Delta$. Recall that $\mathscr{M}_{\mathrm{cvx}}(\gamma,Y,X)\succeq 0$ is just~\cref{eq:Mcvx_succ_0_in_proof}. By switching the second and the third rows of blocks as well as the second and the third columns of blocks on the left-hand side, we see that $\mathscr{M}_{\mathrm{cvx}}(\gamma,Y,X)\succeq 0$ is equivalent to
\begin{equation*}
\begin{bmatrix}
X & (AX\!+\!BY)^\tran & 0 &  X & Y^\tran \\
AX\!+\!BY & X & B_w & 0 & 0 \\
0 & B_w^\tran & \gamma I & 0 & 0 \\
X & 0 & 0 & \gamma Q^{-1} & 0 \\
Y & 0 & 0 & 0 & \gamma R^{-1}
\end{bmatrix}\succeq 0.
\end{equation*}
By Schur complement, the above matrix inequality is further equivalent to
\[
\begin{bmatrix}
X & (AX\!+\!BY)^\tran \\
AX\!+\!BY & X
\end{bmatrix}
-\begin{bmatrix}
0 & X & Y^\tran \\
B_w 0 & 0
\end{bmatrix}
\begin{bmatrix}
\gamma I & 0 & 0 \\ 0 & \gamma Q^{-1} & 0 \\ 0 & 0 & \gamma R^{-1}
\end{bmatrix}^{-1}
\begin{bmatrix}
0 & B_w^\tran \\ X & 0 \\ Y & 0
\end{bmatrix}\succeq 0,
\]
which is just
\[
\begin{bmatrix}
X-\gamma^{-1}(XQX+Y^\tran RY) & (AX\!+\!BY)^\tran \\
AX+BY & X-\gamma^{-1}W
\end{bmatrix} \succeq 0,
\]
where we recall the notation $W\coloneqq B_wB_w^\tran$. This matrix inequality implies that the two diagonal blocks $X-\gamma^{-1}(XQX+Y^\tran RY)$ and $X- \gamma^{-1}W$ are positive semidefinite. By $X\succeq \gamma^{-1}(XQX+Y^\tran RY)$, we get
\[
\begin{aligned}
\operatorname{tr}(X)\geq{} &
\gamma^{-1}\left[\operatorname{tr}(XQX)+\operatorname{tr}(Y^\tran RY)\right] \\
\geq {} & \gamma^{-1}\left(\lambda_{\min}(Q)\|X\|_F^2+\lambda_{\min}(R)\|Y\|_F^2\right),
\end{aligned}
\]
where we use $\lambda_{\min}(\cdot)$ to denote the smallest eigenvalue of a real symmetric matrix. As a result, we have
\[
\lambda_{\min}(Q)\|X\|_F^2+\lambda_{\min}(R)\|Y\|_F^2\leq \gamma \operatorname{tr}(X)\leq\gamma \sqrt{n_x}\|X\|_F\leq \Delta\sqrt{n_x}\|X\|_F,
\]
On the other hand, from $X\succeq \gamma^{-1}W$, we can obtain
\[
\begin{aligned}
\lambda_{\min}(X)\geq{} &
\lambda_{\min}(\gamma^{-1}W) \\
={} & \gamma^{-1}\lambda_{\min}(W)\geq\Delta^{-1}\lambda_{\min}(W).
\end{aligned}
\]
In addition, $X\succeq \gamma^{-1}W\succ 0$ also implies that $\|X\|_F\geq \gamma^{-1}\|W\|_F$, and thus 
\[
\gamma \geq \|W\|_F/\|X\|_F.
\]
By the arbitrariness of $(\gamma,Y,X)$, we see that $\setc*{(\gamma,Y,X)\in\mathfrak{F}_{\mathrm{cvx}}}{\gamma\leq\Delta}$ is a subset of
\[
\mathfrak{C}_\Delta\coloneqq\setv*{(\gamma,Y,X)\in \mathbb{R}_{++}\times
\mathbb{R}^{n_u\times n_x}\times\mathbbm{S}^{n_x}_{++}}
{\begin{aligned}
& \lambda_{\min}(Q)\|X\|_F^2+\lambda_{\min}(R)\|Y\|_F^2\leq \Delta\sqrt{n_x}\|X\|_F, \\
& \lambda_{\min}\big(X\big)\geq\Delta^{-1}\lambda_{\min}(W), \\
&\|W\|_F/\|X\|_F\leq \gamma\leq\Delta
\end{aligned}}.
\]
Since $Q,R$ and $W$ are all positive definite, we have $\lambda_{\min}(Q)>0$, $\lambda_{\min}(R)>0$ and $\|W\|_F\geq \lambda_{\min}(W)>0$. Then it is not hard to see that $\mathfrak{C}_\Delta$ is bounded and closed in $\mathbb{R}\times
\mathbb{R}^{n_u\times n_x}\times\mathbbm{S}^{n_x}$, and thus is compact. As a result, the set $\setc*{(\gamma,Y,X)\in\mathfrak{F}_{\mathrm{cvx}}}{\gamma\leq\Delta}$ has the following equivalent representation:
\[
\begin{aligned}
& \setc*{(\gamma,Y,X)\in\mathfrak{F}_{\mathrm{cvx}}}{\gamma\leq\Delta} \\
={} & \setc*{(\gamma,Y,X)\in\mathfrak{C}_\Delta}{\gamma\leq\Delta,\mathscr{M}_{\mathrm{cvx}}(\gamma,Y,X)\succeq 0} \\
={} & \setc*{(\gamma,Y,X)\in\mathfrak{C}_\Delta}{\gamma\leq\Delta}\cap\setc*{(\gamma,Y,X)\in\mathfrak{C}_\Delta}{\lambda_{\min}(\mathscr{M}_{\mathrm{cvx}}(\gamma,Y,X))\geq 0}.
\end{aligned}
\]
By the continuity of $(\gamma,Y,X)\mapsto \gamma$ and $(\gamma,Y,X)\mapsto\lambda_{\min}(\mathscr{M}_{\mathrm{cvx}}(\gamma,Y,X))$ over $(\gamma,Y,X)\in\mathfrak{C}_\Delta$, we see that both $\setc*{(\gamma,Y,X)\in\mathfrak{C}_\Delta}{\gamma\leq\Delta}$ and $\setc*{(\gamma,Y,X)\in\mathfrak{C}_\Delta}{\lambda_{\min}(\mathscr{M}_{\mathrm{cvx}}(\gamma,Y,X))\geq 0}$ are closed in $\mathfrak{C}_\Delta$, and so is their intersection. Thus $\setc*{(\gamma,Y,X)\in\mathfrak{F}_{\mathrm{cvx}}}{\gamma\leq\Delta}$ is closed in $\mathfrak{C}_\Delta$, and consequently is a compact set since $\mathfrak{C}_\Delta$ is compact.

\end{document}